\documentclass{article}
\usepackage[utf8]{inputenc}

\usepackage[T1]{fontenc}
\usepackage[a4paper,top=3.5cm,bottom=3.5cm,left=3cm,right=3cm,marginparwidth=1.75cm]{geometry}

\usepackage{graphicx}
\usepackage{tikz} 
\usepackage{amsmath} 
\usepackage{amssymb}
\usepackage{amsthm}
\usepackage{fancyhdr}
\usepackage{hyperref}
\usepackage{algpseudocode}
\usepackage{algorithm2e}
\usepackage{soul}
\usepackage{placeins}

\usepackage{caption}
\usepackage{subcaption}
\usepackage{float}
\usepackage{bigints}

\newtheorem{lem}{Lemma}

\newtheorem{remark}{Remark}

\newtheorem{prop}{Proposition}
\theoremstyle{definition}

\usepackage{multirow}
\usepackage{mathtools}
\usepackage{setspace}
\usepackage{tabularx}
\usepackage{array}
\usepackage{epigraph}
\usepackage{dsfont}
\usepackage[noabbrev]{cleveref}

\DeclareMathOperator*{\argmin}{arg\,min}

\newcommand{\rev}[1]{\textcolor{black}{#1}}

\definecolor{codegreen}{rgb}{0,0.6,0}
\definecolor{codegray}{rgb}{0.5,0.5,0.5}
\definecolor{codepurple}{rgb}{0.58,0,0.82}
\definecolor{backcolour}{rgb}{0.95,0.95,0.92}

\newcommand{\xddots}{%
  \raise 4pt \hbox {.}
  \mkern 6mu
  \raise 1pt \hbox {.}
  \mkern 6mu
  \raise -2pt \hbox {.}
}

\begin{document}
\def\N{{\mathbb N}}
\newcommand{\R}{\mathbb{R}}

\newcommand{\cve}[1]{\textcolor{blue}{\texttt{[VE: #1]}}}
\definecolor{auburn}{rgb}{0.43, 0.21, 0.1}
\newcommand{\fc}[1]{\textcolor{auburn}{\texttt{[FC: #1]}}}
\newcommand{\ve}[1]{\textcolor{blue}{#1}}

\fancyfoot[L]{\color{black}{\Large{C2 - Confidential}}}

\newcommand{\bu}{\boldsymbol{u}}
\newcommand{\bv}{\boldsymbol{v}}
\newcommand{\bphi}{\boldsymbol{\phi}}
\newcommand{\bvarphi}{\boldsymbol{\varphi}}
\newcommand{\bpsi}{\boldsymbol{\psi}}
\newcommand{\bxi}{\boldsymbol{\xi}}
\newcommand{\bzeta}{\boldsymbol{\zeta}}
\newcommand{\btheta}{\boldsymbol{\theta}}
\newcommand{\balpha}{\alpha}
\newcommand{\bw}{\boldsymbol{w}}
\newcommand{\bb}{\boldsymbol{b}}
\newcommand{\bn}{\boldsymbol{n}}
\newcommand{\bId}{\boldsymbol{\rm Id}}
\newcommand{\rL}{\mathrm L}
\newcommand{\rbP}{\boldsymbol{\mathrm P}}
\newcommand{\rS}{\mathrm S}
\newcommand{\bfx}{\boldsymbol{x}}

\newcommand{\BT}[1]{\textcolor{black}{#1}}

\newpage

\pagestyle{fancy}

\lhead{ }
\chead{}
\rhead{}
\lfoot{}
\cfoot{\thepage }
\rfoot{}
\begin{center}

\begin{huge} Elasticity-based morphing technique and application to reduced-order modeling \end{huge} \\

\begin{large} A. Kabalan$^{1,2}$, F. Casenave$^2$, F. Bordeu$^2$, V. Ehrlacher$^{1,3}$, A. Ern$^{1,3}$ \end{large} \\
$^1$ Cermics, Ecole nationale des ponts et chaussées, 6-8 Av. Blaise Pascal, Champs-sur-Marne, 77455 Marne-la-Vallée cedex 2, FRANCE,\\
$^2$ Safran Tech, Digital Sciences \& Technologies, Magny-Les-Hameaux, 78114, FRANCE, \\
$^3$ Inria Paris, 48 rue Barrault, CS 61534, 75647 Paris cedex, FRANCE. \\
 
\end{center}

\hrule

\vspace*{1pt}

\setstretch{1}

\paragraph{Abstract}
The aim of this article is to introduce a new methodology for constructing morphings between shapes that have identical topology. The morphings are obtained by deforming a reference shape, through the resolution of a sequence of linear elasticity equations, onto every target shape. In particular, our approach does not assume any knowledge of a boundary parametrization, \rev{and the computation of the boundary deformation is not required beforehand}. Furthermore, constraints can be imposed on specific points, lines and surfaces in the reference domain to ensure alignment with their counterparts in the target domain after morphing. Additionally, we show how the proposed methodology can be integrated in an offline and online paradigm, which is useful in reduced-order modeling involving variable shapes. This framework facilitates the efficient computation of the morphings in various geometric configurations, thus improving the versatility and applicability of the approach. The \rev{robustness and computational efficiency of the} methodology is illustrated on \rev{two-dimensional test cases, including} the regression problem of the drag and lift coefficients of airfoils of non-parameterized variable shapes.

\hrule
\section{Introduction}
\subsection{Background}
Solving parametric partial differential equations (PDEs) for various values of parameters in a given set is a common task in industrial contexts. Examples of sets of parameters include initial and boundary values, coefficients in the PDE of interest or geometrical parameters of the domain where the PDE is posed. When the evaluation of the PDE solution is computationally expensive, model-order reduction techniques offer an efficient tool to speed up computations while maintaining accuracy.

A common situation encountered in reduced-order modeling is the following: Given a set of parameter values $\mathcal{P} \subset \mathbb{R}^p$ for some $p\in \mathbb{N}^*$, and a physical domain $\Omega_0 \subset \mathbb{R}^d$ for some $d=2,3$, one is interested in quickly computing an approximation of the solution $u_\mu: \Omega_0 \to \mathbb{R}$ for all $\mu \in \mathcal P$ of a given parametric PDE of the form $\mathcal{A}_{\mu}(u_{\mu})=0$ on $\Omega_0$ with $\mathcal A_\mu$ some parameter-dependent differential operator, together with appropriate initial/boundary conditions. Then, the reduced-basis method \cite{quarteroni2015reduced,hesthaven2016certified} involves constructing a low-dimensional approximation space $\mathcal{Z}_r={\rm span}\{ \xi_1,\ldots, \xi_r \}$ of the solution set $\mathcal{U}=\{u_{\mu} :\; \mu \in \mathcal{P} \}$, and then computing an approximation of $u_\mu\in\mathcal{Z}_r$, for instance as a Galerkin approximation of the parametric PDE, thus enabling faster solution computations. 
In practice, efficient reduced-order modeling techniques employ a two-phase procedure. First one performs the offline phase, where the PDE $\mathcal{A}_{\mu}(u_{\mu})=0$ is solved for $u_{\mu}$ for some values of the parameter $\mu \in \mathcal{M}$ using the computationally expensive high-fidelity model (HFM); here, $\mathcal{M}$ is a selected training set.  Subsequently, the reduced space $\mathcal{Z}_r$ can be constructed through approximation algorithms such as the Proper Orthogonal Decomposition (POD) \cite{berkooz1993proper,chatterjee2000introduction} or greedy approaches \cite{quarteroni2015reduced}. The online phase, also known as the exploitation phase, consists in computing approximations of the solution of the PDE belonging to $\mathcal{Z}_r$ for new parameter values. This phase leverages the precomputed reduced-order basis to efficiently compute these approximations. Depending on the complexity of the PDE and its parameter dependence, more advanced strategies such as hyper-reduction~\cite{ryckelynck2009hyper} may be required.

The present work deals with the case where the physical domain also depends on the value of the parameter $\mu \in \mathcal P$. More precisely, for all $\mu \in \mathcal P$, we now consider $\Omega_\mu \subset \mathbb{R}^d$ to be some domain which may depend on $\mu$, and assume that the solution of the parametric PDE is now a function $u_\mu: \Omega_\mu \to \mathbb{R}$. In such a situation, standard algorithms such as POD are not directly applicable, since the solutions $u_{\mu}$ are defined on different domains. The most common solution in the literature on reduced-order modeling with geometric variabilities is to find an appropriate morphing $\boldsymbol{\phi}_{\mu}$ from (or to) a reference geometry $\Omega_0$ to (or from) each parametric domain $\Omega_{\mu}$. In this scenario, the problem can be reformulated on the reference domain $\Omega_0$ and reduced-order modeling techniques are applied to the transformed solution set $\{u_{\mu}\circ \boldsymbol{\phi}_{\mu} : \; \mu \in \mathcal{P}\}$. Such a task is often called a registration problem. Registration problems are also of interest when the domain does not depend on the parameter to achieve efficiency of the reduced-order modeling technique. We refer the reader to \cite{welper2020transformed,cagniart2019model} for some seminal works in this setting.

\subsection{Related works on morphing techniques}

The difficulty now lies on the efficient construction of a morphing $ \boldsymbol{\phi}_{\rev{\mu}}: \Omega_0 \to \Omega_{\rev{\mu}}$ from a reference domain $\Omega_0 \subset \mathbb{R}^d$ to a target domain $\Omega_{\rev{\mu}}\subset \mathbb{R}^d$ that captures the target geometry accurately. Early works on model-order reduction with geometrical variability adopted the use of affine mappings \cite{rozza2008reduced}. However, this approach cannot be applied to general domains with curved boundaries and edges. Other commonly used techniques in computational physics to deform a geometry (or a mesh) onto another are free-form deformation (FFD) \cite{sederberg1986free}, radial basis function (RBF) interpolation \cite{de2007mesh}, linear elasticity/harmonic mesh morphings \cite{baker2002mesh,masud2007adaptive}, nonlinear elasticity \cite{shontz2012robust,froehle2015nonlinear}, only to cite a few. Numerous contributions adopted these strategies in reduced-order modeling contexts \cite{manzoni2017efficient,salmoiraghi2018free,demo2021hull,lehrenfeld2019mass}. However, all these strategies share the assumption that the geometries are parameterized \rev{or} that the deformation of the nodes on the boundary is known. This way, the displacement of the nodes on $\partial\Omega_0$ can be imposed 
to map onto $\partial\Omega_{\rev{\mu}}$, and the extension of the deformation to the whole domain can be determined by means of the chosen method.

In many scenarios, however, an explicit parametrization of the geometry is not available, especially in the online phase. In this situation, constructing a suitable morphing $\boldsymbol{\phi}_{\mu}: \Omega_0 \to \Omega_{\mu}$ becomes more challenging. One possibility often advocated in the literature is a two-step procedure which consists of first finding the deformation of the boundary $\boldsymbol{\phi}_{\mu}(\partial \Omega_0)$, and then leveraging the knowledge of $\boldsymbol{\phi}_{\mu}(\partial \Omega_0)$ to compute $\boldsymbol{\phi}_{\mu}(\Omega_0)$, by using either RBF interpolation \cite{casenave2024mmgp,ye2024data,sieger2014rbf}, mesh parametrization \cite{casenave2024mmgp}, geometry registration \cite{taddei2020registration,taddei2022optimization,taddei2023compositional}, optimal transport~\cite{iollo2022mapping,cucchiara2024model}, \rev{iterative spring analogy \cite{GT2014}}, or some other technique. These approaches require the computation of the boundary morphing before calculating the volume morphing. \rev{However, these approaches suffer from two main drawbacks. First, the boundary morphing is usually case specific, and, to the best of our knowledge, there is no generic way to deform $\boldsymbol{\phi}_{\mu}(\partial \Omega_0)$ onto $\partial \Omega_{\rev{\mu}}$. Second, depending on the case, the above two-step procedure could be expensive, which would make these approaches not suitable for model-order reduction.} Another class of methods \rev{which does not require the a priori knowledge of the boundary morphing} is the LDDMM (Large Deformation Diffeomorphic Metric Mapping) \cite{beg2005computing}. This method finds the morphing from $\Omega_0$ to $\Omega_{\mu}$ as a flow of diffeomorphims solving optimal control problems, but is usually expensive to compute \cite{galarce2022state}.

\subsection{Contribution\rev{s}}
The \rev{first main} contribution of the paper \rev{is a novel method for constructing a morphism} from a reference domain $\Omega_0$ to a target domain $\Omega_{\mu}$ without a priori knowledge on \rev{any parametrization of the geometry of the target domain or its boundary. Moreover, the method proceeds in a single, generic step. Additionally, it is possible to impose certain geometrical features, such as points, lines or surfaces on $\partial\Omega_0$, to be mapped onto some a priori chosen counterparts on $\partial\Omega_{\mu}$. Moreover, the constructed morphing allows for a tangential displacement of the boundary, especially near the target shape, thus reducing distortions. While some of the morphing techniques described in the previous section share some of these features, none of them shares all the features.}

The construction of the morphing proceeds as follows.
Starting from $\Omega_0$, the algorithm produces a sequence of morphisms $(\bphi^{(m)})_{m\geq 0}$ defined on $\Omega_0$ such that $\bphi^{(0)} = \bId|_{\Omega_0}$, where $\bId$ denotes the identity mapping from $\mathbb{R}^d$ onto $\mathbb{R}^d$. For all $m \in \N$, denoting by $\Omega^{(m)} = \bphi^{(m)}(\Omega_0)$, the morphing is updated at iteration $m+1$ as 
\begin{equation}\label{eq:update}
\bphi^{(m+1)} =  \left(\bId|_{\Omega_0} + \gamma^{(m)} \bu^{(m)}\right)\circ \bphi^{(m)},
\end{equation}
where $\bu^{(m)}: \Omega^{(m)} \to \mathbb{R}^d$ is the solution of a linear elasticity problem posed on $\Omega^{(m)}$ and $\gamma^{(m)}>0$ is a user-dependent parameter expected to be small. Notice that (\ref{eq:update}) may be seen as a time-discretization scheme associated with the evolution equation
$$
\partial_t \bphi(t) = \bu(t) \circ \bphi(t),
$$
where $\bu(t): \mathbb{R}^d \to \mathbb{R}^d$ is a time-dependent velocity field. In the linear elasticity problem at iteration m, external forces are applied on the boundary of the current domain $\Omega^{(m)}$ to ensure that the new domain $\Omega^{(m+1)}$ is closer in a certain sense to the target domain $\Omega_{\rev{\mu}}$. The present approach shares similarities with \cite{de2016optimization} but differs in the type of linear elasticity problems that are solved. 

The second main contribution is to embed the above morphing technique in a reduced-order modeling context. Given a collection of domains $\{\Omega_i\}_{1\leq i \leq n}$ for some $n\in \mathbb{N}^*$ which forms the training set, we compute  morphisms $\bphi_i: \Omega_0 \to \Omega_i$ in an offline phase using the algorithm proposed above. Then, we design an efficient online reduced-order model to quickly compute a morphing from the reference domain $\Omega_0$ onto a new target domain outside the training set. The efficiency of the approach is strongly linked to the use of an appropriate initial guess used as a starting point in the iterative online procedure. Finally, we provide numerical evidence that the method produces accurate results when employed in regression-based model-order reduction techniques. 

To sum up, the main contributions of this work are as follows:
\begin{enumerate}
\item A novel morphing technique applicable to non-parametric domains with two main highlights:
\begin{enumerate}
\item it is possible to prescribe a priori the displacement of points and lines on the boundary;
\item the morphing allows for a tangential displacement of the boundary, especially near the target shape, thus reducing distortions.
\end{enumerate}
\item The embedding of the above morphing technique into a multi-query context with variable shapes:
\begin{enumerate}
\item allowing for a computationally efficient algorithm based on a reduced-order modeling with an offline/online decomposition;
\item offering the possibility of learning scalar outputs from simulations realized with variable shapes.
\end{enumerate}
\end{enumerate}
While the present work focuses on the application of morphing to model-order reduction, morphing is an important ingredient in many other areas of application, such as shape optimization \cite{porziani2021automatic}, fluid-structure interaction \cite{shamanskiy2021mesh,wick2011fluid}, model generation \cite{pascoletti2021stochastic}, and healthcare \cite{PCBCZ:19}, only to cite a few examples.

\subsection{Motivating example}

We present in this section an example which motivates the interest of the present methodology in the context of reduced-order modeling. 

Let $d = \rev{2}$ and let $\{\Omega_i\}_{1\leq i \leq n}\subset \R^d$ be a collection of domains in $\R^d$, where a domain in $\R^d$ is understood as an open bounded connected subset of $\R^d$ with piecewise smooth boundary. Assume that all the domains share the same topology. Let $\Omega_0 \subset \R^d$ be a fixed reference domain that shares the same topology as well. The collection $\{\Omega_i\}_{1\leq i \leq n}$ is referred to as the training set of target domains. 
 \begin{figure}[ht] 
\begin{center}
\includegraphics[scale=0.4]{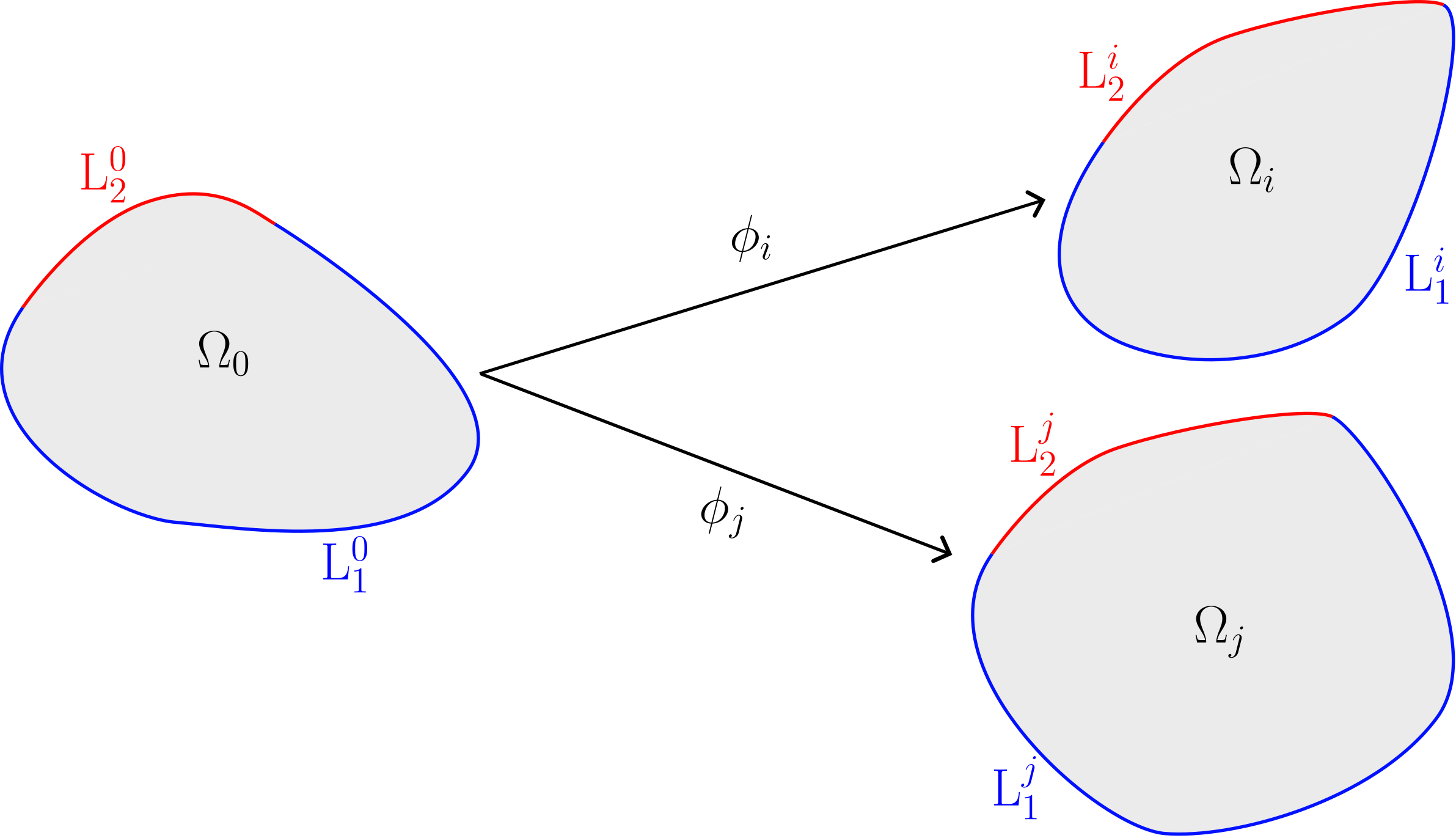}
\end{center}
\caption{Reference domain $\Omega_0$ with two samples $\Omega_i$ and $\Omega_j$ from the target dataset.}
\end{figure}
Assume now that one is actually interested in solving, for all $i\in \{1,\ldots,n\}$, the following parametric (elliptic) PDE with mixed boundary conditions: 
$$
\left\{
\begin{array}{ll}
\mathcal{L}_{\mu_i}(u_{\mu_i })=0 & \text{ in }  \Omega_i, \\
u_{\mu_i}= a & \mbox{ on } L_1^i,\\
\nabla u_{\mu_i} \cdot \bn_i= b & \mbox{ on }{L_2^i},\\
\end{array}
\right.
$$
where $\bn_i$ is the unit normal outward vector to $\Omega_i$, $\mu_i \in \mathcal M$ belongs to the set of parameter values, $u_{\mu_i }: \Omega_i \to \R$ is the solution to the PDE problem of interest, $a,b \in \R$, and ${L_1^i}$ and ${L_2^i}$ are open subsets of $\partial \Omega_i$ which form a partition of $\partial \Omega_i$. In other words, Dirichlet boundary conditions are enforced on $L_1^i$, whereas Neumann boundary conditions are enforced on $L_2^i$. Moreover, one wishes to construct a reduced-order modeling technique to quickly obtain numerical approximations of the problems above. 

Since, for all $1\leq i \leq n$, each solution $u_{\mu_i }$ is defined on a different domain, traditional dimensionality reduction methods such as POD are not directly applicable. One possibility is to rely on so-called registration methods to find a morphing $\bphi_i: \Omega_0 \to \Omega_i$, and then apply POD on the family of functions $\{u_{\mu_i }\circ \bphi_i\}_{1\leq i \leq n}$. Moreover, we want to ensure that $\bphi_i(L_1^0)=L_1^i$ and $\boldsymbol{\phi}_i(L_2^0)=L_2^i$, with $\partial \Omega_0= L_1^0 \bigcup L_2^0$. In this case, the boundary conditions $u_{\mu_i }\circ \boldsymbol{\phi}_i|_{L_1^0}= a$ and  $(\nabla u_{\mu_i} \cdot \bn_i)\circ \boldsymbol{\phi}_i|_{L_2^0}= b$  are satisfied, and the dimensionality reduction problem is expected to be simpler.

\subsection{Outline of the paper}
In Section~\ref{Section Offline}, we present the (offline) methodology to construct a morphing from a reference domain $\Omega_0$ onto a target domain $\Omega$ while respecting certain conditions on the morphing of the boundary. In Section~\ref{Section Online}, we show how, given a training dataset of geometries $\{\Omega_i\}_{1\leq i \leq n}$, we can reduce the complexity of the problem of finding a morphing for a given domain outside the training dataset, so that the method can be efficient in the online phase. In both sections, we provide numerical examples to illustrate the behavior of the proposed methods. In Section~\ref{learning scalars}, we present an application of the proposed morphing strategy to learn scalar outputs from simulations realized on different (non-parameterized) geometries. As an example, we predict the drag coefficient of airfoils of non-parameterized variable shapes. Finally, in Section~\ref{Conclusion}, we provide a brief summary and some concluding remarks. 

\section{ High-fidelity morphing construction} \label{Section Offline}
In this section, we present the new high-fidelity methodology to construct a morphism $\bphi: \Omega_0 \to \Omega$ between a reference domain $\Omega_0 \subset \mathbb{R}^d$ and a target domain $\Omega \subset \R^d$. In the context of model-order reduction with geometric variability, this approach is applied in the offline phase (see Section \ref{Section Online}). In what follows, we denote by $\|\cdot\|$ the Euclidean norm of $\mathbb{R}^d$. Moreover, we use boldface notation for vectors in $\R^d$, fields taking values in $\R^{d}$, and sets and linear spaces composed of such fields.

\subsection{Notation and preliminaries}
Let $\Omega_0$ and $\Omega$ be domains in $\mathbb{R}^d$. For the sake of simplicity, we present the methodology in the case $d=2$. 

Let $N_p, N_l\in \N^*$. Let $\left\{\rbP_1, \ldots , \rbP_{N_p}\right\} \subset \partial \Omega$ be a collection of $N_p$ distinct points of $\partial \Omega$, and $\left\{L_1, \ldots , L_{N_l}\right\} \subset \partial \Omega$ be a collection of disjoint, open, connected subdomains of $\partial \Omega$ with positive $1$-dimensional Hausdorff measure such that $\bigcup_{k = 1}^{N_l} \overline{L_{k}} = \partial \Omega$, where $\overline{L_{k}}$ denotes the closure of $L_{k}$. Similarly, we consider a collection $\left\{\rbP_1^0, \cdots , \rbP^0_{N_p}\right\} \subset \partial \Omega_0 $  of $N_p$ distinct points of $\partial \Omega_0 $ and a collection $\left\{L_1^0, \cdots , L^0_{N_l} \right\} \subset \partial \Omega_0 $ of disjoint, open, connected subdomains of $\partial \Omega_0 $ with positive $1$-dimensional Hausdorff measure such that $\bigcup_{k = 1}^{N_l} \overline{L^0_{k}} = \partial \Omega_0$.
 Our goal is to build a morphing such that each line $L^0_k$ (resp., point $\rbP^0_k$) in $\partial\Omega_0$ is mapped to the corresponding line $L_k$ (resp., point $\rbP_k$) in $\partial\Omega$. 
 
Let us introduce the set $ \boldsymbol{\mathcal{T}}_{\Omega_0 }:=\{ \boldsymbol{\phi} \in \boldsymbol{W}^{1,\infty}(\Omega_0):\;  \boldsymbol{\phi} \text{ is injective}, \: \boldsymbol{\phi}^{-1}\in \boldsymbol{W}^{1,\infty}(\boldsymbol{\phi}(\Omega_0))  \}$.  We wish to find a morphing $\boldsymbol{\phi} \in \boldsymbol{\mathcal{T}}_{\Omega_0}$ such that
\begin{subequations} 
\label{mapping conditions}
\begin{alignat}{2}
    &\boldsymbol{\phi}(\Omega_0)=\Omega ,   \label{domain mapping}\\
    & \boldsymbol{\phi}(\rbP_{k}^0)=\rbP_{k} , \quad && \forall 1\leq k \leq N_p ,  \label{points mapping}\\
    & \boldsymbol{\phi}(L^0_{k})=L_{k} , \quad && \forall 1\leq k \leq N_l.  \label{curve mapping}
    \end{alignat}
\end{subequations}
Our aim here is to propose a new iterative method to construct a morphing $\boldsymbol{\phi}\in \boldsymbol{\mathcal{T}}_{\Omega_0 }$ such that conditions \eqref{domain mapping}-\eqref{points mapping}-\eqref{curve mapping} are satisfied at convergence. The rest of the section is organized as follows. First, in Section~\ref{sec:math_setting}, we collect some auxiliary mathematical results to justify the relevance of the proposed approach. Then, in Section~\ref{shape matching opti}, we propose a first approach, inspired from~\cite{de2016optimization}, to construct a morphing satisfying only the requirement~\eqref{domain mapping}. Finally, in Section~\ref{Shape matching with tags}, we \rev{present the main approach to construct the morphing so that all the constraints in~\eqref{mapping conditions} are taken into consideration.}

\subsection{Mathematical setting} \label{sec:math_setting}

This section collects some auxiliary mathematical results, most of which are classical. For the sake of completeness, we recall some proofs. We start with a classical lemma (see~\cite[Lemma~6.13]{allaire2007conception}).
\begin{lem} \label{variation of phi}
    Let $\boldsymbol{\phi} \in \boldsymbol{\mathcal{T}}_{\Omega_0}$. Define the set $\boldsymbol{\mathcal{T}}'_{\boldsymbol{\phi}, 1}:=\{\bv \circ \boldsymbol{\phi} \; | \;  \bv \in \boldsymbol{W}^{1,\infty}(\bphi(\Omega_0)),\; \|\bv\|_{\boldsymbol{W}^{1,\infty}(\bphi(\Omega_0))} < 1 \}\subset \boldsymbol{W}^{1,\infty}(\Omega_0)$. Then, for all $\boldsymbol{\xi} \in \boldsymbol{\mathcal{T}}'_{\boldsymbol{\phi}, 1}$, we have $\boldsymbol{\phi} + \boldsymbol{\xi} \in \boldsymbol{\mathcal{T}}_{\Omega_0}$.
\end{lem}
\begin{proof}
Let $\boldsymbol{\phi} \in \boldsymbol{\mathcal{T}}_{\Omega_0}$, $\boldsymbol{\xi} \in \boldsymbol{\mathcal{T}}'_{\boldsymbol{\phi},1}$ such that $ \boldsymbol{\xi} = \bv \circ \boldsymbol{\phi}$ for some $\bv \in \boldsymbol{W}^{1, \infty}(\bphi(\Omega_0))$ with $\|\bv\|_{\boldsymbol{W}^{1,\infty}(\bphi(\Omega_0))} < 1$. Then, $\boldsymbol{\phi} + \boldsymbol{\xi} = \boldsymbol{\phi} + \boldsymbol{v} \circ \boldsymbol{\phi} = (\bId+\boldsymbol{v})\circ \boldsymbol{\phi}$. Since $\|\bv\|_{\boldsymbol{W}^{1,\infty}(\bphi(\Omega_0))} < 1$, $(\bId+\bv)$ is bijective from $\Omega_0$ onto $\boldsymbol{\phi}(\Omega_0)$, so that $\boldsymbol{\phi} + \boldsymbol{\xi} =  (\bId+\bv)\circ \boldsymbol{\phi}$ is injective, as the composition of two injective maps. Hence, $\boldsymbol{\phi} + \boldsymbol{\xi} \in \boldsymbol{\mathcal{T}}_{\Omega_0}$.
\end{proof}
This lemma proves in particular that the set $\boldsymbol{\mathcal{T}}'_{\boldsymbol{\phi}}:=\{\bv \circ \boldsymbol{\phi} | \;  \bv \in \boldsymbol{W}^{1,\infty}(\bphi(\Omega_0))\}$ is included in the tangential space of $\boldsymbol{\mathcal T}_{\Omega_0}$ at point $\bphi$.

The following proposition is classical in topology optimization, see, e.g., \cite{allaire2007conception,allaire2004structural} for a proof.
\begin{prop} \label{prop Diff J}
Let $g \in H^1_{{\rm loc}}(\R^d)$, let $\boldsymbol{\phi} \in \boldsymbol{\mathcal{T}}_{0}$, $\boldsymbol{\psi} \in  \boldsymbol{\mathcal{T}}'_{\boldsymbol{\phi}}$, and let $\bn_{\boldsymbol{\phi}}$ be the outward unit normal vector to $\boldsymbol{\phi}(\partial \Omega_0)$. Then the differential of $J_g$ at $\boldsymbol{\phi}$ evaluated in the direction $\boldsymbol{\psi}$ is
\begin{align} \label{diff J}
    DJ_g(\boldsymbol{\phi})(\boldsymbol{\psi})= \int_{\boldsymbol{\phi}(\partial\Omega_0)} g(\boldsymbol{x}) \boldsymbol{\psi} \circ \boldsymbol{\phi}^{-1}(\boldsymbol{x})\cdot  \bn_{\boldsymbol{\phi}} ds.
\end{align}
\end{prop}
We define, for all $\bv\in \boldsymbol{W}^{1,\infty}(\bphi(\Omega_0))$, 
$$
 \widetilde{DJ_g}(\boldsymbol{\phi})(\bv)= \int_{ \boldsymbol{\phi}(\partial\Omega_0)} g(\boldsymbol{x}) \bv(\boldsymbol{x})\cdot  \bn_{\boldsymbol{\phi}} ds,  
$$
so that 
$$
 DJ_g(\boldsymbol{\phi})(\bv \circ \bphi) =  \widetilde{DJ_g}(\boldsymbol{\phi})(\bv). 
$$
The proof of the following lemma is immediate using the trace theorem and the fact that $g\in H^1_{\rm loc}(\mathbb{R}^d)$. 

\begin{lem} 
    Let $\boldsymbol{\phi} \in \boldsymbol{\mathcal{T}}_{\Omega_0}$. The linear functional $\widetilde{DJ_g}(\boldsymbol{\phi}): \boldsymbol{W}^{1,\infty}(\bphi(\Omega_0)) \to \R $ can be uniquely extended to a continuous linear form defined on the space $\boldsymbol{H}^1(\bphi(\Omega_0))$.
\end{lem}

The following proposition is at the heart of the new method we propose in the present work. 
\begin{prop} \label{Prop inner product} Assume that $d=2$. Let  $\alpha >0$ and let $\boldsymbol{\phi} \in \boldsymbol{\mathcal{T}}_{\Omega_0}$ be such that
\begin{itemize}
\item $\boldsymbol{\phi}(\Omega_0)$ is a domain in $\mathbb{R}^2$; 
\item for all $\boldsymbol{p}\in \mathbb{R}^2$ and all $r>0$, $\boldsymbol{\phi}(\Omega_0) \neq B(\boldsymbol{p},r)$, where $B(\boldsymbol{p},r)$ is the open ball of $\mathbb{R}^2$ with center $\boldsymbol{p}$ and radius $r$.
\end{itemize}
 For all $\bu \in \boldsymbol{H}^1(\bphi(\Omega_0))$, let $\varepsilon(\bu)$ and $\sigma(\bu)$ be the linearized strain and stress tensors defined by
\begin{align*}
    &\varepsilon(\bu) := \frac{1}{2} (\nabla \bu + \nabla \bu^T),\\
    &\sigma(\bu) :=  \frac{E}{(1+\nu)}\varepsilon(\bu) + \frac{E\nu}{(1+\nu)(1-\nu)} {\rm Tr}(\varepsilon(\bu))I,
\end{align*}
with $E>0$ and $-1<\nu<\frac{1}{2}$ are, respectively, called the Young modulus and the Poisson ratio. Then the bilinear form
\begin{align} \label{inner product}
    a_{\boldsymbol{\phi}}: \boldsymbol{H}^1({\bphi(\Omega_0)}) \times \boldsymbol{H}^1({\bphi(\Omega_0)}) \ni (\bu,\bv) \longmapsto a_{\boldsymbol{\phi}}(\bu,\bv) := \int_{ \boldsymbol{\phi}(\Omega_0)} \sigma ( \bu):\varepsilon (\bv) d\boldsymbol{x} 
     + \alpha \int_{ \boldsymbol{\phi}(\partial\Omega_0)} (\bu \cdot \bn_{\bphi}) (\bv\cdot \bn_{\bphi}) ds 
\end{align}
defines an inner product on $\boldsymbol{H}^1({\bphi(\Omega_0)})$.
\end{prop}

\begin{proof}
    We only need to show the positive definiteness of $a_{\bphi}$. Let $\bu  \in \boldsymbol{H}^1({\bphi(\Omega_0)}) $ be such that $a_{\bphi}(\bu, \bu) =0$. Let us prove that $\bu = \boldsymbol{0}$. 
    From the definition \eqref{inner product} of $a_{\bphi}$, we infer that $\varepsilon(\bu) = 0$ in $\bphi(\Omega_0)$ and $\bu\cdot \bn_{\bphi} = 0$ on $ \bphi(\partial\Omega_0)$. 
    Thus, since $\boldsymbol{\phi}(\Omega_0)$ is connected, there exists $\boldsymbol{M}\in \mathbb{R}^{2\times 2}$ with $\boldsymbol{M}^T  = - \boldsymbol{M}$ and $\bb\in \mathbb{R}^2$ such that $\bu(\boldsymbol{x}) = \boldsymbol{M}\boldsymbol{x} + \bb$, for all $\boldsymbol{x}\in \boldsymbol{\phi}(\Omega_0)$.
    
    Let us now prove that necessarily $M = \boldsymbol{0}$ and $\bb = \boldsymbol{0}$. Reasoning by contradiction, let us first assume that $\boldsymbol{M}\neq \boldsymbol{0}$. Then, there exists $m\in \mathbb{R}\setminus\{0\}$ and $\boldsymbol{y}= (y_1, y_2)\in \mathbb{R}^2$ such that for all $\boldsymbol{x} = (x_1,x_2)\in \boldsymbol{\phi}(\Omega_0)$, 
$$
\bu(\boldsymbol{x}) = m \left(\begin{array}{c}
x_2 - y_2\\
-(x_1 - y_1)\\
\end{array}\right).
$$
Since $\bu\cdot \bn_{\bphi}=0$ on $ \bphi(\partial\Omega_0)$, we infer that $\bn_{\bphi}(\boldsymbol{x}) = \pm \frac{ (\boldsymbol{x}-\boldsymbol{y})}{\|\boldsymbol{x}-\boldsymbol{y}\|}$ for all $\boldsymbol{x}\in \boldsymbol{\phi}(\partial\Omega_0)$ such that $\boldsymbol{x}\neq \boldsymbol{y}$. Since the boundary of $\boldsymbol{\phi}(\Omega_0)$ is piecewise $\mathcal C^1$, the following holds:
    \begin{itemize}
        \item Either $\bn_{\bphi}(\boldsymbol{x}) = \frac{\boldsymbol{x}-\boldsymbol{y}}{\|\boldsymbol{x}-\boldsymbol{y}\|}$ for all $\boldsymbol{x}\in \boldsymbol{\phi}(\partial\Omega_0)$ and hence $\boldsymbol{\phi}(\Omega_0)$ has to be equal to $B(\boldsymbol{y},r)$ for some $r>0$, which is not possible by assumption; 
        \item Or  $\bn_{\bphi}(\boldsymbol{x}) = \frac{-(\boldsymbol{x}-\boldsymbol{y})}{\|\boldsymbol{x}-\boldsymbol{y}\|}$ for all $\boldsymbol{x}\in \boldsymbol{\phi}(\partial\Omega_0)$ has to be equal to $\overline{B(\boldsymbol{y},r)}^c$ for some $r>0$, which cannot be since $\boldsymbol{\phi}(\Omega_0)$ is a bounded set.
    \end{itemize}
    Hence, $\boldsymbol{M} = \boldsymbol{0}$ and $\bu(\boldsymbol{x}) = \bb$ for all $\boldsymbol{x}\in \boldsymbol{\phi}(\Omega_0)$. Thus, we obtain  $\bb\cdot \bn_{\bphi}=0$ on $ \boldsymbol{\phi}(\partial\Omega_0)$ which is not possible if $\bb\neq \boldsymbol{0}$ since $ \boldsymbol{\phi}(\partial\Omega_0)$ would then be a hyperplane orthogonal to $\bb$. 
    Hence, $\bb=\boldsymbol{0}$, and this completes the proof.
\end{proof}

\subsection{Shape matching without constraints} \label{shape matching opti}

The aim of this section is to propose a new approach to compute a morphism $\bphi \in \boldsymbol{\mathcal{T}}_{\Omega_0 } $ satisfying (\ref{domain mapping}). Our starting point is the approach introduced in \cite{de2016optimization}. As in~\cite{de2016optimization}, our approach consists in reformulating the problem as an optimization problem and the algorithm as a gradient descent.  Let $g \in H^1_{{\rm loc}}(\R^d)$ to be a level set function for $\Omega$, i.e., a function such that, for all $\bfx\in \mathbb{R}^d$,
\begin{align}\label{level set def}
    \left\{
    \begin{array}{ll}
         g(\boldsymbol{x})<0 & \text{if \hspace{1mm}} \boldsymbol{x} \in \Omega, \\
        g(\boldsymbol{x})>0 & \text{if \hspace{1mm}} \boldsymbol{x} \in \overline{\Omega}^c,\\
        0 &  \text{if \hspace{1mm}} \boldsymbol{x} \in \partial \Omega.
    \end{array} \right .
\end{align}
Define the functional
\begin{align}
            J_g: \boldsymbol{\mathcal{T}}_{\Omega_0} \ni \boldsymbol{\phi} \longmapsto J_g(\boldsymbol{\phi}) :=  \int_{\boldsymbol{\phi}(\Omega_0)} g(\boldsymbol{x}) d\boldsymbol{x} \in \R . \label{J}
\end{align}
Since $g$ is a level-set function, the set $ \boldsymbol{\mathcal{T}}^*:= \{\boldsymbol{\phi} \in \boldsymbol{\mathcal{T}}_{\Omega_0}\;  | \; \boldsymbol{\phi}(\Omega_0)=\Omega \} $ coincides with the set of global minimizers of $J_g$ over $\boldsymbol{\mathcal T}_{\Omega_0}$. Thus, in order to find a morphing from $\Omega_0$ to $\Omega$, we can consider the following optimization problem:
\begin{align}\label{min J}
    \text{Find } \boldsymbol{\phi}^* \in \argmin_{\boldsymbol{\phi} \in \boldsymbol{\mathcal{T}}_{\Omega_0}} J_g(\boldsymbol{\phi}).
\end{align} 
One example of level-set function, which is commonly used in practice, is the \textit{signed distance function} defined as
\begin{align}\label{Euclidean SDF}
    d_{\Omega}(\boldsymbol{x})=\left\{
    \begin{array}{ll}
        - d(\boldsymbol{x},\partial \Omega) & \text{if \hspace{2mm}} \boldsymbol{x} \in \Omega, \\
        d(\boldsymbol{x},\partial \Omega) & \text{if \hspace{2mm}} \boldsymbol{x} \in \overline{\Omega}^c,\\
        0 &  \text{if \hspace{2mm}} \boldsymbol{x} \in \partial \Omega,
    \end{array}\right .
\end{align}
where $d(\boldsymbol{x},\partial \Omega)$ is the Euclidean distance of $\boldsymbol{x}$ to $\partial \Omega$.

The high-fidelity algorithm we propose to compute a morphism $\bphi\in \boldsymbol{\mathcal{T}}_{\Omega_0}$ satisfying \eqref{domain mapping} is a particular gradient descent algorithm to minimize the functional $J_g$ for some level set function $g\in H^1_{\rm loc}(\mathbb{R}^d)$ over $\boldsymbol{\mathcal{T}}_{\Omega_0}$, all the iterations being guaranteed to be well-defined. More precisely, the algorithm is an iterative algorithm which computes a sequence of morphisms $(\bphi^{(m)})_{m\geq 0}$ as follows. The starting point is $\bphi^{(0)} = \bId$. At iteration $m \in \N$, knowing $\bphi^{(m)}$, the next iterate $\bphi^{(m+1)}$ is computed as
\begin{equation}\label{eq:iterate}
\bphi^{(m+1)} = (\bId + \gamma^{(m)}\bu^{(m)}) \circ \bphi^{(m)},
\end{equation}
where $\gamma^{(m)}$ is some positive (small) constant and $\bu^{(m)}$ is computed as follows: Given some finite-dimensional subspace $\boldsymbol{V}^{(m)} \subset \boldsymbol{W}^{1,\infty}\left( \Omega^{(m)}\right)$ where $\Omega^{(m)}:= \bphi^{(m)}(\Omega_0)$,
 $\bu^{(m)}\in \boldsymbol{V}^{(m)}$ is defined as the unique solution to 
\begin{equation}\label{eq:var1}
\forall \bv^{(m)}\in \boldsymbol{V}^{(m)}, \quad a_{\bphi^{(m)}}(\bu^{(m)}, \bv^{(m)}) = - \widetilde{DJ_g}(\bphi^{(m)})(\bv^{(m)}).
\end{equation}
In other words, $\bu^{(m)}\in \boldsymbol{V}^{(m)}$ is the unique solution to the following linear elasticity problem:
\begin{equation}\label{eq:var2}
\forall \bv^{(m)}\in \boldsymbol{V}^{(m)}, \quad \int_{ \Omega^{(m)}} \sigma ( \bu^{(m)}):\varepsilon (\bv^{(m)}) d\boldsymbol{x} 
     + \alpha \int_{\partial\Omega^{(m)}} (\bu^{(m)} \cdot \bn^{(m)}) (\bv^{(m)}\cdot \bn^{(m)}) ds  = - \int_{\partial\Omega^{(m)}} g(\boldsymbol{x}) \bv^{(m)}(\boldsymbol{x})\cdot  \bn^{(m)}ds,
\end{equation}
where $\bn^{(m)}$ denotes the outward unit normal vector to $\Omega^{(m)}$. In particular, when the level set function $g$ is chosen to be the distance function $d_\Omega$, problem (\ref{eq:var2}) reads as follows:
\begin{equation}\label{eq:var3}
\forall \bv^{(m)}\in \boldsymbol{V}^{(m)}, \quad \int_{ \Omega^{(m)}} \sigma ( \bu^{(m)}):\varepsilon (\bv^{(m)}) d\boldsymbol{x} 
     + \alpha \int_{\partial\Omega^{(m)}} (\bu^{(m)} \cdot \bn^{(m)}) (\bv^{(m)}\cdot \bn^{(m)}) ds  = - \int_{\partial\Omega^{(m)}} d_\Omega(\boldsymbol{x}) \bv^{(m)}(\boldsymbol{x})\cdot  \bn^{(m)}ds.
\end{equation}
In what follows, we refer to this procedure as the \textbf{signed distance algorithm}.  An illustration is shown in Figure \ref{3domains}. 

\begin{figure}[ht] 
\begin{center}
\includegraphics[scale=0.4]{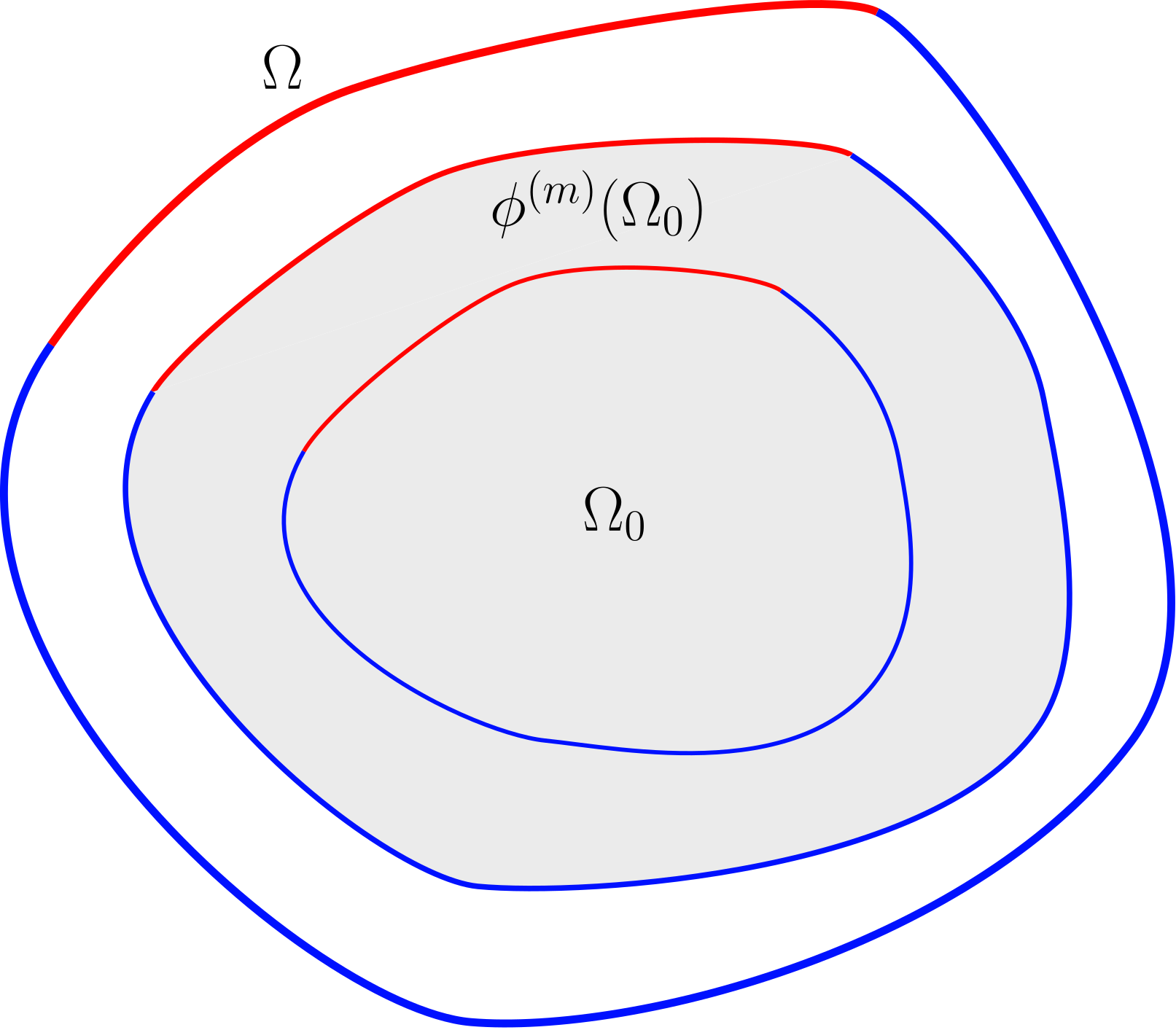}
\caption{ \label{3domains} Example of reference domain $\Omega_0$, target domain $\Omega$, and intermediate domain $\boldsymbol{\phi}^{(m)}(\Omega_0$).  }
\end{center}
\end{figure}
 
\begin{remark}[Comparison with \cite{de2016optimization}]
Let us comment on the differences between the approach we propose and the one in~\cite{de2016optimization}. In~\cite{de2016optimization}, the authors consider a similar iterative algorithm with the difference that the gradient direction $\bu^{(m)}$ is obtained as the solution of a problem of the form
\begin{equation}\label{eq:var5}
\forall \bv^{(m)}\in \widetilde{\boldsymbol{V}}^{(m)}, \quad \int_{ \Omega^{(m)}} \sigma ( \bu^{(m)}):\varepsilon (\bv^{(m)}) d\boldsymbol{x}  = - \int_{\partial\Omega^{(m)}} d_\Omega(\boldsymbol{x}) \bv^{(m)}(\boldsymbol{x})\cdot  \bn^{(m)}ds,
\end{equation}
where $\widetilde{\boldsymbol{V}}^{(m)}$ is a finite-dimensional subspace of $\boldsymbol{H}^1_{0,\omega^{(m)}}(\Omega^{(m)}):= \left\{ \bu \in \boldsymbol{H}^1(\Omega^{(m)}): \; \bu = \boldsymbol{0} \mbox{ on }\omega^{(m)}\right\}$ and $\omega^{(m)} \subset \Omega^{(m)}$ is an open subdomain of $\Omega^{(m)}$ with positive measure. Our motivation for considering problems of the form (\ref{eq:var3}) instead of problems of the form (\ref{eq:var5}) is twofold:
\begin{itemize}
\item On the one hand, in the present approach, there is no need for the choice of a subdomain $\omega^{(m)}$ of $\Omega^{(m)}$, which in particular avoids to have null displacements into some arbitrarily chosen region of the domain $\Omega^{(m)}$; 
\item on the other hand, the second term on the left-hand side of  (\ref{eq:var3}) may be seen as a Tikohonov regularization term to select a solution $\bu^{(m)}$ such that its normal component is as small as possible on the boundary of $\Omega^{(m)}$. This allows, in particular close to convergence, to allow tangential displacements along the boundary of the domain. This feature turns out to be particularly useful in our numerical tests.
\end{itemize}
\end{remark}

\begin{remark}[Parameter $\gamma^{(m)}$]
For simplicity, in our numerical tests, the sequence of parameters $(\gamma^{(m)})_{m\geq 0}$ is chosen to be equal to some constant value $\gamma$. Let us highlight, however, that, if for all $m\in \N$, we choose $\gamma^{(m)} \leq \displaystyle \min \left(\gamma_0, \frac{1}{2\| \bu^{(m)}\|_{\boldsymbol{W}^{1,\infty}(\Omega^{(m)})}} \right)$ for some $\gamma_0>0$, then it is guaranteed by Lemma~$\ref{variation of phi}$ that $\boldsymbol{\phi}^{(m)}\in \boldsymbol{\mathcal{T}}_{\Omega_0}$ for all $ m\in \N$.
\rev{Thus, one approach to adapt the value of $\gamma^{(m)}$ during the computation is to choose small values at the beginning of the iterations (when $\| \bu^{(m)}\|_{\boldsymbol{W}^{1,\infty}(\Omega^{(m)})}$ is large) and larger values as $\phi^{(m)}(\Omega_0)$ approaches $\Omega$ (and $\| \bu^{(m)}\|_{\boldsymbol{W}^{1,\infty}(\Omega^{(m)})}$ gets smaller). It is expected that taking a too small value of $\gamma^{(m)}$ hinders algorithmic performance, whereas a too large value can lead to element distortion and eventually failure.}

\end{remark}
\begin{remark}[Gradient descent]
The present procedure is a gradient descent algorithm for the resolution of the optimization problem (\ref{min J}). Indeed, $\bpsi^{(m)}:= \bu^{(m)} \circ \bphi^{(m)}$ is the unique solution of 
    $$
    \forall \bxi^{(m)} \in \boldsymbol{V}_0^{(m)}, \quad c_{\boldsymbol{\phi}^{(m)}}(\bpsi^{(m)},\bxi^{(m)}) = DJ_g(\bphi^{(m)})(\bxi^{(m)}),
    $$
    where for all $\bphi\in \boldsymbol{\mathcal T}_{\Omega_0}$,
    $$
     c_{\boldsymbol{\bphi}}: \boldsymbol{H}^1(\Omega_0) \times \boldsymbol{H}^1(\Omega_0) \ni (\bpsi,\xi) \longmapsto c_{\boldsymbol{\bphi}}(\bpsi,\bxi) := a_{\bphi}(\bpsi \circ \bphi^{-1}, \bxi\circ \bphi^{-1}),
    $$
    and
    $$
    \boldsymbol{V}_0^{(m)}:= \left\{ \bv \circ \bphi^{(m)}: \; \bv \in \boldsymbol{V}^{(m)}\right\} \subset \boldsymbol{H}^1(\Omega_0).
    $$
    The function $\bpsi^{(m)} \in \boldsymbol{\mathcal{T}}'_{\bphi^{(m)}}$ is a gradient descent direction, computed with respect to the inner product $c_{\bphi^{(m)}}$ on $\boldsymbol{V}_0^{(m)}$, and (\ref{eq:iterate}) amounts to update $\bphi^{(m+1)}$ as $\bphi^{(m+1)} = \bphi^{(m)} + \gamma^{(m)} \bpsi^{(m)}$.
\end{remark}

\subsection{Shape matching with constraints\rev{: main morphing algorithm}} \label{Shape matching with tags}
The goal of this section is to propose \rev{our main morphing algorithm. It is based on a modification} of the iterative procedure presented in the previous section so as to enforce matching conditions concerning points and lines as in \eqref{points mapping}-\eqref{curve mapping}. Specifically, we compute at each iteration $m\in \N$ a displacement field $\bu^{(m)}\in \boldsymbol{V}^{(m)}$ solution to
\begin{equation}\label{eq:linear1}
\forall \bv^{(m)}\in \boldsymbol{V}^{(m)}, \quad a_{\bphi^{(m)}}(\bu^{(m)}, \bv^{(m)}) = b_{\bphi^{(m)}}(\bv^{(m)}),
\end{equation}
for some continuous linear functional $b_{\bphi^{(m)}}: \boldsymbol{H}^1(\bphi^{(m)}(\Omega_0)) \to \mathbb{R}$ encoding the constraints \eqref{points mapping}-\eqref{curve mapping}. The updated morphing $\bphi^{(m+1)}$ is again defined by \eqref{eq:iterate}.

Let us focus more specifically on the case where $d=2$ for the sake of clarity. Then, for all $\bphi \in \boldsymbol{\mathcal T}_{\Omega_0}$ and all $\bv \in \boldsymbol{H}^1(\bphi^{(m)}(\Omega_0))$,  the quantity $b_{\bphi}(\bv)$ is defined as the sum of two terms:
$$
b_{\bphi}(\bv) = b_{\bphi}^p(\bv) + b_{\bphi}^l(\bv), 
$$
where $b_{\bphi}^p$ ( resp., $b_{\bphi}^l$) is a point-matching (resp., line-matching) linear form. 

On the one hand, the point-matching linear form is defined as follows. For all $1\leq k \leq N_p$, we consider a neighborhood $N_{k}^0$ of $\rbP_{k}^0$ in $\partial \Omega_0$ (which is taken to be small), and define the following linear form on $\boldsymbol{H}^1(\boldsymbol{\phi}(\Omega_0))$: 
\begin{align}\label{b1}
 b^p_{\boldsymbol{\phi}}(\bv) := \beta_1  \sum_{k=1}^{N_{p}}  \displaystyle \int_{\boldsymbol{\phi}(N_{k}^0)}  (\rbP_{k}- \boldsymbol{\phi}(\rbP_{k}^0))\cdot \bv ds,
\end{align}
with the user-dependent parameter $\beta_1 >0$. The aim of this term is to force each point $\rbP_{k}^0$ to match with its corresponding point $\rbP_{k}$ at convergence of the scheme. Notice that $\boldsymbol{\phi}(\rbP_{k}^0)$ is well defined since $\boldsymbol{\phi} \in \boldsymbol{W}^{1,\infty}(\Omega_0)$. 

On the other hand, the line-matching linear form is defined as follows. For any bounded closed subset $A\subset \mathbb{R}^2$, we denote by $1_A$ its characteristic function and $\boldsymbol{\Pi}_{A}(\boldsymbol{x})$ denotes one minimizer of the following minimization problem:
\begin{equation}\label{eq:min1}
\boldsymbol{\Pi}_{A}(\boldsymbol{x}) \in \mathop{\argmin}_{\boldsymbol{y} \in A}\left\|\boldsymbol{x}-\boldsymbol{y}\right\|.
\end{equation}
Such an element is not uniquely defined in general, in which case one has to make a choice among all minimizers of (\ref{eq:min1}). Then we define the \textit{vector distance function} $\boldsymbol{D}^{\partial\Omega}_{\boldsymbol{\phi}}: \bphi(\partial \Omega_0) \to \mathbb{R}^2$ as follows:
\begin{align} \boldsymbol{D}^{\partial\Omega}_{\boldsymbol{\phi}}:
\boldsymbol{\phi}(\partial\Omega_0)\ni \boldsymbol{x} \longmapsto  \boldsymbol{D}^{\partial\Omega}_{\boldsymbol{\phi}}(\boldsymbol{x}) := \sum_{k=1}^{N_l}(\Pi_{\overline{L_{k}}}(\boldsymbol{x})-\boldsymbol{x}) 1_{\boldsymbol{\phi}(L_{k}^0)}(\boldsymbol{x}) \in \R^2 .  \label{vectDistance}
\end{align}
An illustration of $\boldsymbol{D}^{\partial\Omega}_{\boldsymbol{\phi}}$ is shown in Figure \ref{vectorialDistanceRep}. The linear form $b_{\bphi}^l : \boldsymbol{H}^1(\bphi(\Omega_0)) \to \mathbb{R}$ is then defined as follows:
\begin{equation}\label{b2}
         \forall \bv \in \boldsymbol{H}^1(\bphi(\Omega_0)), \quad b^l_{\boldsymbol{\phi}}(\bv) := \beta_2 \int_{ \boldsymbol{\phi}(\partial\Omega_0)}  \big(\boldsymbol{D}^{\partial\Omega}_{\boldsymbol{\phi}}\cdot \bn_{\bphi}\big) \hspace{0.1em} \big(\bv\cdot \bn_{\bphi}\big) ds =  \beta_2 \sum_{k=1}^{N_l} \int_{ \boldsymbol{\phi}(L_k^0)} \big( (\Pi_{\overline{L_{k}}}-\bId)\cdot \bn_{\bphi} \big) \hspace{0.1em} \big(\bv\cdot \bn_{\bphi}\big) ds,
\end{equation}
for some user-dependent parameter $\beta_2>0$.

In what follows, we refer to the procedure described above as the \textbf{vector distance algorithm}.
\begin{remark}[Alternative definition]
An alternative definition for $b_{\bphi}^l$ is 
\begin{equation}\label{b22}
         \forall \bv \in \boldsymbol{H}^1(\bphi(\Omega_0)), \quad b^l_{\boldsymbol{\phi}}(\bv) := \beta_2 \int_{ \boldsymbol{\phi}(\partial\Omega_0)}  \boldsymbol{D}^{\partial\Omega}_{\boldsymbol{\phi}}\cdot \bv ds.
\end{equation}
Numerical tests were also performed with this alternative definition. Altogether, the quality of the resulting transported mesh was observed to be better when using \eqref{b2} than when using \eqref{b22}.
\end{remark}

\begin{figure}[ht] 
\begin{center}
\includegraphics[scale=0.4]{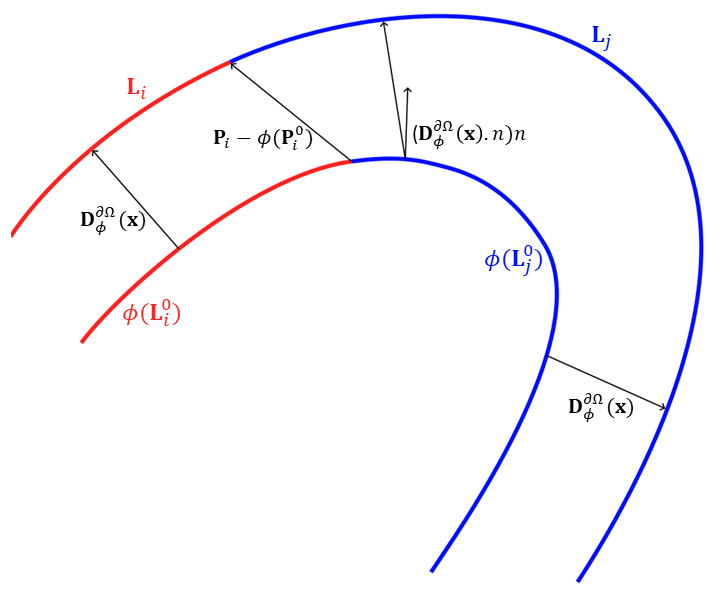}
\caption{\label{vectorialDistanceRep} Visual representation of \eqref{vectDistance}.  $\boldsymbol{D}^{\partial\Omega}_{\boldsymbol{\phi}}(\boldsymbol{x})$ is the vector that points from $\boldsymbol{x}\in \boldsymbol{\phi}(L_i^0)$ to its projection on $L_i$.}
\end{center}
\end{figure}

\subsection{Implementation details} \label{Implementaion details}

The aim of this section is to give some details about the practical implementation of the procedures described in the two previous sections.

In the initialization step, conforming meshes $\mathcal{M}_0$ and $\mathcal{M}$ of $\Omega_0$ and $\Omega$, respectively, are chosen, typically employing simplicial mesh cells. At each iteration $m \in \N$, the mesh $\mathcal M_0$ is transformed into a mesh $\mathcal M^{(m)}$ of $\Omega^{(m)}=\bphi^{(m)}(\Omega_0)$ via the morphing $\bphi^{(m)}$. The finite-dimensional space $\boldsymbol{V}^{(m)}$ is then chosen at each iteration as the classical $\mathbb{P}_1$ finite-element space associated with the mesh $\mathcal M^{(m)}$. In principle, the transported mesh $\mathcal M^{(m)}$ could contain ill-shaped elements. In such a situation, one could potentially introduce a new mesh of $\Omega^{(m)}$. This was not needed in our numerical tests. \rev{We also notice in passing that the computation of the vector distance function 
$\boldsymbol{D}^{\partial\Omega}_{\boldsymbol{\phi}^{(m)}}$ only uses the boundary mesh, whereas our implementation of the signed distance function uses the full mesh of the target domain. In both cases, the target mesh should be refined near the boundary of the target domain so as to capture its curvature. We refer the reader to Section~\ref{Numerical results offline} for some numerical illustrations.}

For each vertex $\boldsymbol{x}$ in the boundary mesh $\boldsymbol{\phi}^{(m)}(\partial \mathcal{M}_0)$, we compute $d_{\Omega}(\boldsymbol{x})$ by checking if $\boldsymbol{x}$ is inside $\mathcal{M}$ and determining the projection of $\boldsymbol{x}$ onto $\partial \mathcal{M}$. This latter operation is realized by determining the closest node on $\partial \mathcal{M}$ to $\boldsymbol{x}$ (using a KD-tree structure for example), identifying boundary elements sharing this node (forming candidate elements), and projecting $\boldsymbol{x}$ onto these elements. To compute $\boldsymbol{D}^{\partial\Omega}_{\boldsymbol{\phi}^{(m)}}(\boldsymbol{x})$, we determine the index $1\leq k \leq N_l$ such that $\boldsymbol{x}\in \boldsymbol{\phi}^{(m)}(L^0_{k}) \subset \boldsymbol{\phi}^{(m)}(\partial \mathcal{M}_0)\subset \R^2$, and compute the projection of $\boldsymbol{x}$ onto $\overline{L_{k}}$. Notice that computing $\boldsymbol{D}^{\partial\Omega}_{\boldsymbol{\phi}^{(m)}}(\boldsymbol{x})$ tends to be less costly than computing $d_{\Omega}(\boldsymbol{x})$, since we do not have to determine the position of $\boldsymbol{x}$ relative to $\partial \Omega$ to determine the sign of $d_{\Omega}(\boldsymbol{x})$. In any case, once the matching term is evaluated, we can compute $\bu^{(m)}$ by solving the variational problem \eqref{eq:var1} or \eqref{eq:linear1}. 

The value of the parameter $\gamma^{(m)}$ can be adjusted throughout the iterations, as long as it remains sufficiently small to ensure that $\bphi^{(m+1)}$ belongs to $\boldsymbol{\mathcal T}_{\Omega_0}$ (see Lemma~\ref{variation of phi}). In our numerical tests, the value of $\gamma^{(m)}$ is chosen to be equal to some constant value $\gamma>0$ which is specified below.

Let us introduce here two quantities that will be used to measure the geometrical error and thus to assess the quality of a given morphing $\bphi \in \boldsymbol{\mathcal T}_{\Omega_0}$, and to define the stopping criterion of the two iterative algorithms we have presented in the previous sections. The first quantity is based on the use of the signed distance function (and is thus suitable to define a convergence criterion for the signed distance algorithm): 
\begin{equation}\label{eq:Delta1}
\mathrm{\Delta}_1(\bphi, \Omega_0, \Omega):=
\sup\{ |d_{\Omega}(\boldsymbol{x})| \:: \: \boldsymbol{x}\in \bphi(\partial\Omega_0)\} = \left\| d_\Omega \right\|_{\boldsymbol{L}^\infty(\bphi(\partial\Omega_0))}.
\end{equation}
The second quantity relies on the vector distance function (and is thus used for the definition of a convergence criterion for the vector distance algorithm): 
\begin{equation}\label{eq:Delta2}
\mathrm{\Delta}_2(\bphi, \Omega_0, \Omega):=\sup\{\|\boldsymbol{D}_{\bphi}^{\partial \Omega}(\boldsymbol{x})\| : \; \boldsymbol{x}\in \bphi(\partial\Omega_0)\} = \left\| \boldsymbol{D}_{\bphi}^{\partial \Omega} \right\|_{\boldsymbol{L}^\infty(\bphi(\partial\Omega_0))}.
\end{equation}
More precisely, both quantities are evaluated using the meshes $\mathcal{M}_0$ and $\mathcal{M}$, \rev{and the $L^\infty$-norms in~\eqref{eq:Delta1} and~\eqref{eq:Delta2} are evaluated approximately by either sampling only the boundary nodes of $\boldsymbol{\phi}(\partial \mathcal{M}_0)$ or using additionally 9 interior points on each boundary edge of $\boldsymbol{\phi}(\partial \mathcal{M}_0)$. The former strategy is obviously less computationally intensive. As illustrated below in our experiments (see Figure~\ref{errors_different_meshes}), it is viable provided the boundary mesh $\boldsymbol{\phi}(\partial \mathcal{M}_0)$ is sufficiently refined.}
Finally, given a stopping threshold $\epsilon>0$, \rev{which should ideally be of the order of the size of the boundary elements of the morphed reference mesh, the convergence criterion for the iterative algorithm is $\Delta_i(\bphi^{(M)},\Omega_0,\Omega) < \epsilon$ for $i=1,2$}. After convergence at iteration $M$, we perform one final correction step, by computing a finite element approximation of the unique solution $\bu^* \in \boldsymbol{H}^1(\bphi^{(M)}(\Omega_0))$ of

\begin{align}
    \Biggl\{
    \begin{array}{ll}
         - {\rm div}\left( \sigma(\bu^*)\right) = 0, & \mbox{ in }\bphi^{(M)}(\Omega_0), \\
       \bu^* = \boldsymbol{D}^{\partial\Omega}_{\boldsymbol{\phi}^{(M)}}, & \mbox{ on } \bphi^{(M)}(\partial\Omega_0).\\
    \end{array}
\end{align}
The final morphism is set to be $\boldsymbol{\phi}^*:= (\bId + \bu^*) \circ \boldsymbol{\phi}^{(M)}$. This final correction step guarantees that $\boldsymbol{\phi}^*(\partial \Omega_0)$ coincides with $\partial \Omega$. An illustration of the output of this final correction step is presented in Figure~\ref{correctice step illustration}.

\begin{figure}[ht] 
     \centering
     \begin{subfigure}[b]{0.45\textwidth}
         \centering
        \includegraphics[scale=0.15]{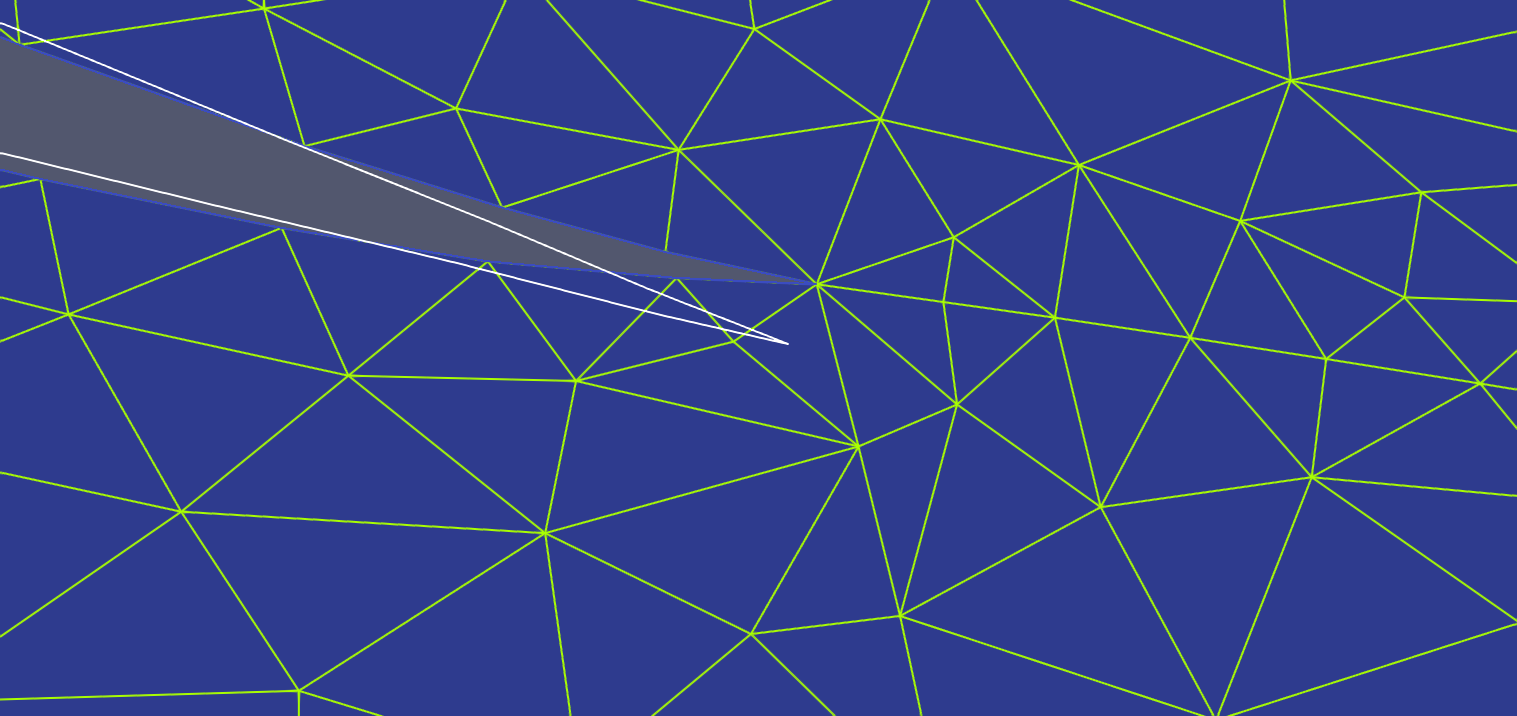}
     \caption{ Mesh of $\boldsymbol{\phi}^{(M)}(\Omega_0)$ .}
     \end{subfigure}
     \hfill
     \begin{subfigure}[b]{0.45\textwidth}
         \centering
\includegraphics[scale=0.15]{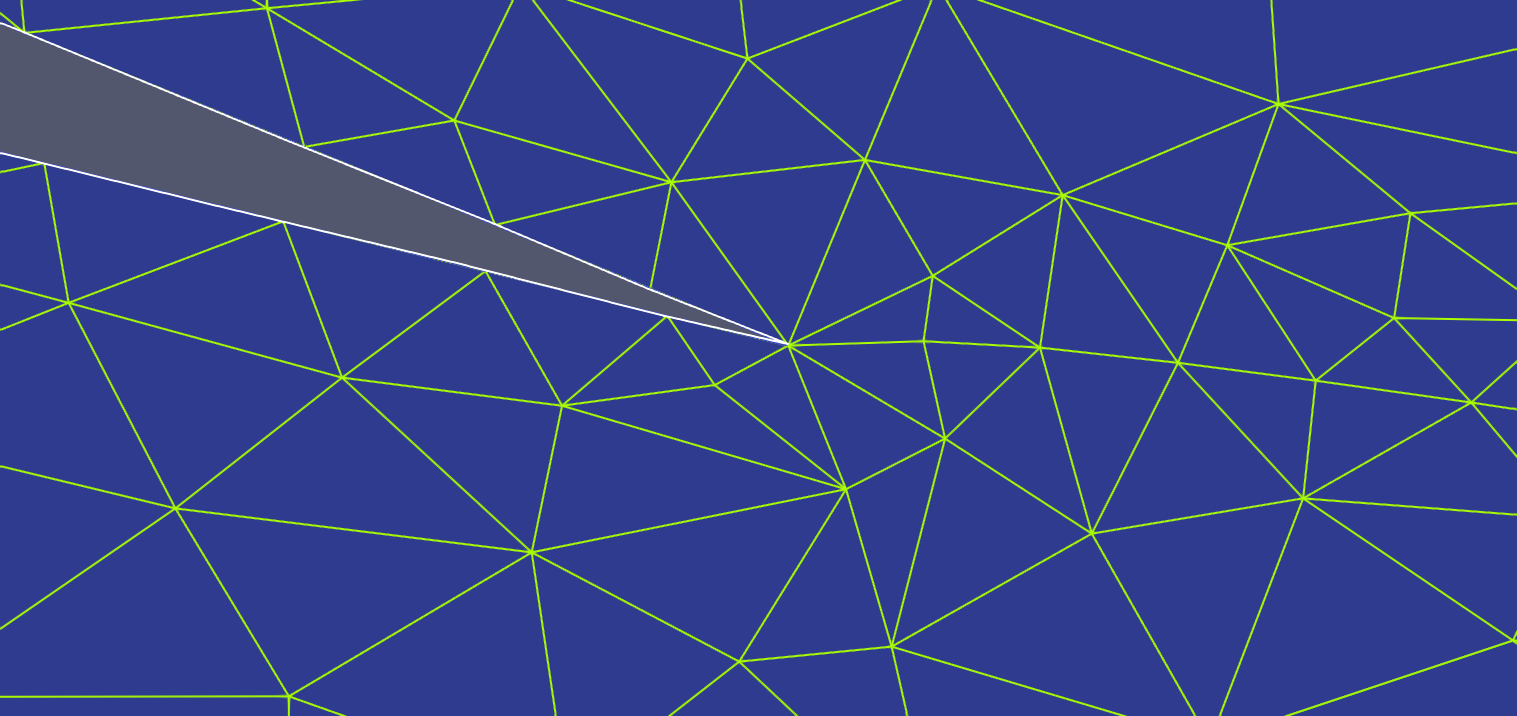}
\caption{ Final morphed mesh aligned with $\partial \Omega$.}
     \end{subfigure}
\caption{\label{correctice step illustration} Illustration of the final correction step.}
\end{figure}

\subsection{Numerical results} \label{Numerical results offline}
In this section, we present numerical results obtained with the procedures described in the previous sections on two two-dimensional test cases.  All the results were obtained using the Muscat library \cite{muscat2024}. 

\subsubsection{Tensile2D dataset} \label{Tensile2D offline}
The first example is taken from the dataset in \cite{casenave_2023_10124594}. For all $R>0$, the set $B(R)$ is defined as $B(R):=\{ (x,y) \in \R^2 / (x-1)^2+y^2\leq R^2\} \bigcup \{ (x,y) \in \R^2 / (x+1)^2+y^2\leq R^2\}$. We consider the reference domain $\Omega_0:=[-1,1]^2\backslash B(0.5)$ and the target domain $\Omega:=[-1,1]^2\backslash B(0.2)$, both shown in Figure~\ref{Tensile2D geo offline}a and~\ref{Tensile2D geo offline}b, respectively. \rev{We emphasize that the parameter $R$ is not used in the morphing computation, but merely as a label to enumerate the geometries in the dataset.} We consider $N_p=4$ control points in $\partial\Omega_0$ having coordinates $(-1,0.5),(-1,-0.5),(1,0.5),(-1,-0.5)$. We consider $N_l=4$ control lines on $\partial\Omega_0$  which consist of the two half-circles (highlighted in red on the reference domain in Figure \ref{Tensile2D geo offline}) and the top and bottom parts of the boundary (highlighted in green). In Figure \ref{fig:Tensile2D ref mesh}, we show the mesh of the reference domain $\Omega_0$, provided in the dataset, superimposed to the boundary of the target domain $\Omega$. The reference mesh has approximately 19k elements. \rev{In order to compute the signed distance function, we use the mesh of the target domain provided in the dataset, which is composed of 23k elements. In Table 1, we report the normalized time to compute the signed distance function for different resolutions of the target mesh. The normalization is made with respect to the time to compute the signed distance function for a target mesh of 1.3k elements. We observe that the size of the target mesh does not excessively impact the computation of the signed distance function.}

\begin{table}[h]
\centering
\rev{\begin{tabular}{|c|c|c|c|} 
\hline
 \#elements (target mesh) & 1.3k & 23k & 356k \\
\hline
 Time ratio & 1   & 1.22 &   1.76   \\
\hline
\end{tabular}}
\caption{\rev{Normalized computational time for the signed distance function on different meshes of the target domain.}}
\label{tab:time_sdf}
\end{table}

The aim of the following tests is to highlight the advantages of the vector distance algorithm with respect to the distance function algorithm. The vector distance algorithm is run with the parameters $E:=1, \nu:=0.3,\alpha:=200, \gamma:=8$, $\beta_1:=0$ and $\beta_2:=1$. The convergence is obtained after 145 iterations with a tolerance $\epsilon =10^{-3}$ and a stopping criterion based on $\Delta_2$. The evolution of the deformed mesh is shown in Figure \ref{fig:Evolution_VDF_SDF} (top row). For comparison, we show in the same figure (bottom row) the evolution of the mesh using the signed distance algorithm, with the parameters $E:=1, \nu:=0.3,\alpha:=200, \gamma:=8$. For the signed distance algorithm, convergence is attained after 180 iterations with a stopping criterion based on $\Delta_1$ and the same value of $\epsilon$ as above. When using the signed distance function, the half-circles in $\partial \Omega_0$ are not mapped onto the half-circles in $\partial \Omega$. 
\rev{Another advantage of the vector distance algorithm is its computational effectivity. Indeed, in our implementation, evaluating the vector distance is typically two times faster than the signed distance.} As mentioned in Section \ref{Implementaion details}, this is due to the fact that calculating the signed distance function is more expensive than calculating the vector distance function. 

\begin{figure}[ht]
     \centering
     \begin{subfigure}[b]{0.3\textwidth}
         \centering
        \includegraphics[scale=0.3]{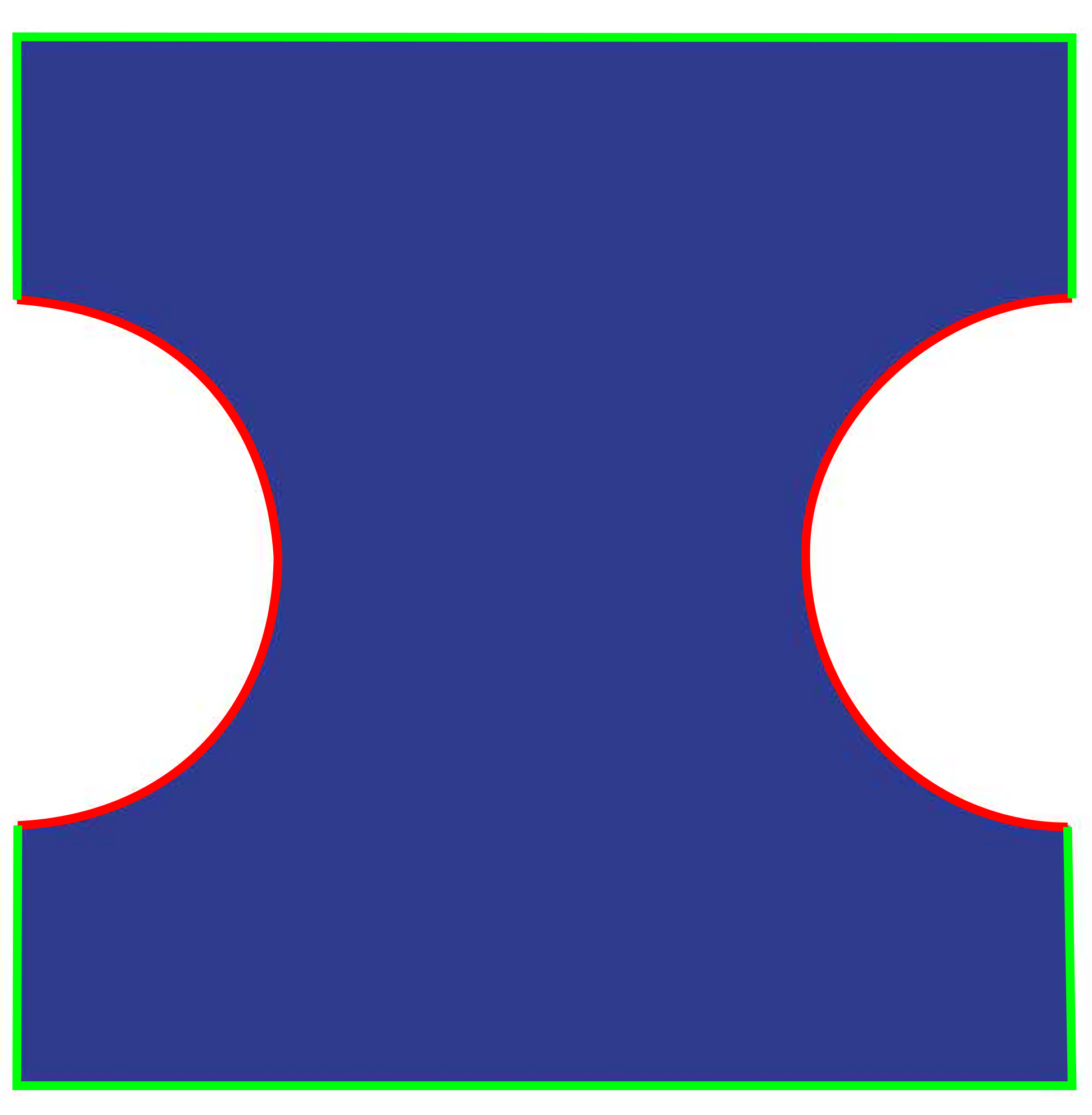}
     \caption{ Reference domain $\Omega_0$.}
     \end{subfigure}
     \hfill
     \begin{subfigure}[b]{0.3\textwidth}
         \centering
\includegraphics[scale=0.3]{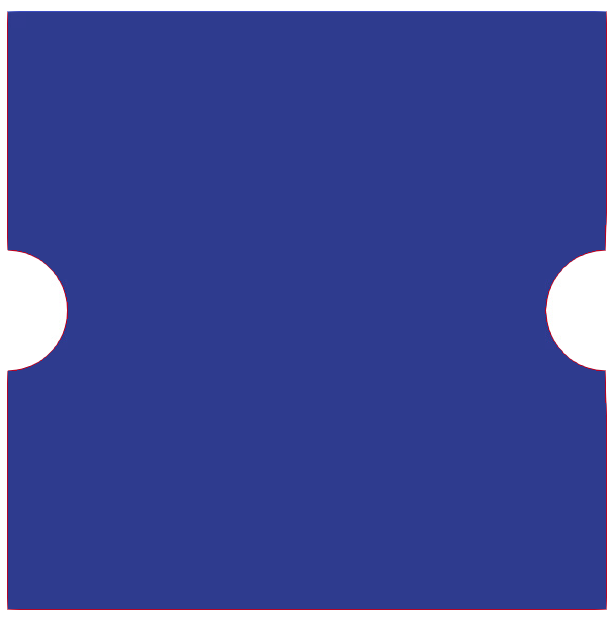}
\caption{ Target domain $\Omega$.}
     \end{subfigure}
      \hfill
          \begin{subfigure}[b]{0.3\textwidth}
         \centering
        \includegraphics[scale=0.3]{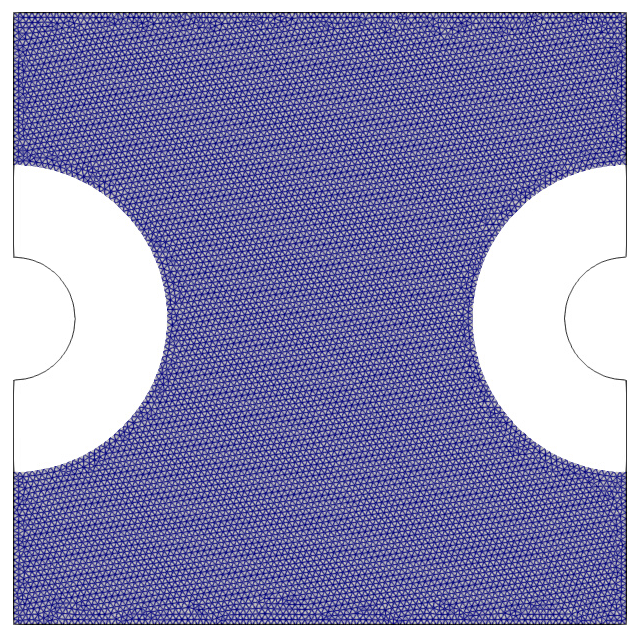}
        \caption{ Mesh $\mathcal M_0$ of the reference domain $\Omega_0$.}
        \label{fig:Tensile2D ref mesh}
     \end{subfigure}
\caption{\label{Tensile2D geo offline} Reference and target domains, with the partition used on the boundary of the reference domain.}
\end{figure}

\begin{figure}[ht] 
     \centering
     \includegraphics[scale=0.3]{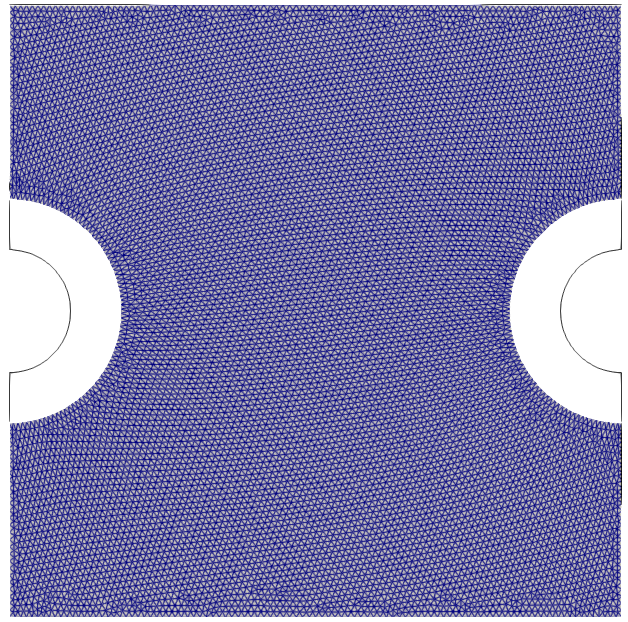}
     \includegraphics[scale=0.3]{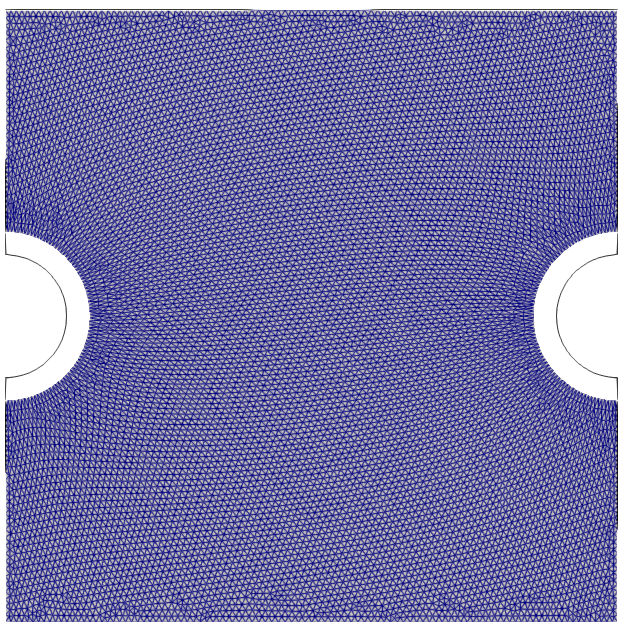}
     \includegraphics[scale=0.3]{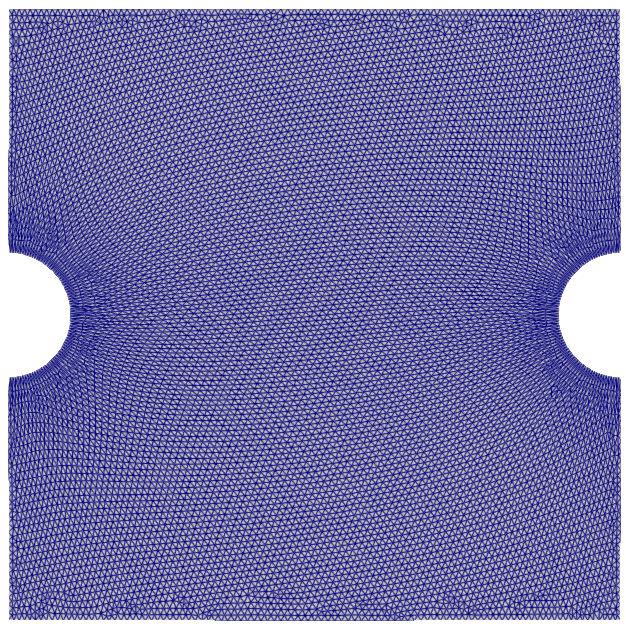}

\vspace{1em}

     \includegraphics[scale=0.3]{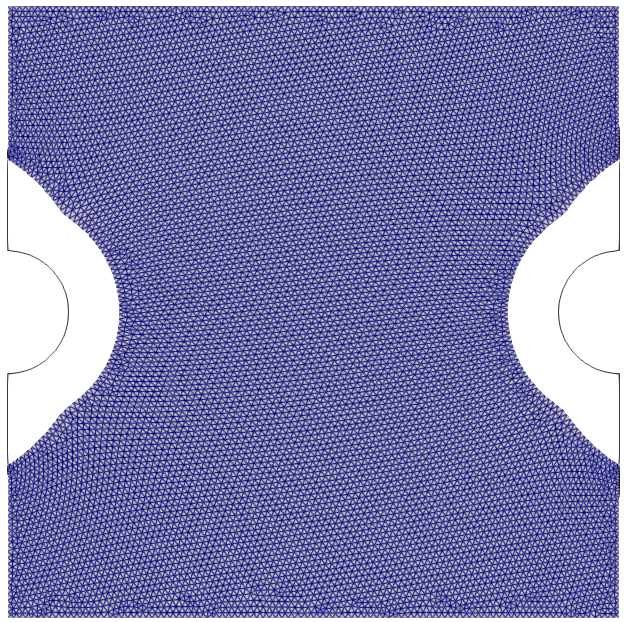}
     \includegraphics[scale=0.3]{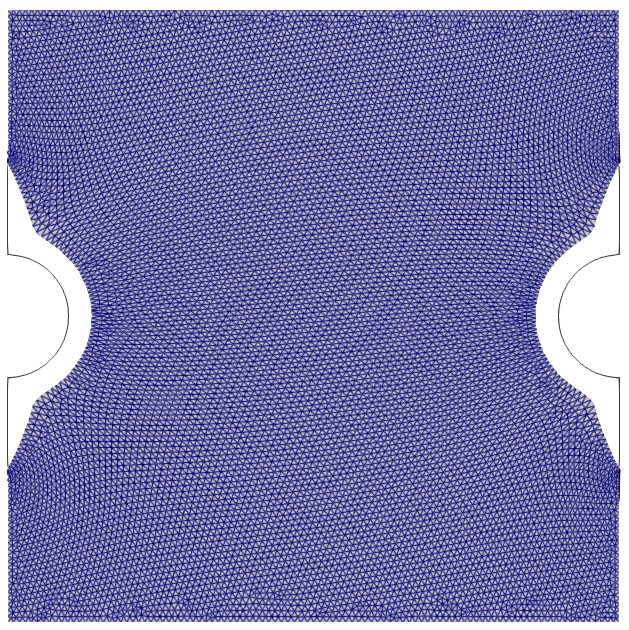}
     \includegraphics[scale=0.3]{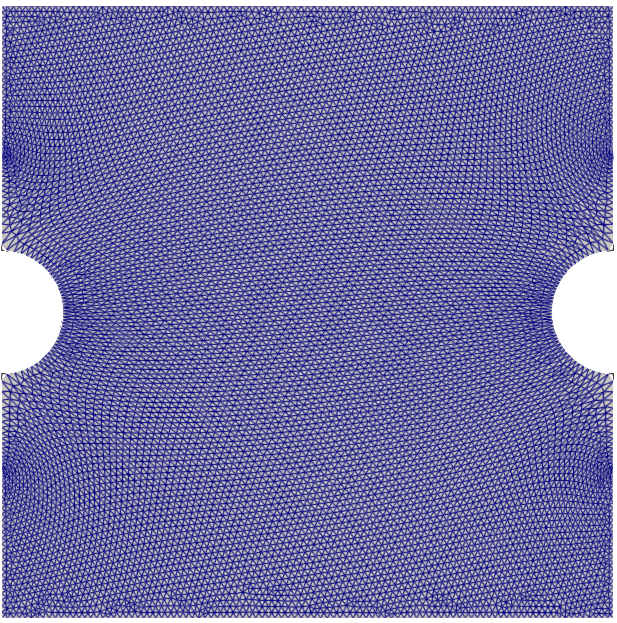}

\caption{\label{fig:Evolution_VDF_SDF}  Evolution of the mesh. Top row: vector distance algorithm;
bottom row: signed distance algorithm. Left column: 15 iterations; middle column: 35 iterations; right column: at convergence.}
\end{figure}

Figure~\ref{Delta_tensile2D} illustrates the behavior of the geometrical error $\Delta_1$ (a) and $\Delta_2$ (b) as a function of the number of iterations of the chosen algorithm (signed distance function ($sdf$) or vector distance function ($vdf$)). We observe that both quantities $\Delta_1^{sdf}$ and $\Delta_1^{\mathrm{vdf}}$ converge to $0$ as the number of iterations increases. However, $\Delta_2^{\mathrm{vdf}}$ also converges to $0$ with respect to the number of iterations, whereas $\Delta_2^{sdf}$ does not. 
\begin{figure}[th] 
     \centering
     \begin{subfigure}[b]{0.49\textwidth}
         \centering
        \includegraphics[scale=0.6]{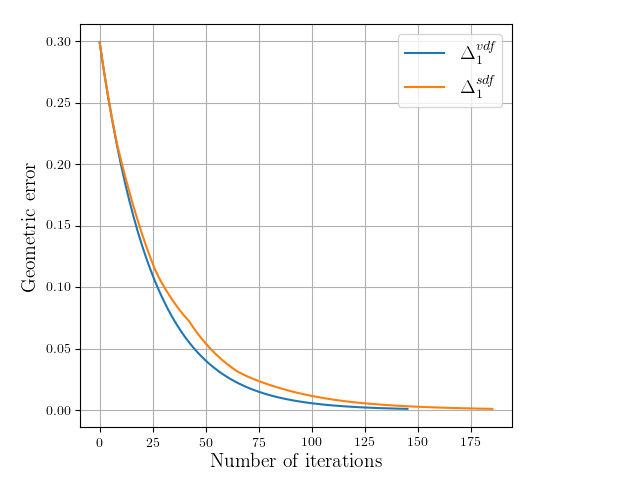}
     \caption{ Evolution of $\Delta_1$.}
     \end{subfigure}
     \hfill
     \begin{subfigure}[b]{0.49\textwidth}
         \centering
\includegraphics[scale=0.6]{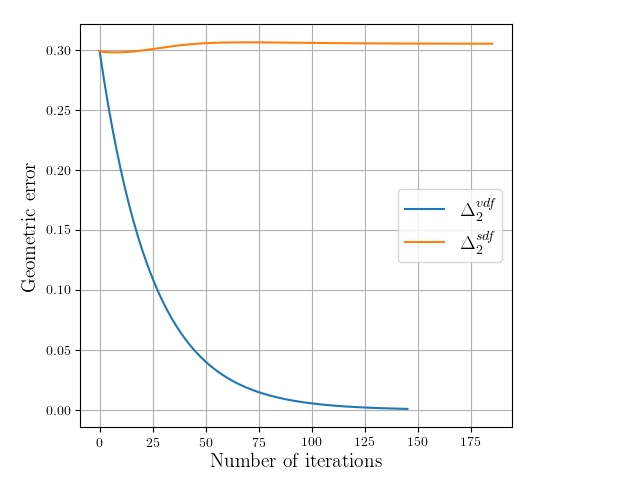}
\caption{ Evolution of $\Delta_2$.}
     \end{subfigure}
\caption{\label{Delta_tensile2D} Evolution of $\Delta_1$ and $\Delta_2$ for the two algorithms for one of the samples. Using the vector distance algorithm (so that $\Delta_2$ converges $0$), we also have that $ \Delta_1$ converges to $0$. This is not necessarily the case when using the signed distance algorithm. We can have that $\Delta_1$ converges to $0$ without having $\Delta_2$ converging to $0$. }
\end{figure}

Figure~\ref{Delta_tensile2D_samples} shows the average, minimum and maximum values of $\Delta_1$ (a) and $\Delta_2$ (b) as a function of the number of iterations of the algorithm. We observe that both quantities $\Delta_1^{\mathrm{vdf}}$ and $\Delta_2^{\mathrm{vdf}}$ converge exponentially to $0$ with respect to the number of iterations, which is not the case of $\Delta_1^{sdf}$. This highlights another advantage of the vector distance algorithm in comparison to the signed distance algorithm. 

\begin{figure}[htb] 
     \centering
     \begin{subfigure}[b]{0.46\textwidth}
         \centering
        \includegraphics[scale=0.30]{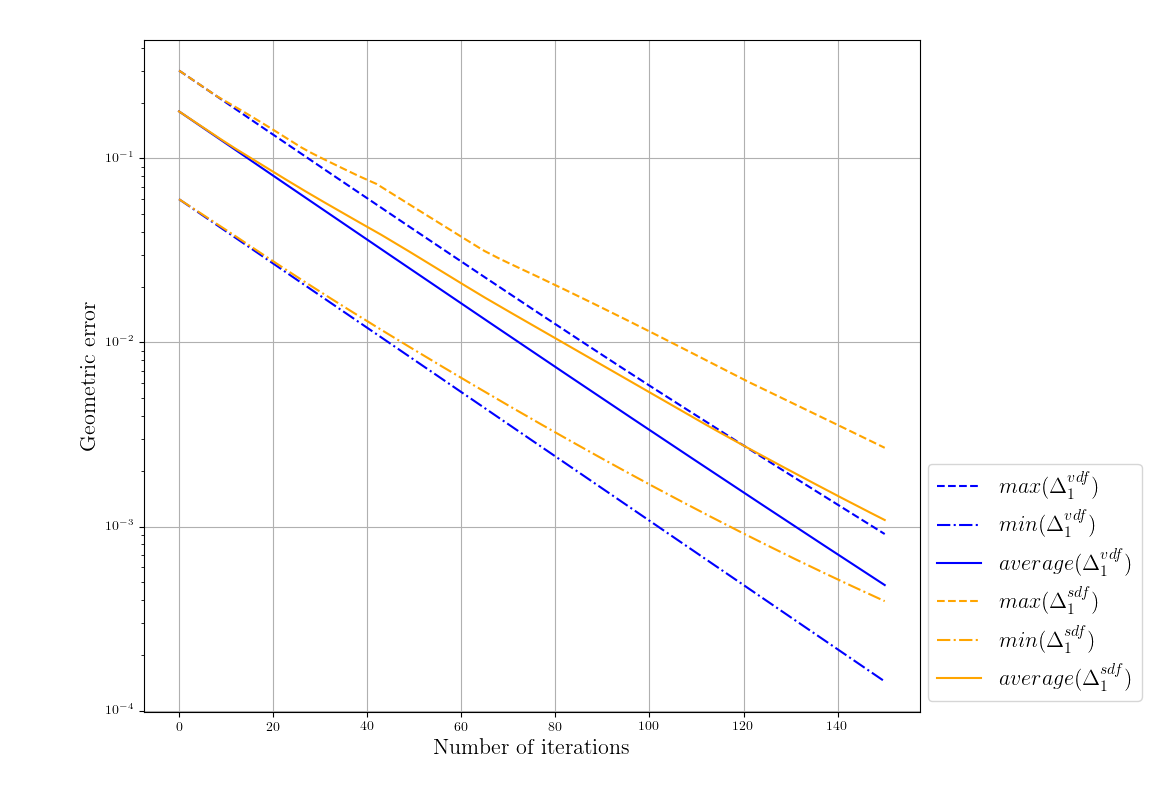}
     \caption{ Average, maximum and minimum error $\Delta_1$ on a subset of 10 samples calculated using the two formulations. }
     \end{subfigure}
     \hfill
     \begin{subfigure}[b]{0.46\textwidth}
         \centering
\includegraphics[scale=0.30]{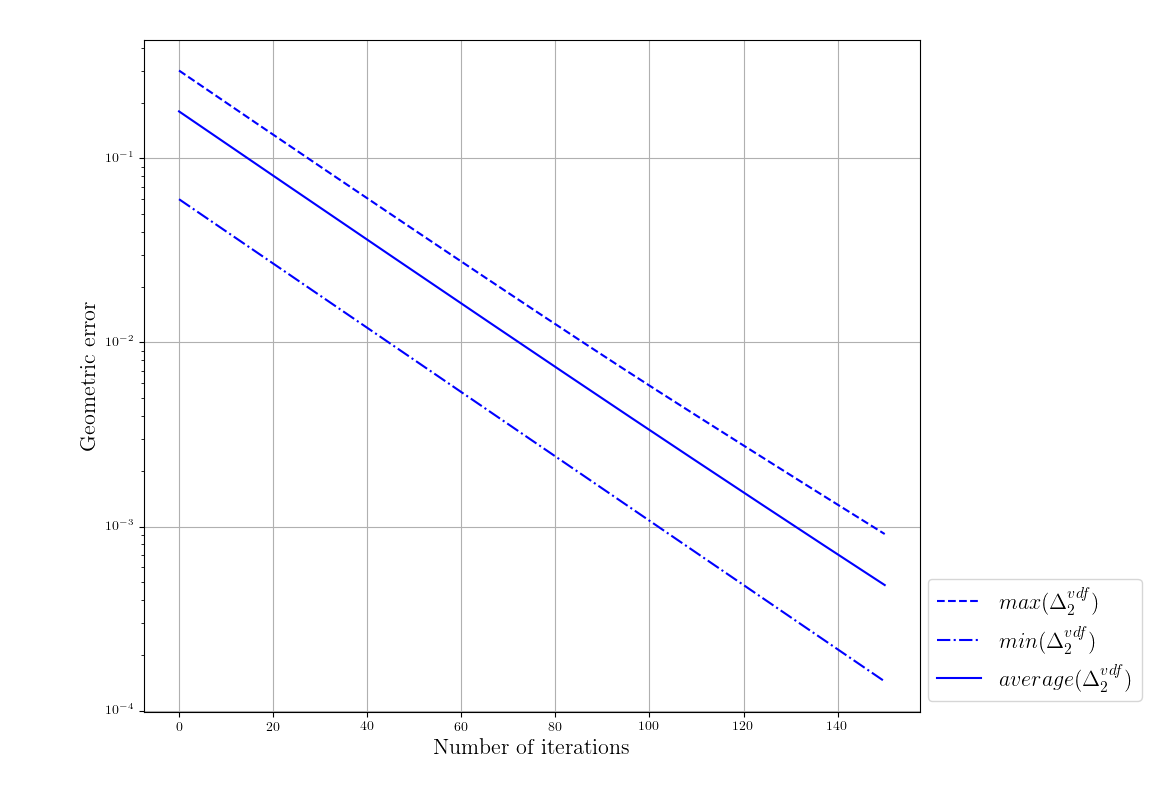}
\caption{ Average, maximum and minimum error $\Delta_2$ on a subset of 10 samples calculated using the vector distance formulation.}
     \end{subfigure}
\caption{\label{Delta_tensile2D_samples} Average, minimum and maximum geometrical errors $\Delta_1$ and $\Delta_2$ on a subset of 10 samples. }
\end{figure}

\rev{In order to illustrate the influence of the reference mesh on the morphing computation, we consider three different meshes of $\Omega_0$. The first one, $\mathcal{M}_0$, is the mesh provided in the dataset, that has approximately 19k volume elements and 562 boundary elements. We consider also $\mathcal{M}_0^{\rm fine}$, a mesh with only 12k volume elements, but which is finer than $\mathcal{M}_0$ at the boundary, with up to 900 boundary elements. Finally, we consider a coarse mesh, $\mathcal{M}_0^{\rm coarse}$, with 158 volume elements and 39 boundary elements. The three reference meshes are shown in Figure \ref{fig:diffrent refMesh}, and in Figure \ref{errors_different_meshes}, we report the geometrical error $\Delta_2^{\mathrm{vdf}}$ on the three reference meshes and the same target mesh $\mathcal{M}$, as a function of the number of iterations. Notice that the $L^\infty$-norms in~\eqref{eq:Delta2} are evaluated either at the boundary nodes and 9 additional internal points in each boundary edge (solid lines) or only at boundary nodes (dashed lines). The first observation is that the error gets smaller as the number of boundary elements increases, whereas the number of elements inside the domain is much less relevant. Furthermore, evaluating the $L^\infty$-norm only at the boundary nodes leads (expectedly) to smaller errors. For the coarse mesh, the difference is quite pronounced as the errors become smaller (below $10^{-3}$), but this is not the case for the finer meshes. We can draw two conclusions: (i)
to better approximate the target mesh $\mathcal{M}$, the reference mesh should be sufficiently refined near the boundary; (ii) for geometrical errors of the order of $10^{-3}$, evaluating the $L^\infty$-norm at boundary nodes is sufficient. We will use this setting in what follows.}

\begin{figure}[htb]
   \centering

     \begin{subfigure}[b]{0.3\textwidth}
         \centering
         \includegraphics[scale=0.3]{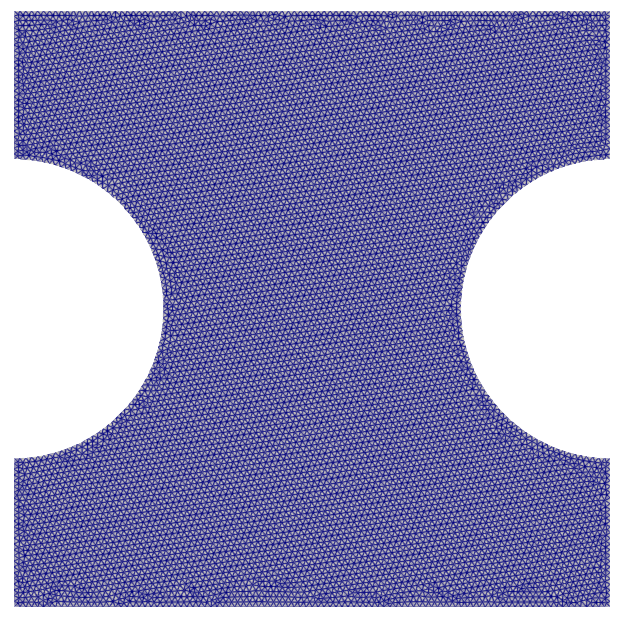}
     \caption{Reference mesh $\mathcal{M}_0$ in the dataset.}
     \end{subfigure}
     \hfill
     \begin{subfigure}[b]{0.3\textwidth}
         \centering
\includegraphics[scale=0.3]{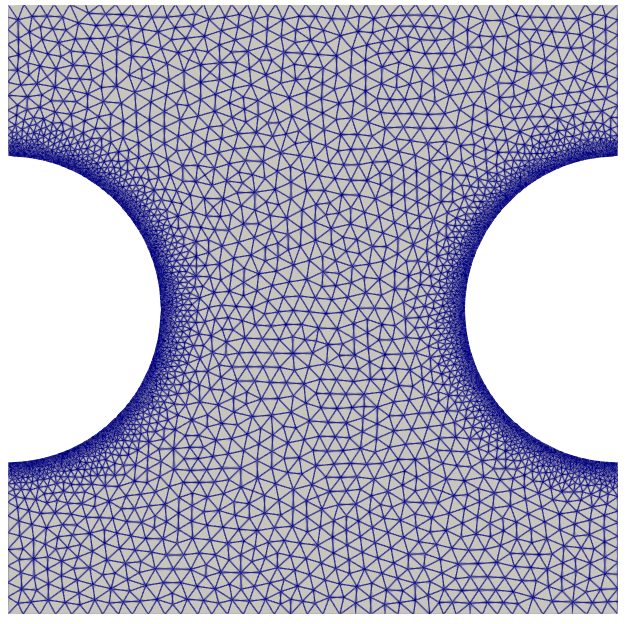}
\caption{Refined mesh near the boundary, $\mathcal{M}_0^{\rm fine}$.}
     \end{subfigure}
      \hfill
          \begin{subfigure}[b]{0.3\textwidth}
         \centering
        \includegraphics[scale=0.3]{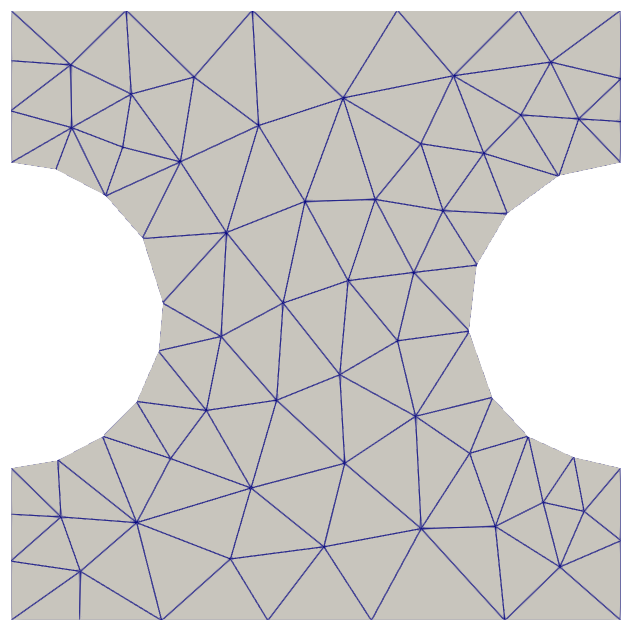}
        \caption{Coarse mesh, $\mathcal{M}_0^{\rm coarse}$. }
     \end{subfigure}
\caption{\label{fig:diffrent refMesh} \rev{Three meshes of the reference domain $\Omega_0$.}}
\end{figure}

\begin{figure}[htb] 
\begin{center}
\includegraphics[scale=0.5]{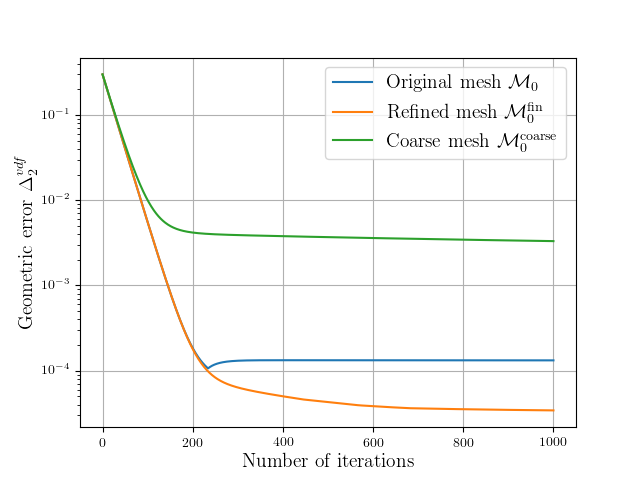}
\caption{\label{errors_different_meshes} \rev{Geometric error $\Delta_2^{\mathrm{vdf}}$ on the three meshes shown in Figure~\ref{fig:diffrent refMesh}. Solid lines: $L^\infty$-norms in~\eqref{eq:Delta2} evaluated at the boundary nodes and 9 additional internal points in each boundary edge; dashed lines: $L^\infty$-norms evaluated only at boundary nodes.}}
\end{center}
\end{figure}

\rev{Finally, we assess the quality of the morphed mesh in terms of shape-regularity
evaluated as the maximum over the mesh elements of the ratio of the cell diameter to the 
length of the smallest edge. Results are reported in the left panel of Figure~\ref{fig:Mesmetric}.
We observe that the mesh quality somewhat deteriorates, with an increase of the 
shape-regularity parameter by a factor of three at the end of the iterations. This seems reasonable in view of the large tangential deformations required in this test case. }

\begin{figure}[htb]
     \centering
     \includegraphics[scale=0.35]{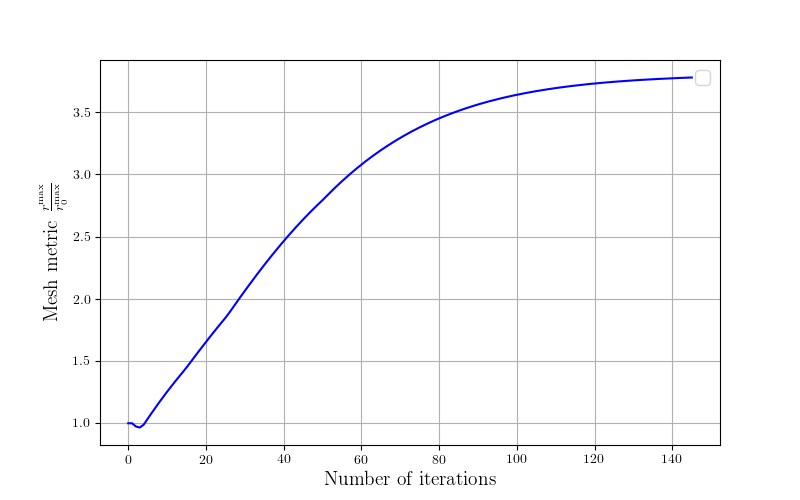} \qquad
     \includegraphics[scale=0.35]{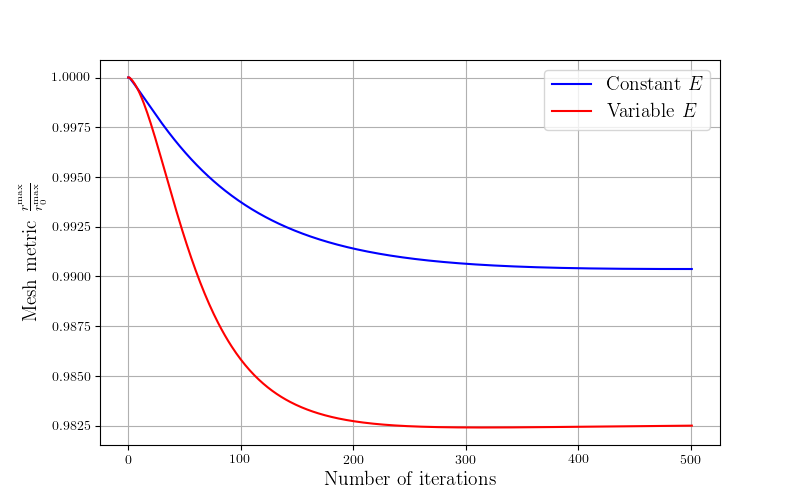}
\caption{\label{fig:Mesmetric} \rev{Shape-regularity of the morphed meshes 
(evaluated as the maximum over the mesh elements of the ratio of the cell diameter to the 
length of the smallest edge) as a function of the number of iterations. Left panel: Tensile2D dataset; Right panel: AirfRANS dataset (see Section~\ref{AirfRANS offline}).}}
\end{figure}

\subsubsection{AirfRANS dataset} \label{AirfRANS offline}
In this second test, we consider 2D airfoils taken from the dataset in \cite{bonnet2022airfrans}. \rev{The airfoils in the dataset are generated from some parametrization, but, once again, this parametrization is not used in the morphing computation.}

We choose two samples, one as the reference and the other as the target domain. The morphing should map the lower wing surface (resp., the upper wing surface) of the airfoil $\Omega_0$ to the lower wing surface (resp., the upper wing surface) of the airfoil $\Omega$. We consider $N_p=2$ target points, the two points being located at the leading edge $(0,0)$ and the trailing edge $(1,0)$ of the wing. At the start of the algorithm, these points coincide in the reference and the target geometries, but do not remain coincident through all the iterations. The external boundary representing the far field is fixed. The mesh of $\Omega_0$ that is used to compute the morphing is different than the mesh provided in the dataset. To alleviate the computational burden, we use a coarse shape-regular mesh of $\Omega_0$ with approximately 8000 elements. Morphings that are computed on the coarse mesh can then be interpolated on the original finer meshes of the dataset (see Figure \ref{Naca mesh offline after}). Note, however, that this step may be delicate since the interpolation of the morphing may not preserve bijectivity in general, although we never encountered this issue in our numerical results. 

The parameters used for the simulations are $E:=0.1, \nu:=0.3,\alpha:=500$, $\gamma=5$, $\beta_1:=10$ and $\beta_2:=1$. We observe that the signed distance algorithm does not always converge: this is the reason why we only present results obtained with the vector function algorithm on the AirfRANS dataset. 
The value of the stopping criterion is chosen to be equal to $\epsilon = 5 \times 10^{-4}$. Convergence is obtained after 492 iterations, in about 92 seconds. The evolution of the deformation of the reference airfoil is shown in Figure~\ref{fig:Evolution_VDF_naca} after 50, 100 and 492 iterations.

\begin{figure}[ht] 
     \centering
     \begin{subfigure}[b]{0.3\textwidth}
         \centering
        \includegraphics[scale=0.1]{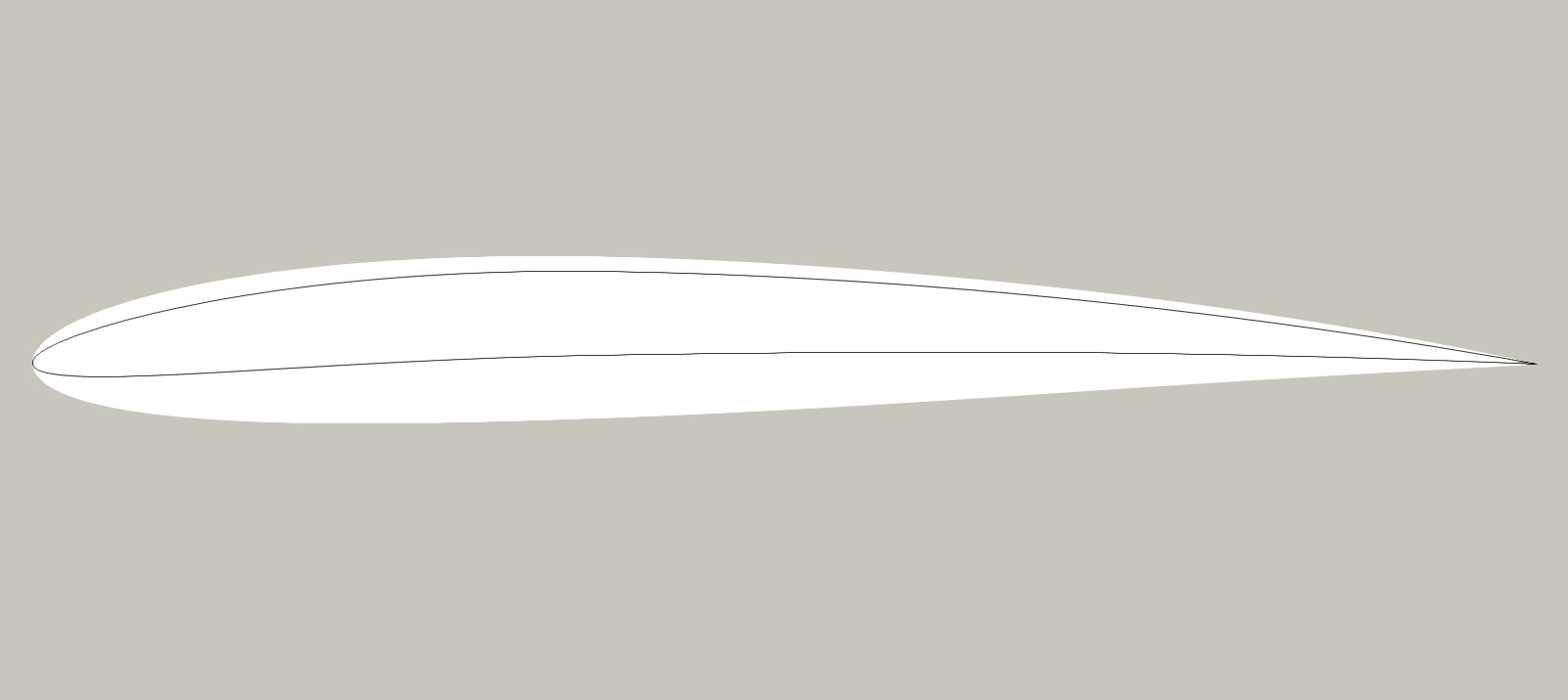}
        \caption{Deformed airfoil after 50 iterations.}
     \end{subfigure}
     \hfill
     \begin{subfigure}[b]{0.3\textwidth}
         \centering
\includegraphics[scale=0.1]{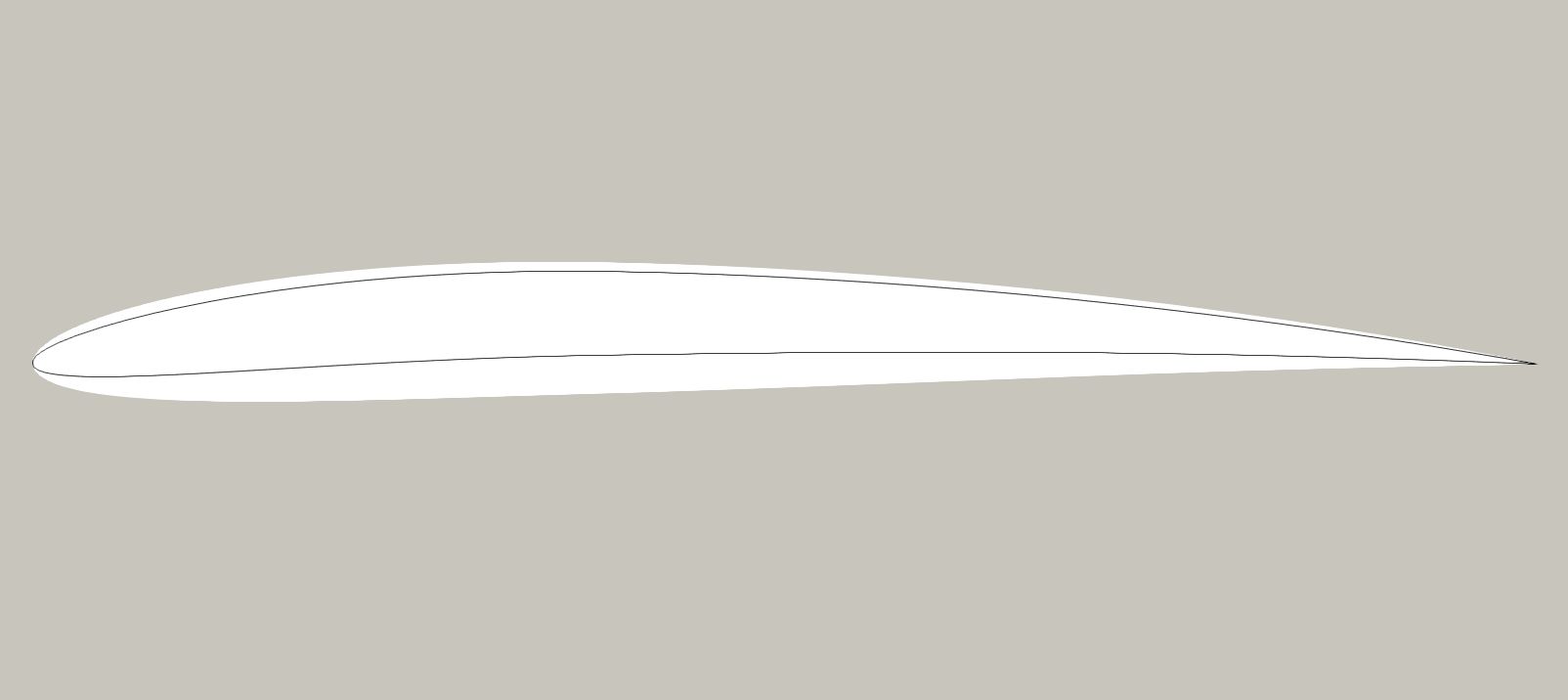}
     \caption{Deformed airfoil after 100 iterations.}
     \end{subfigure}
     \hfill
     \begin{subfigure}[b]{0.3\textwidth}
         \centering
\includegraphics[scale=0.1]{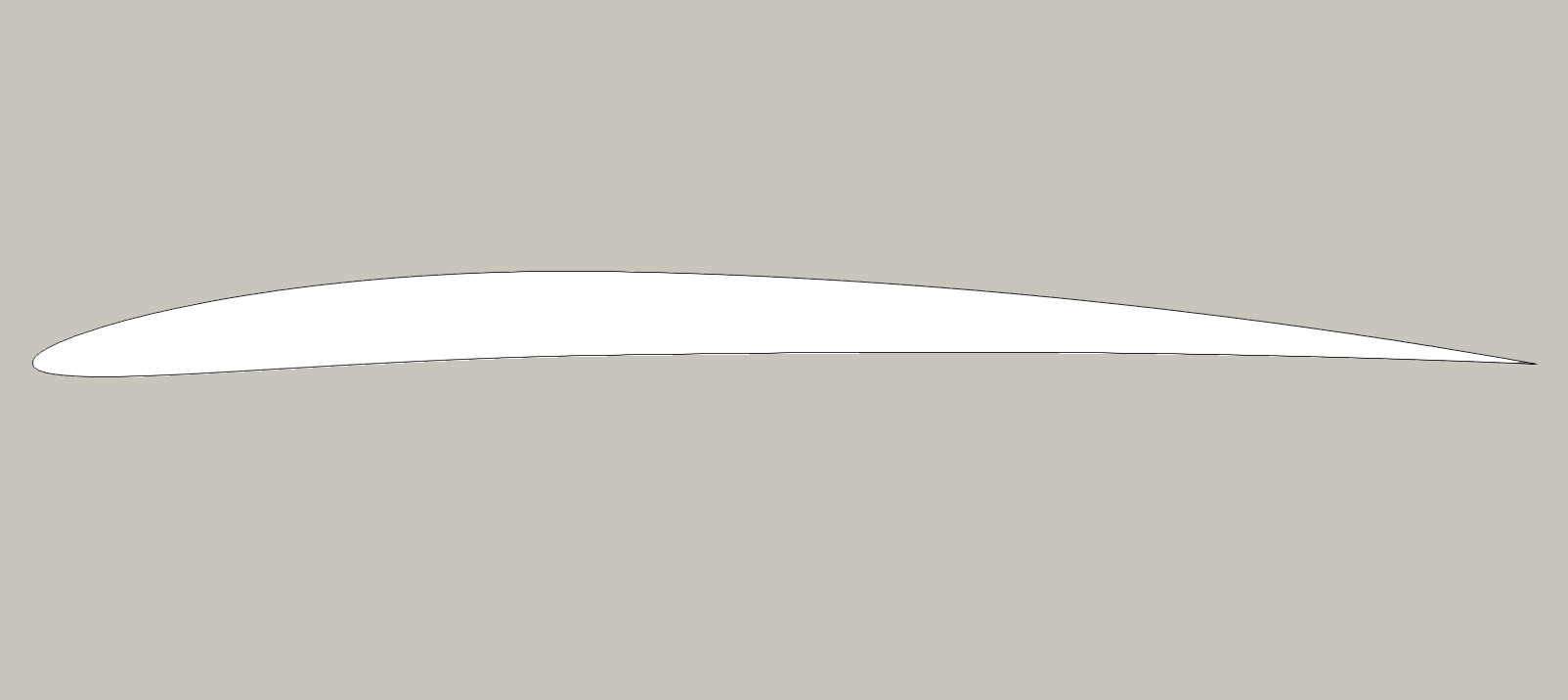}
     \caption{Deformed airfoil at convergence.{\color{white} align images}}
            
     \end{subfigure}
\caption{\label{fig:Evolution_VDF_naca}  Evolution of the airfoil using the vector distance algorithm.}
\end{figure}

\begin{figure}[ht]
     \centering
     \begin{subfigure}[b]{0.3\textwidth}
         \centering
        \includegraphics[scale=0.2]{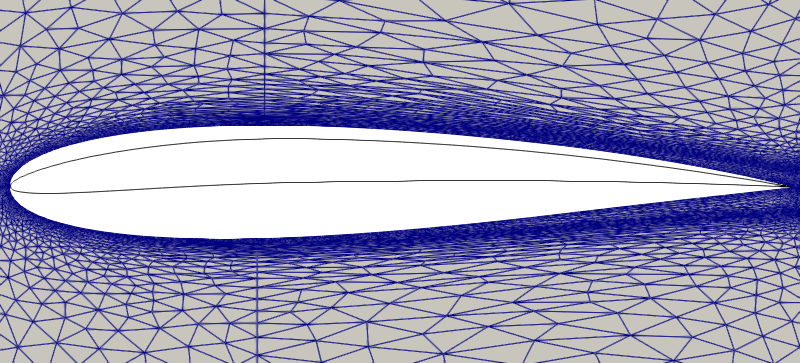}
     \caption{Mesh used to calculate the morphing.\vspace{1.1em}}
     \end{subfigure}
     \hfill
          \begin{subfigure}[b]{0.3\textwidth}
         \centering
        \includegraphics[scale=0.2]{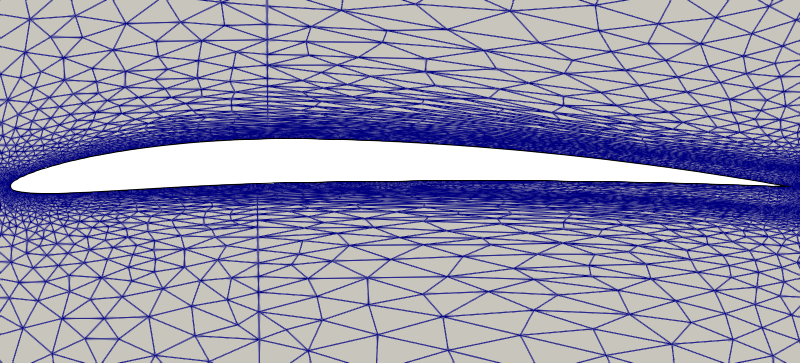}
     \caption{Deformation of the mesh at convergence.\vspace{1.1em}}
     \end{subfigure}
     \hfill
     \begin{subfigure}[b]{0.3\textwidth}
         \centering
\includegraphics[scale=0.2]{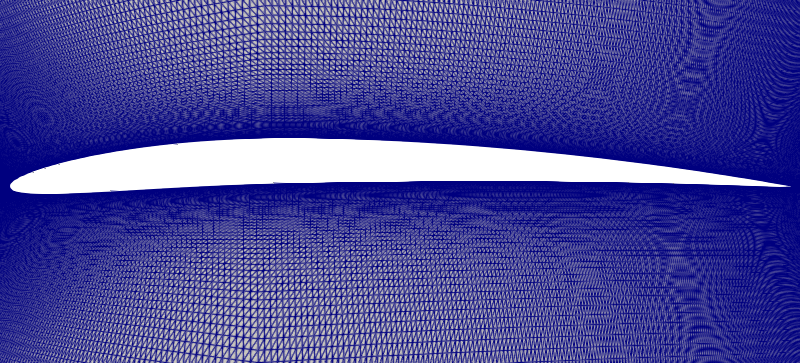}
\caption{Morphed mesh in the dataset obtained by interpolation. }
     \end{subfigure}
\caption{\label{Naca mesh offline after} Morphing of the two meshes.}
\end{figure}

In Figure~\ref{Delta_2_naca_samples} we plot the average, minimum and maximum value of $\Delta_2^{\mathrm{vdf}}$ as a function of the number of iterations over a set of $10$ samples. As in the previous test case, we observe that the vector distance algorithm converges exponentially with respect to the number of iterations. \rev{Finally, in the right panel of Figure~\ref{fig:Mesmetric}, we report the shape-regularity parameter of the morphed meshes as a function of the iterations. Here, we only consider the shape-regularity of the mesh elements touching the airfoil. We observe that the shape-regularity parameter remains practically unmodified along the iterations. A known device from the literature \cite{baker2002mesh,masud2007adaptive} to improve the shape-regularity of the morphed mesh is to consider a spatially variable Young modulus depending on the size of the mesh elements. Using this technique allows a (slight) further improvement on mesh quality (see the red curve in Figure~\ref{fig:Mesmetric}).}

\begin{figure}[ht] 
\begin{center}
\includegraphics[scale=0.29]{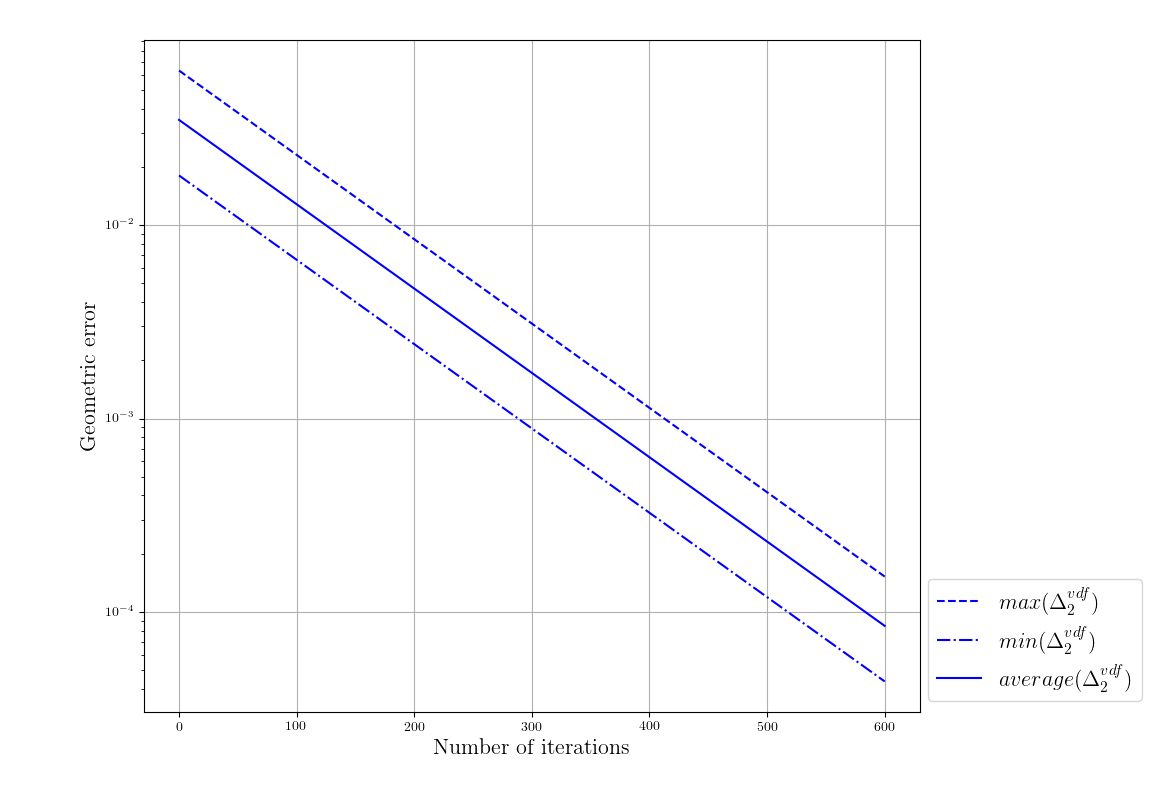}
\caption{\label{Delta_2_naca_samples} Average, maximum and minimum geometrical errors $\Delta_2$ in logarithmic scale on a subset of 10 samples, using the vector distance algorithm.  }
\end{center}
\end{figure}

\FloatBarrier

\section{Reduced-order modeling with geometric variability} \label{Section Online}
The proposed high-fidelity morphing technique requires the resolution of a linear elasticity problem at each iteration, a process which can be time-consuming. In the context of model-order reduction with geometric variability, fast computation of this morphing is crucial for deriving efficient reduced-order models. Therefore, we introduce here a reduction technique aiming at speeding up these calculations to improve overall efficiency.
\rev{Since the numerical tests of the previous section showed the superiority of the vector distance algorithm from Section~\ref{Shape matching with tags} over the signed distance algorithm from Section~\ref{shape matching opti}, we henceforth use exclusively the former (the latter being also applicable in a reduced-order modeling context).}

\subsection{Offline phase}
Given $n$ target domains $\{\Omega_i\}_{1\leq i \leq n}$ which compose our training set, we start by calculating the $n$ morphings $\{\boldsymbol{\phi}_i\}_{1\leq i \leq n} \subset \boldsymbol{\mathcal T}_{\Omega_0}$ from a fixed reference domain $\Omega_0$ to each target domain $\Omega_i$ so that $\boldsymbol{\phi}_i(\Omega_0)=\Omega_i$ for all $1\leq i \leq n$. This can be done using either the signed distance algorithm or the vector distance algorithm presented in the previous section. 
\rev{We emphasize that, in the offline phase, the morphings are constructed by iterating to convergence the iterative algorithm as described in Section~\ref{Implementaion details}.}

We then apply snapshot-POD (Proper Orthogonal Decomposition) on the family of displacement fields $\bpsi_i:=\boldsymbol{\phi}_i-\bId$ ($1\leq i \leq n$) with respect to the $\boldsymbol{L}^2(\Omega_0)$-inner product. We denote by $\lambda_1 \leq \ldots \leq \lambda_n$ the $n$ eigenvalues of the correlation matrix $\boldsymbol{C}:=(\langle \bpsi_i, \bpsi_j \rangle_{\boldsymbol{L}^2(\Omega_0)})_{1\leq i,j\leq n} \in \mathbb{R}^{n\times n}$ and by $\{\bzeta_i\}_{1\leq i \leq n} \subset \boldsymbol{W}^{1,\infty}(\Omega_0)$ the corresponding POD modes. For a given number $r\in \mathbb{N}^*$ of selected POD modes, we introduce the mapping
\begin{align}  
        \boldsymbol{\varphi}_r: \R^{r} \ni \alpha:= (\alpha_j)_{1\leq j \leq r}   \longmapsto  \boldsymbol{\varphi}_r(\alpha) := \bId +  \sum_{j=1}^r \alpha_j \boldsymbol{\zeta}_j \in  \boldsymbol{W}^{1,\infty}(\Omega_0).
\end{align}   
For all $1\leq i \leq n$, we define the vector $\alpha^i=(\alpha_j^i)_{1\leq j \leq r} \in \R^r$ such that 
$$
\forall 1\leq j \leq r, \; \alpha_j^i:= \langle \bphi_i - \bId, \bzeta_j \rangle_{\boldsymbol{L}^2(\Omega_0)} =  \langle \bpsi_i, \bzeta_j \rangle_{\boldsymbol{L}^2(\Omega_0)},
$$
so that $\sum_{j=1}^r \alpha_j^i \bzeta_j$ is the $\boldsymbol{L}^2(\Omega_0)$-orthogonal projection of $\bpsi_i := \bphi_i - \bId$ onto ${\rm Span}\{\bzeta_1, \ldots, \bzeta_r\}$. Each morphing $\bphi_i$ can then be approximated as 
\begin{align} \label{morphing approximation}
\boldsymbol{\phi}_i \approx \boldsymbol{\varphi}_r(\alpha^i) := \bId +  \sum_{j=1}^r \alpha^i_j \boldsymbol{\zeta}_j = \bId + \sum_{j=1}^r \langle \boldsymbol{\phi}_i-\bId,\boldsymbol{\zeta}_j\rangle_{\boldsymbol{L}^2(\Omega_0)} \boldsymbol{\zeta}_j,
\end{align}
and each geometry $\Omega_i=\boldsymbol{\phi}_i(\Omega_0)$ can be identified with the vector $\alpha^i \in \R^r$.

In general, one can choose the value $1\leq r\leq n$ in one of the following two ways:
\begin{enumerate}
    \item[(i)] For a prescribed tolerance $\delta^{\mathrm{POD}}>0$, choose $r$ as the smallest positive integer such that
\begin{align} \label{eq:critere_POD}
1-\frac{ \displaystyle \sum_{j=1}^{r} \lambda_j }{\displaystyle  \sum_{j=1}^n \lambda_j} \leq \delta^{\mathrm{POD}}.
\end{align}    
This criterion aims at controlling the accuracy of the reconstruction of the morphings using the first $r$ POD modes in $\boldsymbol{L}^2(\Omega_0)$-norm.
\item[(ii)] For a prescribed tolerance $\delta^{\mathrm{geo}}>0$, we choose $r$ as the smallest integer such that
\begin{align}\label{Hausdorff criterion}
    \max_{1\leq i\leq n} \Delta_2(\boldsymbol{\varphi}_r(\alpha^i), \Omega_0,\Omega_i) < \delta^{\mathrm{geo}},
\end{align}
where the error measure $\Delta_2$ is defined in~\eqref{eq:Delta2}. This second criterion allows one 
to control more directly the geometrical error between $\boldsymbol{\varphi}_r(\alpha^i)(\Omega_0)$ and $\Omega_i$. \rev{Here, $\delta^{\mathrm{geo}}$ should be taken of the order of the size of the boundary elements of the morphed reference meshes from the training set.}
\end{enumerate}
\rev{In what follows, we focus on the second criterion~\eqref{Hausdorff criterion}. Moreover,}
we always ensure that $r$ is large enough so that, for all $1\leq i \leq n$, the POD approximation $\bvarphi_r(\balpha^i)$ is indeed a 
diffeomorphism from $\Omega_0$ onto $\bvarphi_r(\balpha^i)(\Omega_0)$.

\begin{remark}[Geometry vs.~vector $\alpha$]
The correspondence between a geometry $\Omega_i$ and a vector $\balpha^i \in \mathbb{R}^r$ is not necessarily unique: for a geometry $\Omega_i$ such that $\Delta_2(\boldsymbol{\varphi}_r(\alpha^i), \Omega_0,\Omega_i) < \delta^{\mathrm{geo}}$, we may find another vector $\Bar{\alpha}$ such that $\Delta_2(\boldsymbol{\varphi}_r(\Bar{\alpha}),\Omega_0,\Omega_i) < \delta^{\mathrm{geo}}$ as well, without having $\boldsymbol{\phi}_i=\boldsymbol{\varphi}_r(\alpha^i)$ up to the error on the POD. In other terms, we can find multiple morphings mapping $\Omega_0$ onto $\Omega_i$ in the affine space $\bId + {\rm Span}\{ \bzeta_1, \ldots, \bzeta_r\}$.
\end{remark}

\subsection{Online phase}\label{reduced problem}
Given a new geometry $\widetilde{\Omega} \subset \mathbb{R}^d$ that is a domain of $\mathbb{R}^d$, we search for a morphing $ \widetilde{\boldsymbol{\phi}} \in \boldsymbol{\mathcal T}_{\Omega_0}$ in the affine space $\bId + {\rm Span}\{\boldsymbol{\zeta}_{i} \}_{ 1\leq i \leq r}, $ so that $\widetilde{\boldsymbol{\phi}}(\Omega_0)$ is close to $\widetilde{\Omega}$ with respect to the criterion \rev{$\Delta_2$ defined in~\eqref{eq:Delta2}}. More precisely, the morphing $\widetilde{\boldsymbol{\phi}}$ will be computed as $\widetilde{\boldsymbol{\phi}}=\boldsymbol{\varphi}_r(\widetilde{\alpha})$ for some $\widetilde{\balpha} \in \R^{r}$.

\rev{W}e introduce the functional
\begin{align} \label{I}
          \mathcal{B}: \R^{r} \ni \alpha \longmapsto \mathcal{B}(\alpha) := (\mathcal{B}_j(\alpha))_{1\leq j \leq r} \in  \R^{r},   
\end{align}
such that, for all $1\leq j \leq r$ and all $\balpha \in \mathbb{R}^r$,
\begin{align}    
    \mathcal{B}_j(\alpha) := {}& b^p_{\boldsymbol{\varphi}_r(\alpha)}(\boldsymbol{\zeta}_j\circ \boldsymbol{\varphi}_r^{-1}(\alpha)) +  b^l_{\boldsymbol{\varphi}_r(\alpha)}(\boldsymbol{\zeta}_j\circ \boldsymbol{\varphi}_r^{-1}(\alpha)) \nonumber \\
    = {}& \beta_1  \sum_{k=1}^{N_{p}}  \displaystyle \int_{N(\boldsymbol{\varphi}_r(\alpha)(\rbP_{k}^0))}(\rbP_{k}- \boldsymbol{\varphi}_r(\alpha)(\rbP_{k}^0))\cdot \boldsymbol{\zeta}_j\circ \boldsymbol{\varphi}_r^{-1}(\alpha)(\boldsymbol{x}) \,ds \nonumber\\
    & + \beta_2 \int_{\boldsymbol{\varphi}_r(\alpha)(\partial\Omega_0)}\left(\boldsymbol{D}^{\partial\widetilde{\Omega}}_{\boldsymbol{\varphi}_r(\alpha)} \cdot \bn_{\boldsymbol{\varphi}_r(\alpha)} \right) \left(\boldsymbol{\zeta}_j\circ \boldsymbol{\varphi}_r^{-1}(\alpha) \cdot \boldsymbol{n}_{\boldsymbol{\varphi}_r(\alpha)} \right)(\boldsymbol{x})\, ds, \label{I explicit}
\end{align}
where $b^p_{\boldsymbol{\varphi}_r(\alpha)}$ and $b^l_{\boldsymbol{\varphi}_r(\alpha)}$ are defined in \eqref{b1} and  \eqref{b2}, respectively. 

The online procedure we propose to compute $\widetilde{\balpha}$ is an iterative procedure which we now describe. 

\subsubsection{Initialization using the vector distance function} \label{Initialization}
In line with the high-fidelity construction of the morphing, we could initialize the online iterative procedure so that $\boldsymbol{\phi}^{(0)}=\boldsymbol{\varphi}_r(\widetilde{\alpha}^{(0)})=\bId$. This corresponds to $\widetilde{\alpha}^{(0)}=0_{\R^r}$. However, this approach was observed to yield results which were not satisfactory, neither from an accuracy nor from an efficiency point of view. The initialization procedure we propose here remedies these shortcomings. It builds on the observation that if the new geometry $\widetilde{\Omega}$ is close to one of the geometries belonging to the training set, one should be able to use this information to initialize the algorithm with a solution near an optimal solution. The idea is to rely on the construction of an appropriate  regression model. More precisely, suppose that we have (or we can determine) a quantity that defines each geometry in the dataset. We can then build a regression metamodel that, for a given geometry $\widetilde{\Omega}$, takes as input that quantity and produces as output the morphing coordinates $\widetilde{\balpha} \in \mathbb{R}^r$.

We propose to use the vector distance function defined in \eqref{vectDistance} by proceeding as follows:
\begin{enumerate}
    \item In the offline phase:
    \begin{enumerate}
        \item For each geometry $\Omega_i$, calculate the function $\boldsymbol{D}^{\partial\Omega_i}_{\bId} $ such that
        \begin{align} \boldsymbol{D}^{\partial\Omega_i}_{\bId} :
        \partial\Omega_0 \ni \boldsymbol{x} \longmapsto  \boldsymbol{D}^{\partial\Omega_i}_{\bId}(\boldsymbol{x}):= \sum_{k=1}^{N_l}(\boldsymbol{\Pi}_{L^i_{k}}(\boldsymbol{x})-\boldsymbol{x})1_{L_{k}^0}(\boldsymbol{x}) \in \R^2,
        \end{align}
    where $\{L^i_{k}\}_{1\leq k \leq N_l}$ is the set of curves partitioning the boundary of $\Omega_i$. 
    \item Compute the POD of the family of functions $\{\boldsymbol{D}^{\partial\Omega_i}_{\bId}\}_{1 \leq i \leq n}$ in $\boldsymbol{L}^2(\partial \Omega_0)$ and denote by $(\btheta_j)_{1\leq j \leq n}$ the corresponding set of POD modes. Fix some $q\in \mathbb{N}^*$ and for all $1\leq i \leq n$, compute $d^i = \left( d^i_j \right)_{1\leq j \leq q}\in \mathbb{R}^{q}$ as
    $$
    \forall 1\leq j \leq q, \quad d^i_j = \left\langle \boldsymbol{D}^{\partial\Omega_i}_{\bId}, \btheta_j \right\rangle_{\boldsymbol{L}^2(\partial \Omega_0)}.
    $$
    \item Train a regression model that takes as input the vector $d^i \in \R^q$ and as output the generalized coordinates $\alpha^i \in \R^r$ of the morphing $\boldsymbol{\phi}_i$.  Denote by $\mathcal R: \mathbb{R}^q \to \mathbb{R}^r$ the corresponding regression model. 
    \end{enumerate}   
    \item In the online phase:
    \begin{enumerate}
        \item For a new geometry $\widetilde{\Omega}$, calculate the vector distance $\boldsymbol{D}^{\partial\widetilde{\Omega}}_{\bId}$, then project on the low-dimensional representation to obtain the corresponding vector $\widetilde{d} = \left(\widetilde{d}_j\right)_{1\leq j \leq q} \in \R^q$ such that 
        $$
        \forall 1\leq j \leq q, \quad \widetilde{d}_j = \left\langle \boldsymbol{D}^{\partial\widetilde{\Omega}}_{\bId}, \btheta_j \right\rangle_{\boldsymbol{L}^2(\partial \Omega_0)}.
        $$
        \item Define $\widetilde{\alpha}^{(0)} = \mathcal R\left(\widetilde{d} \right)$.
    \end{enumerate}
\end{enumerate}
The regression model used in this work is the Gaussian process regression (GPR)~\cite{rasmussen2006gaussian}.

We make the following two observations. First, for a geometry $\widetilde{\Omega}$, taking $\boldsymbol{D}^{\partial\Tilde{\Omega}}_{\bId}$ as input does not mean that the output of the metamodel will map each point $\boldsymbol{x}\in \partial \Omega_0 $ to its projection onto $\partial \Tilde{\Omega}$. Here, the vector distance is used only to measure in some way the deviation of each $\partial \Tilde{\Omega}$ from $\partial \Omega_0$. 
Second, we emphasize that the above approach is not devised as a means of directly predicting the morphing coefficients without using the iterative algorithm (to be presented in the next section). Indeed, this would lead to two main drawbacks. Firstly, the output of the metamodel does not generally precisely satisfy $\boldsymbol{\varphi}_r(\widetilde{\alpha})(\Omega_0)=\widetilde{\Omega}$. Secondly, it is possible for two different geometries to yield identical inputs $\widetilde{d}$, resulting in the same coefficients $\widetilde{\alpha}$ from the regression model. In conclusion, the above approach merely serves as a means of predicting an initialization $\widetilde{\alpha}^{(0)}$ for the online optimization algorithm, that is (hopefully) sufficiently close to the optimal solution. Thus, even if two distinct geometries share the same initialization during the online phase, they will not produce identical morphings after solving the optimization problem.

\subsubsection{Online iterative algorithm}\label{Optimization algorithm}

To find the final reduced coordinates $\widetilde{\alpha}\in \mathbb{R}^r$ for the new geometry $\Tilde{\Omega}$, we use an iterative algorithm which consists in updating at each iteration $m$ the vector $\widetilde{\alpha}^{(m)}\in \R^r$ as
\begin{align} \label{iterative classic}
    \widetilde{\alpha}^{(m+1)}=\widetilde{\alpha}^{(m)} - \gamma^{(m)} \mathcal{B}(\widetilde{\alpha}^{(m)}),
\end{align}
for some $\gamma^{(m)}>0$ starting from the initial value $\widetilde{\balpha}^{(0)}\in \mathbb{R}^r$ obtained from the initialization procedure described in the previous section. In practice, the value $\gamma^{(m)}$ is always chosen to be equal to some constant value $\gamma>0$ for all iterations $m \in \N$. 

\begin{remark}[Elasticty-based update]
   Another possibility could have been to use an inner product associated with the elasticity bilinear $a_{\boldsymbol{\varphi}_r(\alpha)}$ defined in \eqref{inner product}. We have $a_{\boldsymbol{\varphi}_r(\alpha)}(\boldsymbol{\varphi}_r(\bu),\boldsymbol{\varphi}_r(\bv))=\langle M(\alpha) \bu, \bv \rangle_{\boldsymbol{L}^2(\mathbb{R}^r)}$ with the stiffness matrix $M(\alpha):= \left(a_{\boldsymbol{\varphi}_r(\alpha)}(\bzeta_i, \bzeta_j) \right)_{1\leq i,j \leq r} \in \mathbb{R}^{r\times r}$. The iterative algorithm then becomes   
\begin{equation} \label{iterative elas}
        \alpha^{(m+1)}=\alpha^{(m)} - \gamma^{(m)} M^{-1}(\alpha^{(m)}) \mathcal{B}(\alpha^{(m)}).
\end{equation}
While this inner product actually introduces physical information to deform the mesh, it requires determining, at each iteration, the stiffness matrix $M(\alpha^{(m)})$, which boils down to calculating $\frac{r(r+1)}{2}$ volume integrals. This can be quite costly. The advantage of the approach relying on \eqref{iterative classic} is that we only need to compute surface integrals, instead of computing volume integrals and solving a linear elasticity system at each iteration. Thus, the computational efficiency is much higher than with the approach relying on \eqref{iterative elas}. 
\end{remark}

\subsubsection{Stopping criterion and out-of-distribution geometry}
\rev{Recall the geometrical error tolerance $\delta^{\mathrm{geo}}>0$ introduced above}. Then, for every new geometry $\widetilde{\Omega}$ considered in the online phase, the iterative procedure described in Section \ref{Optimization algorithm} is carried out until the following stopping criterion is met: 
$$\Delta_2(\boldsymbol{\varphi}_r(\Tilde{\alpha}^{(m)}), \Omega_0,\widetilde{\Omega})<\delta^{\mathrm{geo}}.$$
We also choose a value $M_{max}\in \mathbb{N}^*$ corresponding to a maximum number of iterations and an a priori error threshold $\delta_{\nabla}>0$.
One practical way to choose the value of $\delta_{\nabla}$ is to define it as $\displaystyle \delta_{\nabla} := \frac{1}{n}\sum_{i=1}^n \|\mathcal{B}(\boldsymbol{\alpha}^i)\|$. 
If the above stopping criterion is not reached after $M_{max}$ iterations, we evaluate $\eta:=\|\mathcal{B}(\Tilde{\alpha}^{(M_{max})})\|$ 
and proceed as follows:    
\begin{enumerate}
    \item If $\eta  \geq \delta_{\nabla}$, the iterative algorithm did not reach convergence. Depending on the required precision and the cost per iteration, we may allow here to increase the number of iterations $M_{max}$.
    \item On the other hand, if $\eta <\delta_{\nabla}$, this means that the target domain $\widetilde{\Omega}$ cannot be well-approximated in the form $\bphi(\Omega_0)$ for some morphism $\bphi$ computed as an element of $\bId + {\rm Span}\{\boldsymbol{\zeta}_i, \; 1\leq i \leq r\}$. The geometry $\widetilde{\Omega}$ is then classified as being out-of-distribution (ood) and one of the following two steps is performed:
\begin{enumerate}
    \item Either increase the number of modes $r$; this allows for more flexibility in finding a reduced morphing $\widetilde{\bphi}$ such that $\widetilde{\bphi}(\Omega_0)$ is close to $\widetilde{\Omega}$.
    \item Or use the high-fidelity routine to compute a high-fidelity map $\widetilde{\boldsymbol{\phi}}: \Omega_0 \to \widetilde{\Omega}$, possibly initialized with $\displaystyle \widetilde{\boldsymbol{\phi}}^{(0)} = \boldsymbol{\varphi}_r(\Tilde{\alpha}^{(M_{max})})$, update $r := r+1$ and define $\bzeta_{r+1}:= \widetilde{\bphi}$.
\end{enumerate}
\end{enumerate}

\subsection{Overall workflow} \label{Overall workflow}
The following tables summarize our offline and online workflows:\\
\noindent\fbox{%
\begin{algorithm}[H]
\SetAlgoLined
\caption{Offline workflow}
\KwData{Training set of domains $\{\Omega_i\}_{1\leq i \leq n}$}
\KwIn{Reference domain $\Omega_0$, tolerance $\delta^{\mathrm{geo}}>0$, step size $\gamma>0$ }
\For{$i \gets 1$ \KwTo $n$}{
    Calculate $\boldsymbol{D}^{\partial\Omega_i}_{\bId}$\;
    Initialize $\boldsymbol{\phi}_i^{(0)}=\bId$\;
    $m \gets 0$\;
    \Repeat{$\Delta_2\left( \boldsymbol{\phi}_i^{(m)}, \Omega_0, \Omega_i\right)< \epsilon$}{
        Solve for $\bu_{\Omega_i}^{(m)}$\;
        Update $\boldsymbol{\phi}_i^{(m+1)} \gets \boldsymbol{\phi}_i^{(m)} + \gamma \bu_{\Omega_i}^{(m)}\circ \boldsymbol{\phi}_i^{(m)} $ \;
        Calculate $\boldsymbol{D}^{\partial\Omega_i}_{\boldsymbol{\phi}_i^{(m+1)}}$\;
        $m \gets m + 1$\;
    }
    $\boldsymbol{\phi}_i \gets \boldsymbol{\phi}_i^{(m)}$\;
}
\textbf{POD:} $\{\boldsymbol{\phi}_i\}_{1\leq i \leq n} \rightarrow \{\boldsymbol{\zeta}_j\}_{1\leq j \leq r}, \{\alpha^i\}_{1\leq i \leq n}$ with tolerance $\delta^{\mathrm{geo}}$\;
\textbf{SVD:} $\{\boldsymbol{D}^{\partial\Omega_i}_{\bId}\}_{1\leq i \leq n} \rightarrow \{d^i\}_{1\leq i \leq n}$\;
\textbf{Train GPR :} $\{d^i\}_{1\leq i \leq n}, \{\alpha^i\}_{1\leq i \leq n} \rightarrow \mathcal R$\;
\textbf{Determine :} $\delta_{\nabla}$\;
\end{algorithm}
}
\\
\noindent\fbox{%
\begin{algorithm}[H]
\SetAlgoLined
\caption{Online Workflow}
\KwData{Reduced-order basis $\{\boldsymbol{\zeta}_j\}_{1\leq i \leq r}$, bounds: $\delta^{\mathrm{geo}}, \delta_{\nabla}$}
\KwIn{Reference domain $\Omega_0$, target domain $\widetilde{\Omega}$, maximum number of iterations $M_{\text{max}}$, step size $\gamma>0$}
\KwOut{Generalized coordinates $\widetilde{\alpha}$}
Calculate $\boldsymbol{D}^{\partial\Tilde{\Omega}}_{\text{Id}}$\;
Project $\boldsymbol{D}^{\partial\Tilde{\Omega}}_{\text{Id}}$ to obtain $\widetilde{d}$\;
Use GPR to obtain $\widetilde{\alpha}^{(0)}$\;
$m \gets 0$\;

\While{$m \leq M_{\text{max}}$}{
    $\widetilde{\alpha}^{(m+1)} \gets \widetilde{\alpha}^{(m)} - \gamma \mathcal{B}(\widetilde{\alpha}^{(m+1)})$\;
    Calculate $\boldsymbol{D}^{\partial\widetilde{\Omega}}_{\boldsymbol{\varphi}_r(\widetilde{\alpha}^{(m+1)})}$\;

\If{$\Delta_2(\boldsymbol{\varphi}_r(\Tilde{\alpha}^{(m)}), \Omega_0,\widetilde{\Omega})<\delta^{\mathrm{geo}}$ is true}{
    Terminate the loop\;
}
    $m \gets m + 1$\;
    \If{ $m=M_{\text{max}}$ \text{ and } $\Delta_2(\boldsymbol{\varphi}_r(\Tilde{\alpha}^{(M_{\text{max}})}), \Omega_0,\Tilde{\Omega})>\delta^{\mathrm{geo}}$}{
    \If{$ \|\mathcal{B}(\widetilde{\alpha}^{(M_{\text{max}})})\|_2 \geq \delta_{\nabla}$}{Increase $M_{\text{max}}$ \;}
    \Else{Increase $r$ or perform offline routine for $\widetilde{\Omega}$\;}
} 
}

\end{algorithm}
} 
\subsection{Complexity}\label{complexity}
The cost of one iteration in the offline phase comprises the assembly of the stiffness matrix associated with the bilinear form $a_{\boldsymbol{\phi}}$ defined in \eqref{inner product}, the computation of the matching term (the vector distance function \eqref{vectDistance}), the assembly of the right-hand-side vector corresponding to the linear form ($\Tilde{b}^p_{\boldsymbol{\phi}}$ and $\Tilde{b}^l_{\boldsymbol{\phi}}$ in \eqref{b1} and \eqref{b2}), and finally the resolution of the resulting sparse linear system to obtain the coordinates of the displacement field in the finite element basis.

The cost of one iteration in the online phase comprises computing the matching term (the vector distance function \eqref{vectDistance}), and the evaluation of $\mathcal{B}(\alpha)$ for $\balpha \in \mathbb{R}^r$, which corresponds to computing $r$ integrals. 

We denote by $\mathcal{N}$ the number of nodes of $\mathcal{M}_0$, the mesh of $\Omega_0$ that is used in the computation. The number of nodes on $\partial\Omega_0$ depends on the dimension of the problem, and for 2D elements, is of the order of $O(\sqrt{\mathcal{N}})$. We also denote by $p$ the number of nodes used to discretize the boundary of the target domain. 
\begin{table}[h]
\centering
\begin{tabular}{|c|c|c|} 
\hline
  & \textbf{Offline } & \textbf{Online } \\ 
\hline
\textbf{Matrix assembly} & $O(\mathcal{N}$) & - \\
\textbf{Matching term computation} & $O(\log(p)\sqrt{\mathcal{N}})$ & $O(\log(p)\sqrt{\mathcal{N}})$ \\
\textbf{Computation of the gradient} & $O(\mathcal{\sqrt{\mathcal{N}}}$) & $ O(r\mathcal{\sqrt{\mathcal{N}}})$ \\
\textbf{Linear system resolution (dense matrix)} & $ O(\mathcal{N}^3) $ & - \\
\hline
\end{tabular}
\caption{Cost of one iteration in the offline and online phases.}
\label{tab:complexity}
\end{table}

In Table~\ref{tab:complexity}, we report the complexity per iteration for the offline and online phases. The complexity of the computation of the matching term is shown for the vector distance function using the KD-tree algorithm to determine the closet point. The complexity is actually larger for the signed distance as we need to calculate also the sign for each node.

For the general case of dense matrices of size $\mathcal{N}$, the complexity of solving a linear system by a direct method is of order $O(\mathcal{N}^3)$. For sparse matrices such as the ones encountered here,  
the complexity depends on the algorithm used and the sparsity of the matrix. In our implementation, we used the LU decomposition for sparse matrices to solve the linear systems. Usually, this step is the most expensive one in the offline phase. 

The efficiency of the online phase results from the fact that we do not need to solve any linear system. Another important aspect which speeds up the computations in the online phase is the initialization step described in Section \ref{Initialization}. Because we initialize the iterative procedure close to the solution, the number of iterations needed to achieve convergence is significantly smaller than the number of iterations needed in the offline phase. Additional speed-up can be gained also from using parallel implementation to calculate the $r$ integrals in \eqref{I explicit}.

\subsection{Numerical results} \label{Numerical results online}
In this section, we present numerical results to illustrate the performance of the above offline/online algorithm.
\subsubsection{Tensile2D dataset} \label{Tensile2D online}
\textbf{Offline phase:} We adopt the same notation as in Section \ref{Tensile2D offline}. The size of the training set is $n=500$. For all $1\leq i \leq n$, we define $\Omega_i=[-1,1]^2\ \backslash B(R_i)$, with $R_i=0.2+0.6\times\frac{i}{n}$. We define the reference domain $\Omega_0:=\Omega_{250}$. This is the same reference domain as the one used in Section \ref{Tensile2D offline}. We start by calculating the morphings $\{\boldsymbol{\phi}_i\}_{1\leq i \leq n}$ using the vector distance algorithm. We recall that the parameterization is not used in the construction of the morphings. 

\begin{figure}[ht] 
\begin{center}
\includegraphics[scale=0.55]{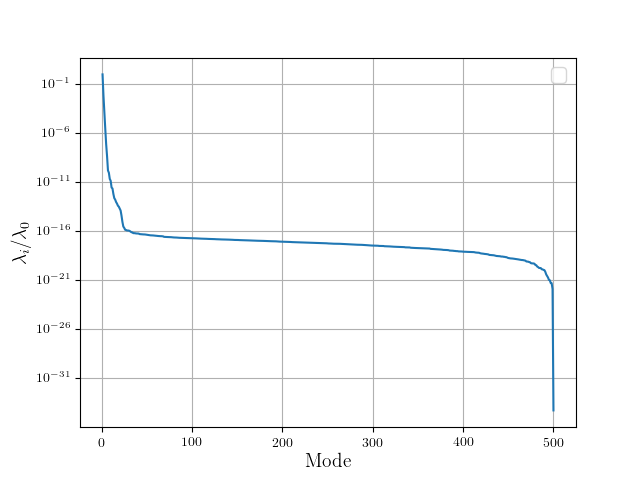}
\caption{\label{eignevalues Tensile2D} \rev{Decay of the eigenvalues of the correlation matrix for the Tensile2D dataset.}}
\end{center}
\end{figure}

\begin{figure}[ht]
    \centering
        \includegraphics[scale=0.15]{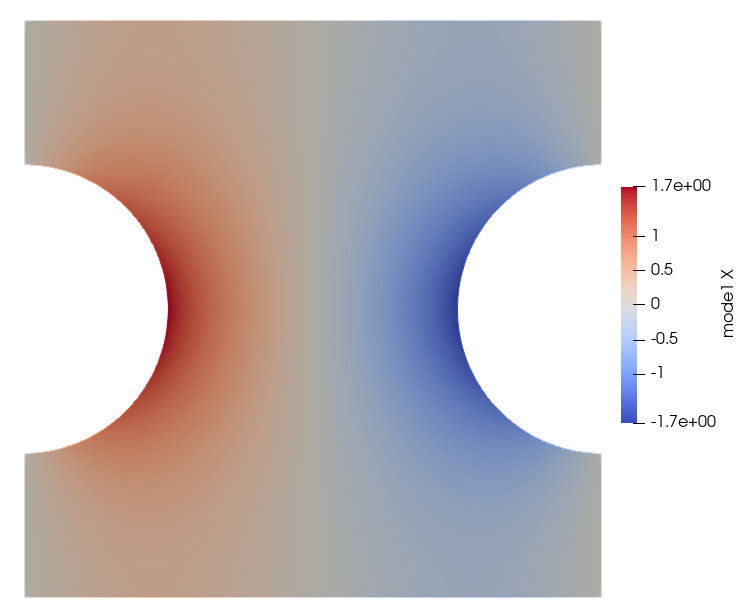} \quad
        \includegraphics[scale=0.15]{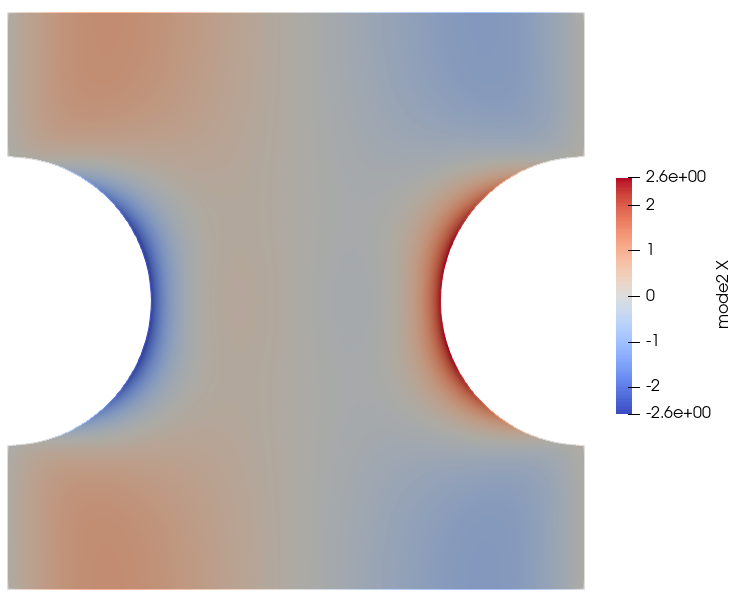} \quad
        \includegraphics[scale=0.15]{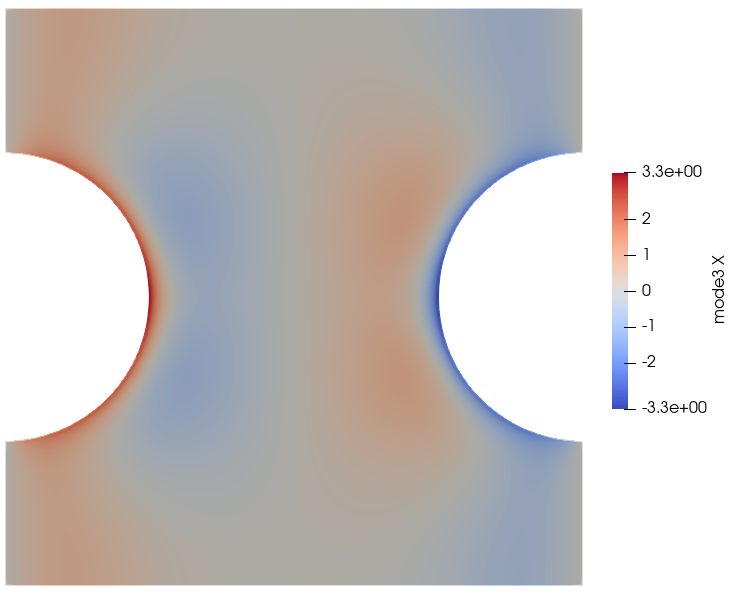} \quad
        \includegraphics[scale=0.15]{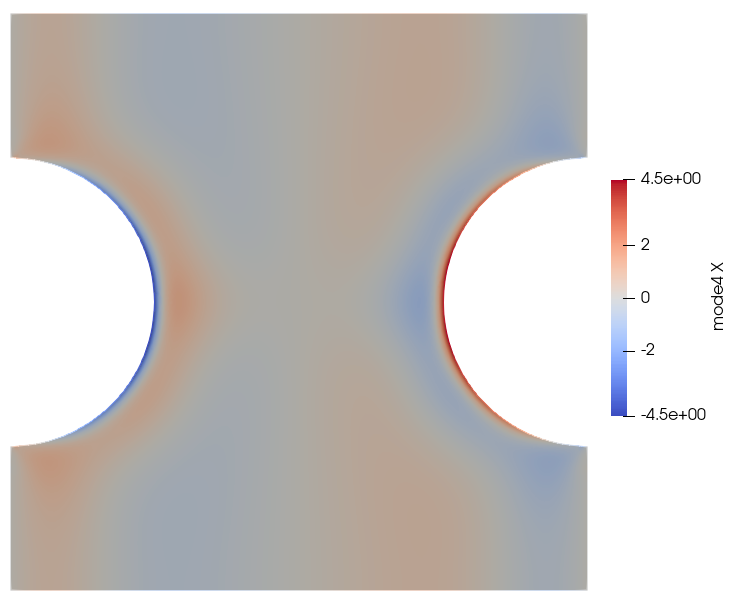} 
        \vspace{1em}

        \includegraphics[scale=0.15]{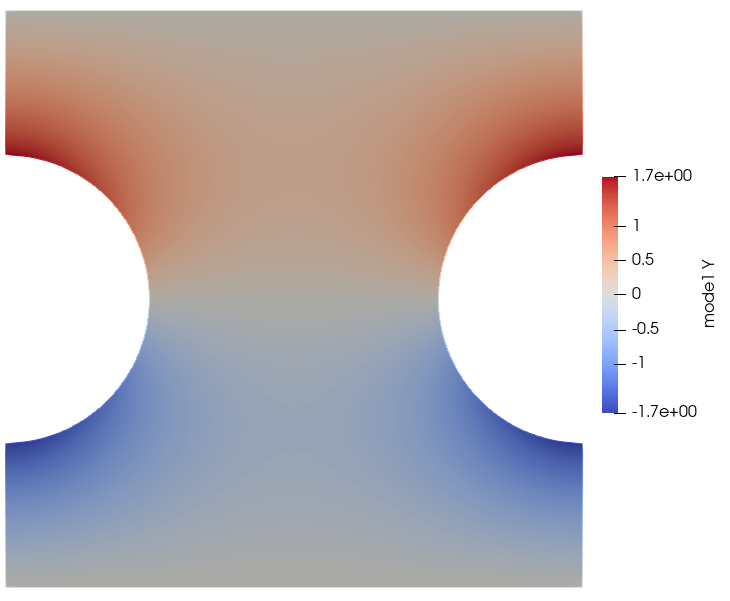} \quad
        \includegraphics[scale=0.15]{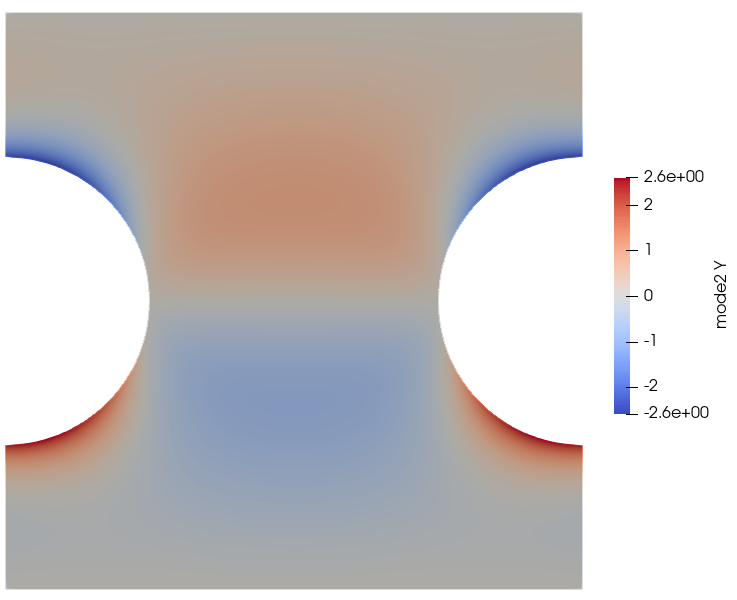} \quad
        \includegraphics[scale=0.15]{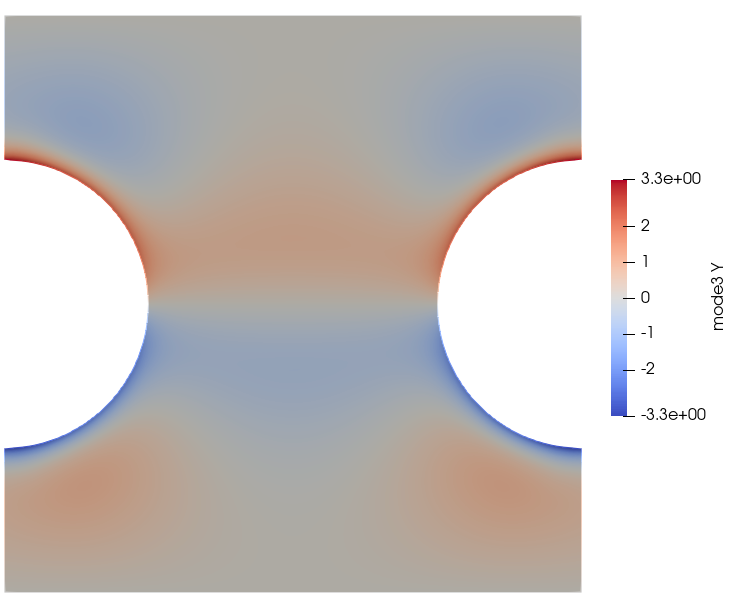} \quad
        \includegraphics[scale=0.15]{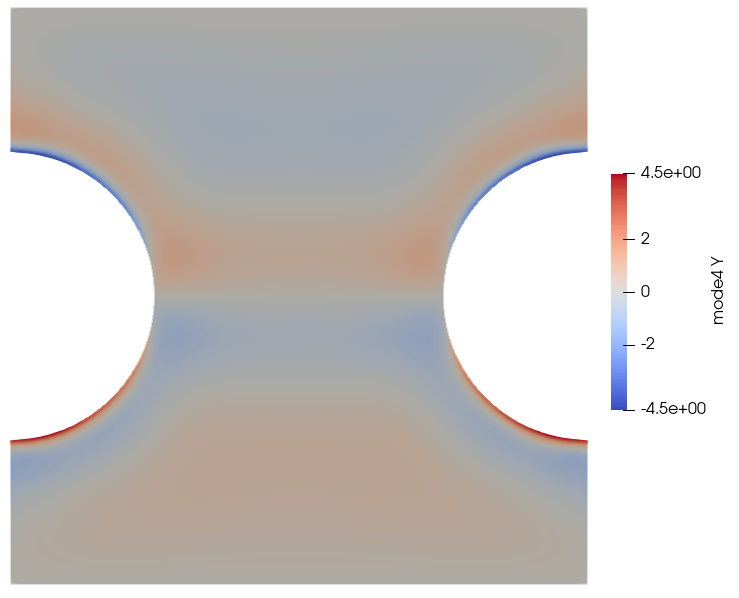} \quad
    \caption{\rev{First four POD modes for the Tensile2D dataset. First (resp., second) row: $x$- (resp., $y$-) component. From left to right: modes 1, 2, 3, 4.}}
    \label{fig:modes displacement fields}
\end{figure}

Next, we choose $\delta^{\mathrm{geo}}:=5\times10^{-4}$ \rev{(the size of the boundary elements of the morphed reference meshes is in the range $0.01$ to $0.05$)}. Employing POD with the criterion \eqref{Hausdorff criterion} leads to $r=5$ modes. \rev{In Figure \ref{eignevalues Tensile2D}, we report the decay of the eigenvalues of the correlation matrix corresponding to the displacement fields. We observe a swift decay after a few modes, but two modes are not sufficient to collect most of the energy. We show in Figure \ref{fig:modes displacement fields} the first four POD modes. As expected, the POD modes are essentially concentrated near the curved boundaries.} Finally, we train a Gaussian process regression (GPR) that takes as input the SVD coordinates of the vector distance function with $q=5$, and gives as output the generalized morphing coordinates to initialize the online optimization problem.

\textbf{Online phase:} The testing set is composed of $n_{\rm test}:=200$ geometries $\{\widetilde{\Omega}_j\}_{1\leq j\leq n_{\rm test}}$ which have the same form as the training set, that is, $\widetilde{\Omega}_j=[-1,1]^2\ \backslash B(\Tilde{R}_j)$ for some (supposedly unknown) radius $ \Tilde{R}_j$. All the radii $\Tilde{R}_j$ are different from those of the training set. For each $\widetilde{\Omega}_j$, we use the vector distance function in the regression model to predict the initial iteration to the online optimization problem. In Table \ref{tab:Tensilde2D_online_Delta2}, we report the quantities 
\begin{align*}
  &\Delta^{\mathrm{geo}}_{\rm avg}(r,N)\displaystyle := \frac{1}{N}\sum_{i=1}^N \Delta_2(\boldsymbol{\varphi}_r(\alpha^i)(\Omega_0),\widetilde{\Omega}_i),\\
  &\Delta_{\rm max}^{\rm geo}(r,N) := \max_{1\leq i\leq N} \{\Delta_2(\boldsymbol{\varphi}_r(\alpha^i)(\Omega_0),\widetilde{\Omega}_i)\}, \\
  &\Delta_{\rm min}^{\rm geo}(r,N) := \min_{1\leq i\leq N} \{\Delta_2(\boldsymbol{\varphi}_r(\alpha^i)(\Omega_0),\widetilde{\Omega}_i)\},
\end{align*}
for $r=5$ and $N= n_{\rm test}$. As observed in Table \ref{tab:Tensilde2D_online_Delta2}, for all the samples in the test set, the initialized solution here is an optimal one that satisfies the stopping criterion, so the online iterative procedure is not used. Thus, the cost of each morphing calculation is only one evaluation of the vector distance function and one evaluation of the GPR which drastically cuts down the cost of morphing computation. In Table \ref{tab:Tensilde2D}, we report the ratio of the average (resp., maximum) time needed to compute the high-fidelity morphing (offline) over the average (resp., maximum) time needed to compute the reduced-order morphing (online) using our implementations. We observe that the reduced-order model we propose is about 270 times faster than the high-fidelity one. The steps required to construct the online phase model are morphing computation, morphing POD, vector distance function POD and GPR training. The time required to construct the online phase model is dominated by the morphing computation.
\begin{table}[h]
\centering
\begin{tabular}{|c|c|c|} 
\hline
$\Delta^{\mathrm{geo}}_{\rm avg}(r,n_{\rm test})$ & $\Delta^{\mathrm{geo}}_{\rm max}(r,n_{\rm test})$& $\Delta^{\mathrm{geo}}_{\rm min}(r,n_{\rm test})$ \\
\hline
 $1.5\times 10^{-4}$   & $3.6 \times 10^{-4}$  &   $5.2\times 10^{-7}$   \\
\hline
\end{tabular}
\caption{Average, maximum, and minimum values of the criterion $\Delta_2$ for all the samples from the dataset in the online phase after the initialization.}
\label{tab:Tensilde2D_online_Delta2}
\end{table}

\begin{table}[h]
\centering
\begin{tabular}{|c|c|} 
\hline
\textbf{Ratio of average time (offline/online)} & 267.1  \\
\hline
\textbf{Ratio of maximum time (offline/online)} & 269.6  \\
\hline
\end{tabular}
\caption{Ratio of average and maximum time to compute the morphing in the offline and online phases for the Tensile2D dataset.}
\label{tab:Tensilde2D}
\end{table}

\subsubsection{AirfRANS} \label{AirfRANS online}
\textbf{Offline phase.} For this test, we use all the 1000 airfoils in the AirfRANS dataset. The number of samples in the training set is $n:=800$. We use the same reference geometry, mesh and physical parameters as in Section \ref{AirfRANS offline}. After calculating each morphing, we apply POD on the displacement fields to obtain the principal modes of the displacements. \rev{The eigenvalues of the correlation matrix of the displacement field are plotted in Figure~\ref{eignevalues AirfRANS}. The decrease is not as swift as for the Tensile2D dataset. In particular, we observe that 
more than a couple of modes are necessary to capture most of the energy (typically, about 50).}
Finally, we train the GPR to use it to predict an initial iteration for the samples in the testing set.

\begin{figure}[ht] 
\begin{center}
\includegraphics[scale=0.55]{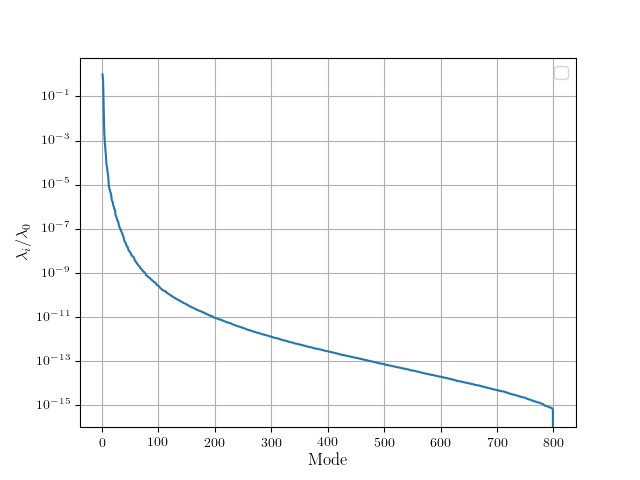}
\caption{\label{eignevalues AirfRANS} \rev{Decay of the eigenvalues of the correlation matrix for the AirfRANS dataset.}}
\end{center}
\end{figure}

\textbf{Online phase.} The testing set is composed of the remaining $n_{\rm test}:=200$ samples.
Here, the initialization of the morphing is not sufficient to satisfy our criterion on the error, so that we also use the online optimization strategy. 
\begin{figure}[ht] 
\begin{center}
\includegraphics[scale=0.5]{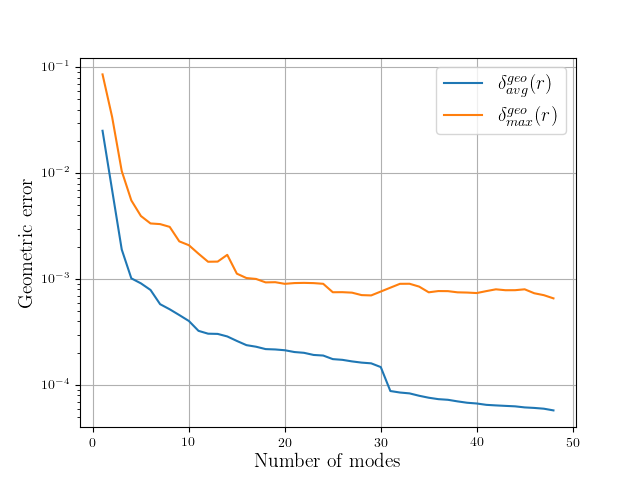}
\caption{\label{Geometric_error_func_modes} AirfRANS dataset: Evolution of the geometrical errors $\Delta^{\mathrm{geo}}_{\rm avg}(r,n)$ and $\Delta_{\rm max}^{\rm geo}(r,n) $ (in logarithmic scale) for the training set as a function of the number of modes $r$. The errors are not equal to zero owing to the POD truncation error.}
\end{center}
\end{figure} 

In Figure \ref{Geometric_error_func_modes}, we report both the average and maximum geometrical errors in the training set as a function of the number of modes. As expected, both errors tend to zero as we add more modes for morphing reconstruction. Note, however, that the convergence process is not monotone. This is due to the fact that additional modes actually can have the effect of better approaching the morphing field $\boldsymbol{\phi}_i$, and not necessarily minimizing the geometrical error. Obviously, taking all the modes produces zero error between $\boldsymbol{\phi}_i$ and $\boldsymbol{\varphi}_r(\alpha^i)$, and, as a result, zero geometrical error.

We test the morphing strategy for $r \in \{12,16,20,24,28,32,48\}$. For each value, we re-initialize $\Tilde{\alpha}^{(0)}$ in $\R^r$ for each sample in the testing set and perform the optimization in $\R^r$. For all the values of $r$, we fix the same geometric tolerance $\delta^{\mathrm{geo}}=1.5\times 10^{-4}$ \rev{(which is consistent with the values reported in Figure~\ref{Geometric_error_func_modes} for $r=32$ modes). Notice also that the size of the boundary elements of the morphed reference meshes is of the same order.}. In Figure \ref{converged_samples}, we show the number of samples for which convergence is achieved, i.e., satisfying $\Delta_2(\boldsymbol{\varphi}_r(\Tilde{\alpha}^{(m)})(\Omega_0),\Tilde{\Omega})<\delta^{\mathrm{geo}}$, as a function of the number of iterations. Here, zero iteration means that the initialized solution is sufficiently close so that further optimization is not needed.  \rev{In the left panel of Figure \ref{fig:converged_samples_48modes}, we report the number of samples that converged at the first iteration, whereas, in the right panel, we present an histogram of the number of samples that converged at each iteration for $r=48$.}

\begin{figure}[ht] 
\begin{center}
\includegraphics[scale=0.34]{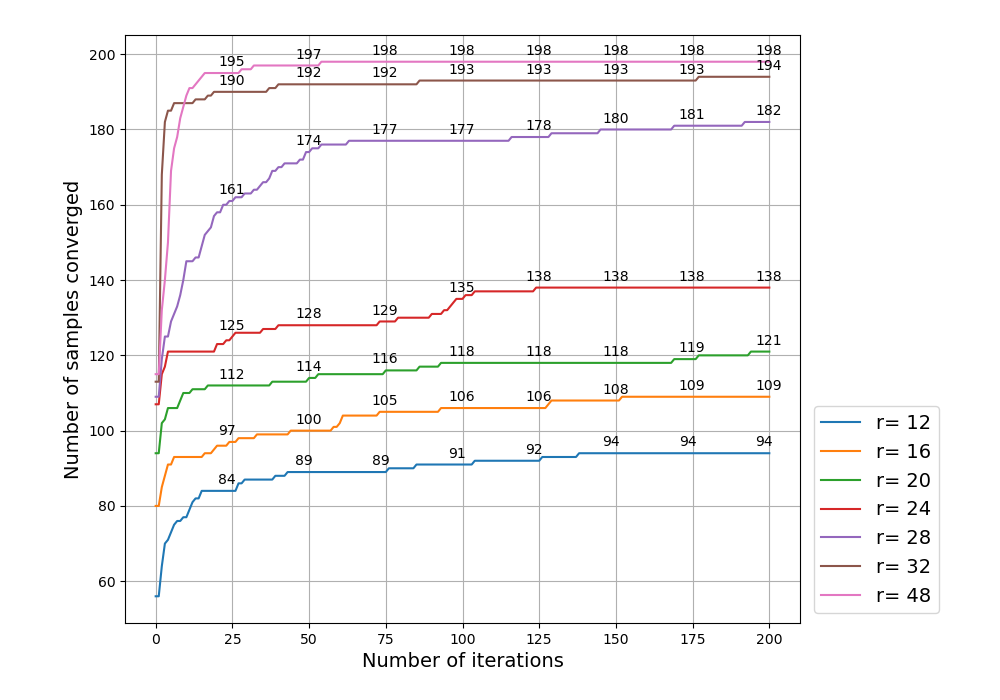}
\caption{\label{converged_samples} AirfRANS dataset: Number of converged samples as a function of the number of iterations for different values of $r$.}
\end{center}
\end{figure}

\begin{figure}[ht] 
     \centering
     \begin{subfigure}[b]{0.449\textwidth}
         \centering
        \includegraphics[scale=0.34]{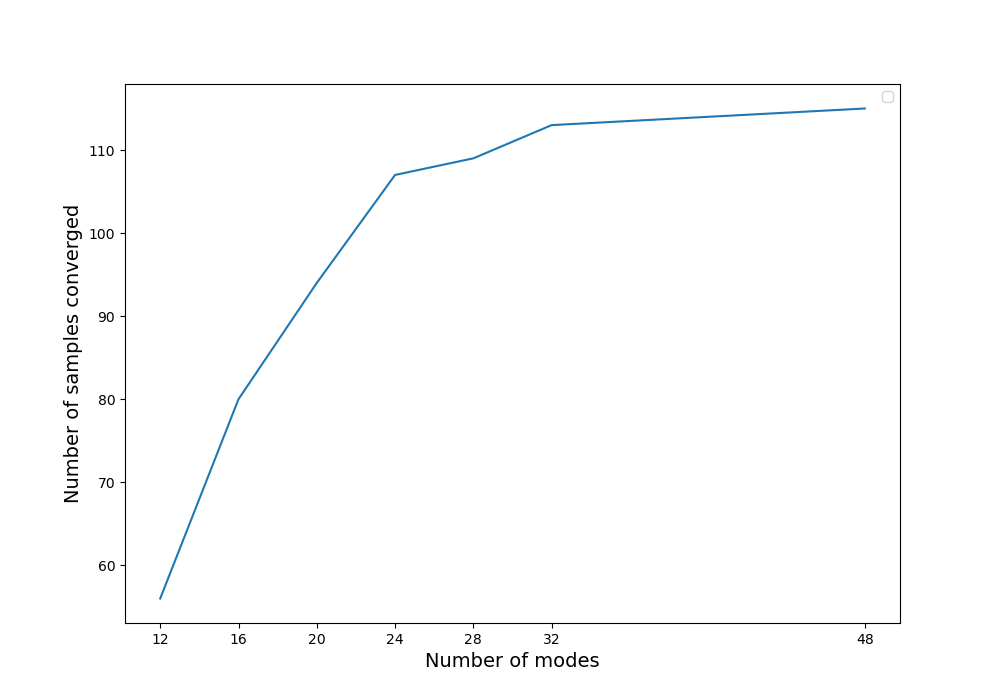}
        \caption{\rev{Convergence in one iteration.}}
     \end{subfigure}
     \hfill
     \begin{subfigure}[b]{0.449\textwidth}
         \centering
\includegraphics[scale=0.34]{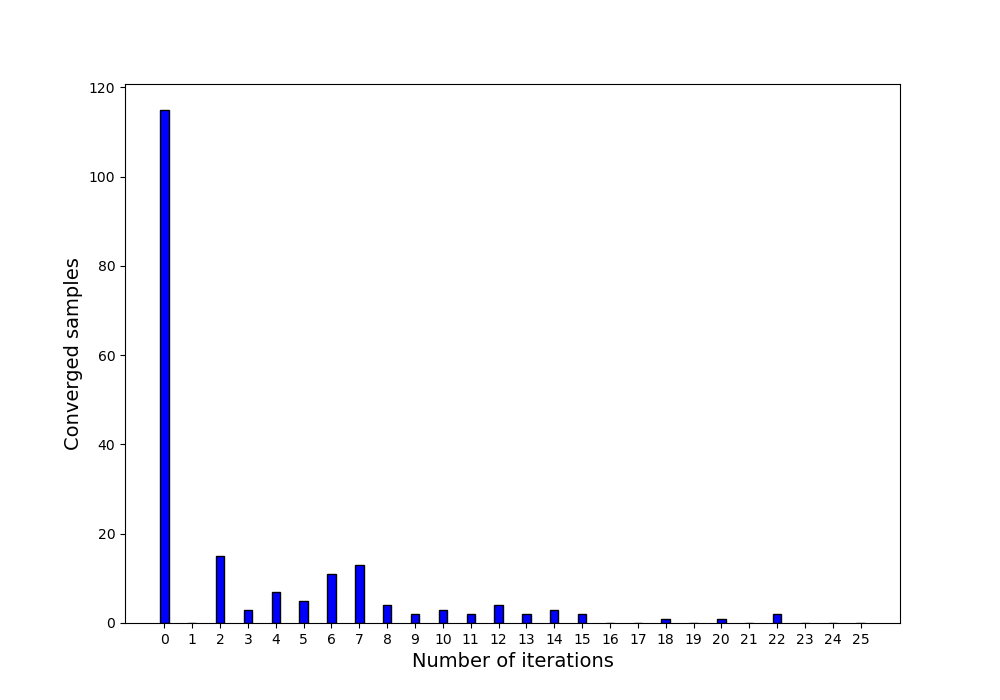}
     \caption{\rev{Converged samples at each iteration using $r=48$ modes.}}
     \end{subfigure}
\caption{\label{fig:converged_samples_48modes} \rev{AirfRANS dataset: Number of converged samples}}.
\end{figure}

When using more modes, the cost per iteration is higher but the convergence is achieved in fewer iterations, so that the overall cost to convergence is actually significantly lower. In Figure \ref{time_naca_online}, we report the time needed to compute all the morphings in the test set for the different values of $r$. 
As we can see, increasing the number of modes allows for faster convergence.

\begin{figure}[ht] 
     \centering
     \begin{subfigure}[b]{0.49\textwidth}
         \centering
        \includegraphics[scale=0.38]{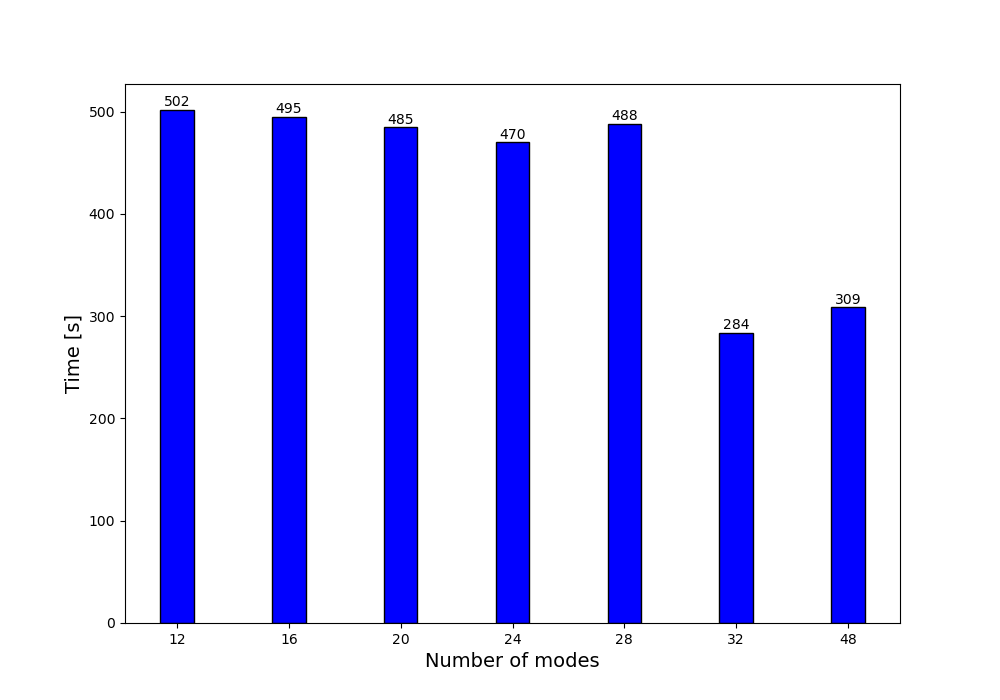}
     \caption{Time to converge with $M_{\rm max}=25$ iterations. }
     \end{subfigure}
     \hfill
     \begin{subfigure}[b]{0.49\textwidth}
         \centering
\includegraphics[scale=0.38]{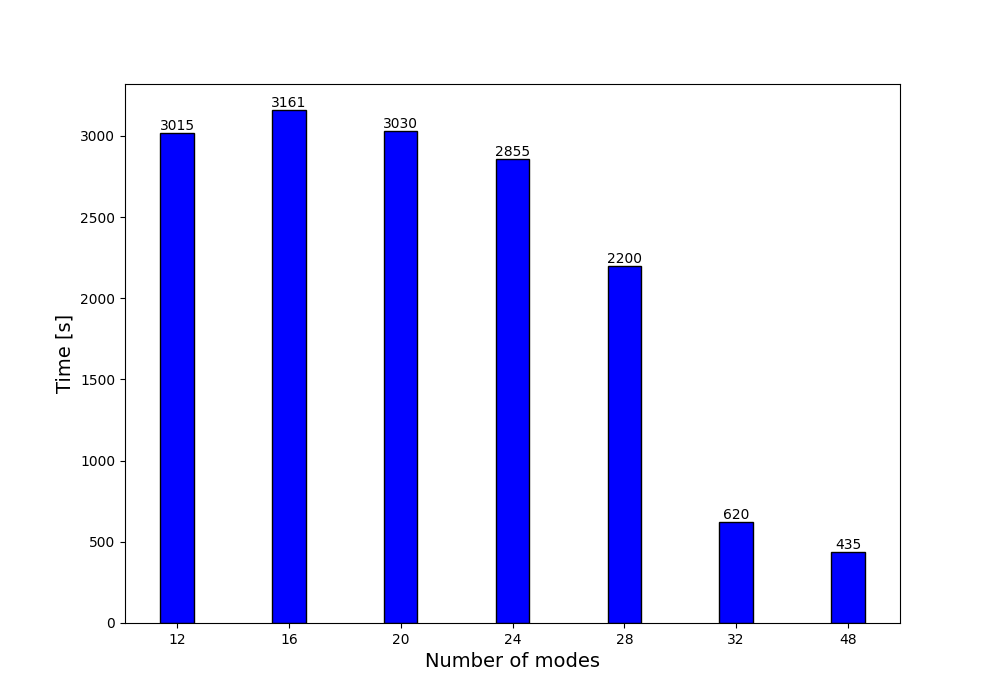}
\caption{ Time to converge with $M_{\rm max}=200$ iterations.}
     \end{subfigure}
\caption{\label{time_naca_online} Overall time needed to compute all the morphings for different values of $r$ and for the maximum number of iterations $M_{\rm max}$. } 
\end{figure}

\FloatBarrier

 \section{Learning scalar outputs from simulations}\label{learning scalars}

In this section, we show numerical results to illustrate how the above morphing strategy can be exploited to build regression models to predict scalar outputs from physical simulations under non-parameterized geometrical variability. This approach is physics-agnostic, that is, the physical equations do not play a role in the process.

\subsection{Methodology}
Let $\{\Omega_i\}_{1\leq i \leq n}$ to be a collection of different geometries. Each geometry is equipped with a (non-geometrical) parameter $\mu_i \in \mathcal P$ where $\mathcal P \subset \mathbb{R}^p$ is a set of parameter values that is used to perform the physical simulations. The parameters can be boundary conditions, material properties and so on; however, we emphasize here that the parametrization of the geometries is not known. In this context, the objective is to determine the outputs of interest of the physical problem, which consist of:
\begin{enumerate}
    \item The physical fields $U_i:=(u_{i,j})_{1\leq j \leq n_{\rm fields}}$ with $u_{i,j}: \Omega_i \to \R$ for all $1\leq i \leq n$ and $1\leq j \leq n_{\rm fields}$ with $n_{\rm fields}\in \mathbb{N}^*$. These fields are usually solutions to a set of partial differential equations. For example, depending on the problem, these can be stress, deformation, velocity, pressure, etc...
    \item The scalar outputs $W_i:=(w_{i,j})_{1\leq j \leq n_{\rm scalars}}$ for all $1\leq i \leq n$ and $1\leq j \leq n_{\rm scalars}$ with $n_{\rm scalars}\in \mathbb{N}^*$. Examples of scalar quantities of interest are the drag and lift coefficients.
\end{enumerate}
Here, we restrict ourselves to the prediction of scalars outputs.
Given the set of input pairs $(\Omega_i,\mu_i)_{1\leq i \leq n}$, and outputs $(W_i)_{1\leq i \leq n}$, calculated using a high-fidelity model, our goal is to learn a mapping $\mathcal{W}$ which maps a pair $(\Omega, \mu)$, where $\Omega$ is a subdomain of $\mathbb{R}^d$ and $\mu\in \mathcal P$ is a parameter value, to the corresponding output $W \in \mathbb{R}^{n_{scalars}}$ so that $W = \mathcal W(\Omega, \mu)$.

Because the geometries are not parameterized, the only available information that represents each geometry is its mesh $\mathcal{M}_i$. However, the learning task on meshes can be quite challenging owing to the high number of degrees of freedom that should be taken as input. To deal with large meshes, solutions using deep neural network architectures are the most popular of machine learning techniques \cite{brunton2020machine,willard2022integrating}. Furthermore, recent advances rely on graph neural networks \cite{scarselli2008graph} as they can overcome the limitation of having graph input with different numbers of nodes \cite{pfaff2020learning}. In \cite{perez2024gaussian}, the authors propose a method that does not rely on neural network architecture, and uses Gaussian process regression model based on the sliced Wasserstein--Weisfeiler--Lehman kernel between graphs to deal with variable geometry to predict scalar outputs. 

Instead, we propose here to consider the offline/online morphing technique described above. We proceed as follows:
\begin{enumerate}
    \item In the offline phase, given the input pairs $(\Omega_i,\mu_i)_{1\leq i \leq n}$ and the outputs $(W_i)_{1\leq i \leq n}$: \begin{enumerate}
        \item Choose a reference domain $\Omega_0$ and calculate the morphings $\boldsymbol{\phi}_i:  \Omega_0 \to \Omega_i$ (Section \ref{Section Offline}).
        \item Apply the snapshot-POD on $(\boldsymbol{\phi}_i)_{1\leq i \leq n}$, and calculate the generalized coordinates $\alpha^i \in \R^r$ for each geometry (Section \ref{Section Online}).
        \item Train the regression model $\mathcal{W}$:
        \begin{align}
            \R^r \times  \mathcal P \ni (\alpha , \mu) \mapsto \mathcal{W}(\alpha,\mu) \in \R^{n_{\rm scalars}}.
        \end{align}
    \end{enumerate}
    Notice that, each geometry is parameterized by the coordinates of the POD modes of the displacement field $\boldsymbol{\phi}_i - \bId$.
    \item In the online phase, given a new pair $(\widetilde{\Omega},\widetilde{\mu})$: \begin{enumerate}
        \item Calculate the vector $\widetilde{\alpha}\in \mathbb{R}^r$ that corresponds to the morphing $\boldsymbol{\varphi}_r(\Tilde{\alpha}) \in \boldsymbol{\mathcal T}_{\Omega_0}$ corresponding to the target domain $\widetilde{\Omega}$ (Section \ref{Section Online}).
        \item Use the regression model to obtain the scalar outputs $\widetilde{W}=\mathcal{W}(\widetilde{\alpha},\widetilde{\mu})$.
    \end{enumerate}
\end{enumerate}
We use a Gaussian process regression for the learning task. For the training phase, we employ anisotropic Matern-5/2 kernels and zero mean-functions for the priors, and the training is done using the GPy package \cite{gpy2014}.

The proposed strategy is similar to the MMGP method from \cite{casenave2024mmgp}. The two main differences are: (i) the morphing algorithm used here is more versatile and is not tailored to specific cases; (ii) the increased efficiency of the present method owing to the offline-online separation of the morphing algorithm. Moreover, the present method computes morphings (and thus displacement fields) from the reference domain, eliminating the need for some finite element interpolation of the displacement fields to a common support in order to apply the snapshot-POD as in MMGP (where morphings are computed towards the reference domain).

\subsection{AirfRANS: drag coefficient prediction}
We apply the above methodology to the AirfRANS dataset. In addition to a mesh of the NACA profile, each sample in the AirfRANS dataset has two scalars as input: the inlet velocity $v_0$ and the angle of attack $\theta_0$. The outputs of the physical simulation are the velocity, pressure and dynamic viscosity fields, as well as the drag and lift coefficients. The outputs are obtained using a 2D incompressible RANS model.

We focus our attention here on learning the drag coefficient $C_d$ from the inputs $\mu=(v_0,\theta_0)$ and the geometry $\Omega$. The Gaussian process regression model $\mathcal{W}$ takes as input the morphing generalized coordinates $\alpha^i$ and the physical parameters $\mu_i=(v_0^i,\theta_0^i)$, and gives as output the drag $C^i_d$. To see the effect of the number of modes and the stopping criterion on the precision of the prediction, we perform the following tests:
\begin{enumerate}
    \item Test 1: we compute the morphings online using $r$ modes for $r\in \{12,16,20,24,28,32,36,40,44,48\}$ (including the initial solution prediction). We use the calculated coordinates to predict $C_d$. We use the same stopping geometrical criterion as above, $\delta^{\mathrm{geo}}:=1.5\times10^{-4}$, for the different values of $r$. We stop the optimization after $M_{max}=200$ iterations if the algorithm does not converge.
    \item Test 2: we compute the morphings using $r'=48$ modes, but we take only the first $r$ coordinates to perform the prediction, with $r \in \{12,16,20,24,28,32,36,40,44,48\}$. In this case, each morphing $\boldsymbol{\varphi}_{r'}(\alpha^i)$ is calculated once, but the number of used components of $\alpha^i$ changes. We use the same tolerance $\delta^{\mathrm{geo}}$ and maximum number of iterations $M_{max}$ as test 1.
    \item Test 3: similar to Test 2, but with $r'=64$ modes.
    \item Test 4: as in Test 2, we take $r'=48$ modes. But we change the stopping criterion: we perform a fixed number of iterations for all the samples regardless of the geometrical error. We choose here to perform 25 iterations for all samples. 
\end{enumerate}
\begin{figure}[ht]
\begin{center}
\includegraphics[scale=0.55]{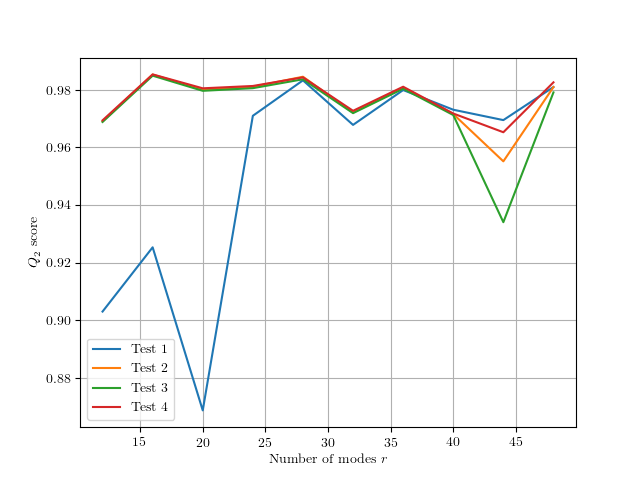}
\caption{\label{Q_2} $Q^2$ scores (see \eqref{Q2 definition}) for different values of $r$. }
\end{center}
\end{figure}

We notice that for different values of $r$, the regression model $\mathcal{W}_r$ changes (we use the subscript $r$ to indicate this). However, the model does not change for a given value of $r$ over the different tests. All the models $\mathcal{W}_r$ are trained once in the offline phase and used for the different tests. 

To evaluate the performance of the method to predict the drag coefficient $C_d$, we evaluate, for each test and for each value of $r$, the $Q^2$-score defined as:
\begin{align}\label{Q2 definition}
Q^2 := 1 - \frac{\displaystyle \sum_{i=1}^{n_{\rm test}}(y_i - f_i )^2}{\displaystyle \sum_{i=1}^{n_{\rm test}}(y_i -\Bar{y} )^2},   
\end{align}
with $y_i$ the true values of the drag $C_d$, $f_i$ the predicted values using the model $\mathcal{W}$, and $\Bar{y}:=  \displaystyle \frac{1}{n_{\rm test}} \sum_{i=1}^{n_{\rm test}} y_i $ the mean. In the best-case scenario, the score is $Q^2=1$, which means that the model predicts correctly all the true values.
 
In Figure~\ref{Q_2}, we present the various $Q^2$ scores. The main observation is that using more modes ($r'$ modes) to calculate the morphing can be beneficial when using models that take $r$ ($r < r'$) mode coefficients as inputs. For instance,  calculating the morphing with $48$ modes but utilizing only the first 16 components of $\alpha$ to predict the values of $C_d$ with $\mathcal{W}_{16}$ yields a superior $Q^2$ score compared to calculating the morphing using only 16 modes and using the obtained coordinates $\mathcal{W}_{16}$ to predict $C_d$. Thus, employing more modes to calculate the morphings enhances the quality of the coefficients for the prediction.
 
From the conducted tests, our best $Q^2$ score is obtained for Test 4 when using 16 modes for the prediction, with $Q^2=0.9853$. A similar result is obtained for Test 2 using also 16 modes for the prediction, with $Q^2=0.9852$. In comparison with the results shown in \cite{casenave2024mmgp}, both results surpassed the scores obtained using, for the same dataset, MMGP ($Q^2=0.9831$), a graph convolutional neural network GCNN \cite{scarselli2008graph} ($Q^2=0.9596$), and MeshGraphNets (MGN) \cite{pfaff2020learning} ($Q^2=0.9743$). 

\section{Conclusion}\label{Conclusion}

We presented a new method to construct morphings between geometries that share the same topology. The technique is suitable to model-order reduction with non-parameterized geometries, as it does not suppose any knowledge of a parameterization of the geometries. In the offline phase, morphings are constructed using elastic deformations from a reference domain onto a target domain. \rev{The approach shares similarities with the method proposed in \cite{de2016optimization}, but also adds the ability of matching of points and lines at the boundary of the target geometry}. In the online phase, morphings are \rev{computed directly in the POD basis} as the solution to a low-dimensional fixed-point iteration problem\rev{, and the algorithm can detect geometries that are out of distribution.}
 
We provided numerical examples \rev{in 2D} to show the performance of the proposed method. \rev{First, in the offline phase, the vector distance algorithm was shown to be more efficient than the signed distance algorithm, both in terms of computation time and convergence. Second, the results of Section \ref{Section Online} showed that, with the proposed initialization step, the online morphing computation is very fast, reaching time ratios between the online and offline phases of the order of 300. This is crucial in a reduced-order modeling.} Third, we illustrated how the computed morphings can be used to predict scalar quantities in physical problems using Gaussian process regression models. \rev{The results were shown to outperform state-of-the-art methods, achieving the best $Q^2$-score.} 
 
\rev{Among several possibilities, we outline two main directions for future work. First, the extension of the method to 3D. While the principles of the algorithm remain unchanged, it deserves to be extensively tested numerically in 3D. Second, the devising of optimal morphing strategies with the objective of minimizing the number of modes to represent new geometries. Even more interestingly, one can aim at minimizing the number of modes with the combined goal of representing new geometries and physical fields on these geometries.}

\section*{Acknowledgements}

Funded/Co-funded by the European Union (ERC, HighLEAP, 101077204). Views and opinions expressed are however those of the author(s) only and do not necessarily reflect those of the European Union or the European Research Council. Neither the European Union nor the granting authority can be held responsible for them.

\newpage
\bibliographystyle{plain}
\bibliography{bib.bib}

\end{document}